\documentclass[a4paper,12pt,reqno]{amsart}
\usepackage[utf8]{inputenc}
\usepackage{amsmath, amsthm, amssymb,todonotes}
\newtheorem{theorem}{Theorem}[section]
\newtheorem*{theorem*}{Theorem}
\newtheorem{lemma}[theorem]{Lemma}
\newtheorem{proposition}[theorem]{Proposition}

\theoremstyle{definition}
\newtheorem{definition}[theorem]{Definition}

\theoremstyle{remark}
\newtheorem{remark}[theorem]{Remark}
\numberwithin{equation}{section}
\usetikzlibrary{angles,quotes,patterns,arrows.meta,bending,decorations.markings,intersections}
\usetikzlibrary{calc}

\DeclareMathOperator{\im}{Im}
\DeclareMathOperator{\diag}{\rm diag}

\usepackage{booktabs} 
\usepackage{array}
\usetikzlibrary{arrows.meta, decorations.markings}

\tikzset{arc arrow/.style args={%
    to pos #1 with length #2}{
    decoration={
        markings,
         mark=at position 0 with {\pgfextra{%
         \pgfmathsetmacro{\tmpArrowTime}{#2/(\pgfdecoratedpathlength)}
         \xdef\tmpArrowTime{\tmpArrowTime}}},
        mark=at position {#1-\tmpArrowTime} with {\coordinate(@1);},
        mark=at position {#1-2*\tmpArrowTime/3} with {\coordinate(@2);},
        mark=at position {#1-\tmpArrowTime/3} with {\coordinate(@3);},
        mark=at position {#1} with {\coordinate(@4);
        \draw[-{Stealth[length=#2,bend]}]
        (@1) .. controls (@2) and (@3) .. (@4);},
        },
     postaction=decorate,
     }
}

\title[Jacobi-Piñeiro orthogonal polynomials]{Strong asymptotics for Jacobi-Piñeiro orthogonal polynomials}

\author{Sergei Kalmykov$^{1,2}$}
\address{$^{1}$School of Mathematical Sciences,   CMA-Shanghai, Shanghai Jiao Tong University, 800 Dongchuan RD, Shanghai 200240, China}
\address{$^{2}$Khabarovsk Division of the Institute for Applied Mathematics, Far Eastern Branch, Russian Academy of Sciences, Russia, Khabarovsk, 60  Seryshev st., 680038, Russia} 
\email{kalmykovsergei@sjtu.edu.cn} 

\author{Vladimir Lysov$^{3}$}
\address{$^{3}$Keldysh Institute of Applied Mathematics
of Russian Academy of Sciences
Miusskaya pl., 4, 125047, Moscow, Russia} 
\email{vlysov@mail.ru} 

\author{Vinay Shukla$^{4,5,\#}$}{\thanks{$^{\#}$Corresponding author} }
\address{$^{4}$School of Mathematical Sciences,  Shanghai Jiao Tong University, 800 Dongchuan RD, Shanghai 200240, China} 
\address{$^{5}$Department of Mathematics, School of Computer Science Engineering and Technology, Bennett University, Greater Noida, India} 
\email{vinayshukla4321@gmail.com, vnshukla01@sjtu.edu.cn} 

\allowdisplaybreaks

\begin{document}
	
	
\begin{abstract}
We investigate the asymptotic behavior of Jacobi-Piñeiro polynomials of degree $2n$ orthogonal on $[0,1]$ with respect to weights $w_j(x) = x^{\alpha_j}(1-x)^{\beta}$, $j=1,2$ where $\alpha_1,\alpha_2, \beta>-1$, and $\alpha_1-\alpha_2 \notin \mathbb{Z}$. These polynomials are characterized by a Riemann-Hilbert problem for a $3 \times 3$ matrix-valued function. We use the Deift-Zhou steepest descent method for Riemann-Hilbert problems to obtain strong uniform asymptotics in the complex plane. The local parametrix around the origin is constructed using Meijer G-functions. We match the local parametrix around the origin with the global parametrix with a double matching, a technique that was recently introduced.
\end{abstract}
\maketitle

\section{Introduction}
Let \(\mu_1, \mu_2, \dots, \mu_r\) be a system of \(r\) positive measures on the real line. 
A sequence of polynomials \(\{P_{\vec{n}}\}\) indexed by a multi-index 
\(\vec{n} = (n_1, \dots, n_r) \in \mathbb{N}^r\) is called a system of 
\emph{multiple orthogonal polynomials} if
\[
\int x^k P_{\vec{n}}(x) \, d\mu_j(x) = 0, \qquad k = 0, 1, \dots, n_j-1, \quad j = 1, \dots, r.
\]

The polynomial \(P_{\vec{n}}\) has degree at most \(|\vec{n}| = n_1 + \cdots + n_r\) and is 
uniquely determined under suitable conditions on the measures. This definition gives rise to 
two types of multiple orthogonal polynomials:

\begin{itemize}
    \item \textbf{Type II}: The monic polynomial \(P_{\vec{n}}\) of degree \(|\vec{n}|\) 
    satisfying the above orthogonality conditions.
    \item \textbf{Type I}: A vector of polynomials \((A_{\vec{n},1}, \dots, A_{\vec{n},r})\) 
    with \(\deg A_{\vec{n},j} \le n_j-1\) such that
    \[
    \sum_{j=1}^r \int x^k A_{\vec{n},j}(x) \, d\mu_j(x) = 
    \begin{cases}
        0, & k = 0, 1, \dots, |\vec{n}|-2, \\
        1, & k = |\vec{n}|-1.
    \end{cases}
    \]
\end{itemize}

Multiple orthogonal polynomials (MOPs) have found diverse and profound applications across mathematics and physics over the past 25 years. MOPs appear in random matrix theory as MOP ensembles, giving determinantal point processes whose correlation kernels can be expressed via an $(r+1) \times (r+1)$ Riemann–Hilbert problem. Kuijlaars’ survey \cite{Kuijlaars_ICM_Survey_2010} explicitly places MOPs in random matrices, non-intersecting Brownian motions, and the two-matrix model, while Van Assche’s tutorial \cite{Assche_book_2020} emphasizes their role in random matrices and Painlevé-type structures. 

Daems and Kuijlaars \cite{Daems_Kuijlaars_2007} showed that the correlation kernel for Brownian paths with multiple starting and ending points coincides with a Christoffel–Darboux kernel for mixed-type MOPs with Gaussian weights. In other words, MOPs encode the spatial statistics of constrained diffusion models. A closely related example is non-intersecting squared Bessel processes. Kuijlaars, Martínez-Finkelshtein, and Wielonsky \cite{Kuijlaars_Finkelshtein_Wielonsky_2009} showed that the positions of these paths form a MOP ensemble for modified Bessel weights, again leading to a matrix Riemann–Hilbert characterization and asymptotic kernels such as sine, Airy, and Bessel kernels. Kuijlaars and Zhang \cite{Kuijlaars_Zhang_2014} showed that the squared singular values of products of Ginibre matrices form a determinantal point process that can be interpreted as a multiple orthogonal polynomial ensemble. Products of random matrices model phenomena in wireless telecommunication and econophysics; that application context is stated explicitly in Akemann, Ipsen, and Kieburg’s work \cite{akemann2013products}. Some discrete families of MOPs have been given an explicit quantum-mechanical realization through non-Hermitian oscillator Hamiltonians. Miki, Vinet, and Zhedanov \cite{Miki_Vinet_Zhedanov_2011} constructed a set of simultaneously diagonalizable non-Hermitian oscillators whose common eigenstates are expressed in terms of multiple Charlier polynomials. This work is extended by Ndayiragije and Van Assche \cite{Ndayiragije_Assche_2013} to multiple Meixner polynomials, again producing $r$-dimensional non-Hermitian oscillator Hamiltonians for which the common eigenstates are expressed through MOPs. MOPs are used in simultaneous Gaussian quadrature, where one wants to approximate several integrals with the same function evaluations \cite{Laudadio_Mastronardi_Dooren_2025}. Alqahtani and Reichel \cite{Alqahtani_Reichel_2018} used MOP-based quadrature ideas for matrix function evaluation, and current papers \cite{Laudadio_Mastronardi_Dooren_2024, Laudadio_Mastronardi_Dooren_2025} continue to develop efficient algorithms for simultaneous Gaussian quadrature based on MOP recurrence structures. 

Jacobi–Piñeiro (JP) polynomials—the canonical example of type II multiple orthogonal polynomials (MOPs) on $[0, 1]$ with respect to an $r$-tuple of Jacobi-type weights occupy a distinguished position at the confluence of approximation theory, analytic number theory, and mathematical physics. Their utility is perhaps most classical in the context of Hermite–Padé approximation, where the denominators of simultaneous rational approximants to Stieltjes transforms of the weights are precisely the JP polynomials.  

Adler, Moerbeke, and Wang \cite{Adler_Mark_Wang_2013} studied matrix minor processes in which the JP model appears alongside the Gaussian Unitary Ensemble-with-external-source and multiple Laguerre models. They showed that the interlacing eigenvalues form a determinantal point process, and they connected the model to last-passage percolation. They further showed that an appropriate scaling leads to the Pearcey process, linking the model to random directed polymers.
A second important application is in quantum integrable systems, specifically the Gaudin model. Mukhin and Varchenko \cite{Mukhin_Varchenko_2007} showed that, for a certain Bethe-ansatz problem, the first polynomial in the relevant polynomial sequence is exactly the Jacobi–Piñeiro MOP. Lu, Mukhin, and Varchenko \cite{Lu_Mukhin_Varchenko_2016} wrote that in the Gaudin model for classical Lie algebras, “the solutions of the Bethe ansatz equations are related to zeros of JP polynomials.” 
A newer and quite concrete application is to stochastic processes. Branquinho and coauthors \cite{branquinho2021multiple} developed a framework in which nonnegative higher-order recurrence matrices for MOPs generate dual stochastic matrices. In this work, the JP family is used as a case study, and the resulting objects are interpreted as 1D random walks / Markov chains. The research \cite{Branquinho_Juan_Moreno_Manuel_2024} states that the JP case yields two dual Markov chains, with explicit stochastic matrices, Karlin–McGregor-type formulas, and a classification of recurrent/transient regimes. A recent extension goes further to finite Markov chains \cite{Branquinho_Juan_Moreno_Manuel_2025}, and explicitly lists Jacobi–Piñeiro among the multiple orthogonal families for which recurrent finite Markov chains and time-reversed chains are derived. 
There is also a direct application to urn models, which are classical simplified models in statistical physics and probability. Grünbaum and de la Iglesia \cite{Alberto_Iglesia_2022} emphasize that the set of physically motivated urn models solvable through orthogonal polynomials is very small, and they then construct an urn model for the JP polynomials. 

Jacobi–Piñeiro polynomials go back to Piñeiro Díaz’s 1987 work \cite{pineiro1987simultaneous} on simultaneous approximation. Soon after, Aptekarev, Branquinho, Van Assche and others \cite{Assche_Coussement_2001, Aptekarev_Branquinho_Assche_2003, Branquinho_Juan_Moreno_Manuel_SAM_2025} gave a systematic treatment of the JP case. These works are important because they place the family inside a unified framework based on Pearson equations, Rodrigues operators, and explicit structural formulas. They also showed that the classical families in this framework satisfy a linear differential equation of order $p+1$ and derived explicit formulas and recurrence relations.
Beckermann, Coussement, and Van Assche \cite{Beckermann_Coussement_Assche_2005} studied JP polynomials with complex parameters, gave explicit formulas, including expressions in terms of Kampé de Fériet series, and placed them in a network of limiting relations linked to multiple Wilson polynomials.
Neuschel and Van Assche \cite{Neuschel_Assche_2016} obtained the asymptotic zero distribution of JP polynomials, using nearest-neighbor recurrences and showed that the limiting distribution is related to the Fuss–Catalan distribution. Martínez-Finkelshtein, Morales, and Perales \cite{Finkelshtein_Morales_Perales_2026} using finite free convolution studied zeros of generalized hypergeometric polynomials and then applied the method to multiple orthogonal families, including Jacobi–Piñeiro, obtaining results on interlacing, monotonicity, and asymptotics of zeros. Van Assche’s paper \cite{Assche_2017} on Mehler–Heine asymptotics studied local behavior near endpoints of the support for several MOP families, including Jacobi–Piñeiro. Smet’s paper \cite{Smet_Assche_2010} worked out the Mellin transform of multiple JP polynomials. Lysov’s paper \cite{Lysov_2018, Lysov_2019} focused on JP polynomials and the functions of the second kind, emphasizing asymptotics, integral representation, and strong asymptotics. Branquinho, Díaz, Foulquié-Moreno, and Mañas \cite{Branquinho_Juan_Moreno_Manuel_PAMS_2024} gave explicit multiple-hypergeometric formulas for type I JP polynomials with arbitrary number of weights. A related paper \cite{Branquinho_Juan_Moreno_Manuel_Wolfs_2025} derived contour integral representations for both type I and type II JP polynomials, with evaluation leading to new explicit hypergeometric representations.

The asymptotic analysis of orthogonal polynomials reveals profound insights into their limiting behavior as the degree tends to infinity. Wong’s surveys \cite{Wong_Zhao_2009, Wong_2018} are good entry points for classical asymptotic methods for families such as Stieltjes–Wigert, Hahn, Racah, and pseudo-Jacobi. Darboux’s method \cite{Ismail_book_2005} extracts the
asymptotics of the n-th coefficient of a power series from the singularity structure of
the generating function. Another important stream derives asymptotics from the difference equation or recurrence itself, often via Liouville–Green/WKB and turning-point analysis. Huang, Cao, and Wang’s paper \cite{Huang_Lihua_Xiang_2019} develops a unified asymptotic-expansion method for orthogonal polynomials from second-order difference equations and explicitly points out that many Askey-scheme families—such as Hermite, Laguerre, Krawtchouk, Meixner, Hahn, Racah, and Wilson—fit this framework. Chen and Ismail \cite{Chen_Ismail_2005} applied a ladder operator
technique to generate differential equations for semi-classical Jacobi-type polynomials
and extracted leading-order WKB asymptotics, asymptotics of recurrence coefficients, Hankel determinants, etc. The saddle point analysis of  Cauchy integral representations yields asymptotics for several OP families \cite{Lysov_2018}. The theory of logarithmic potential theory provides the natural framework for weak asymptotics via the vector equilibrium problem.


Weak asymptotics usually means the asymptotic zero distribution, or equivalently $n^{th}$ root / logarithmic asymptotics outside the support. Totik’s survey \cite{Totik_Survey_2005} explains that $n^{th}$ root asymptotics is the weakest level and is “mostly equivalent” to the asymptotic distribution of zeros. Ratio asymptotics studies limits like $p_{n+1}(z)/p_{n}(z)$ in the scalar case, or ratios of neighboring MOPs in the MOP setting. Totik’s survey \cite{Totik_Survey_2005} places it strictly above weak asymptotics. Strong asymptotics gives full asymptotic expansions, often of Szegő or Plancherel–Rotach type, together with explicit prefactors and error estimates. This is exactly where the \emph{Deift–Zhou steepest-descent method}, also called the \emph{nonlinear steepest-descent method} \cite{Deift_Zhou_1993} becomes central.

The Riemann–Hilbert (RH) characterization of orthogonal polynomials was established by
Fokas et al. \cite{Fokas_Kitaev_1992}. They found a $2 \times 2$ matrix as the unique solution of a $2 \times 2$ Riemann–Hilbert problem (RHP) with jump matrix $\begin{pmatrix}
      1  & w(t)\\
       0 & 1
    \end{pmatrix}$ 
on $\mathbb{R}$ and normalization $Y(z)e^{-n \sigma_3} \to I$ as $z \to \infty$. In general, to carry out steepest-descent analysis, we begin by formulating a related matrix Riemann-Hilbert problem that uniquely characterizes the orthogonal polynomials. The asymptotic analysis then proceeds through a series of explicit and invertible transformations. The first transformation normalizes the behavior of the solution at infinity. The second transformation, known as "opening of the lenses," leverages the factorization of the jump matrix to convert the highly oscillatory behavior on the support of the orthogonality measure into exponential decay away from this set. With the oscillatory behavior tamed, we construct a global or "outer" parametrix, which captures the leading-order asymptotics away from the endpoints of the support. To obtain a uniform approximation, we then build local parametrices near the endpoints, typically using special functions like Bessel, Airy or Meijer G functions, ensuring that the jump matrices there are also reduced to near-identity forms. By matching these local solutions with the global parametrix and carefully tracing back through all the transformations, we finally obtain an explicit asymptotic expression for the polynomials, valid in the entire complex plane.

The core RH literature for scalar orthogonal polynomials developed first around weights on the real line. Deift–Kriecherbauer–McLaughlin–Venakides–Zhou overview \cite{Deift_Kriecherbauer_McLaughlin_Venakides_Zhou_2001} summarized the RH steepest-descent method for orthogonal polynomials with varying exponential weights. In the same early phase, Kuijlaars, McLaughlin, Van Assche, and Vanlessen \cite{Kuijlaars_McLaughlin_Assche_2004} carried out a full RH analysis for orthogonal polynomials on $[-1,1]$ with modified Jacobi weights, obtaining strong asymptotic expansions. For Laguerre-type families, Vanlessen \cite{Vanlessen_2007} obtained full Plancherel–Rotach asymptotics for orthogonal polynomials on $[0,\infty)$ with weights $x^\alpha e^{-Q(x)}$. Dai’s survey \cite{Dai_2018} reviews RH asymptotics for orthogonal polynomials connected with Painlevé transcendents. RH analysis was also pushed far for orthogonal polynomials on the unit circle (OPUC). 
Finkelshtein, McLaughlin, and Saff \cite{Finkelshtein_McLaughlin_Saff_2006} obtained strong asymptotics for Szegő orthogonal polynomials with analytic weights. 
For nonanalytic weights on the circle, McLaughlin and Miller \cite{McLaughlin_Miller_2006} developed the $\overline{\partial}$-steepest-descent extension of RH analysis and applied it to OPUC. For a recent work on generalized semiclassical OPUC for modified Jacobi and Bessel-type weights on the unit circle, we refer to \cite{branquinho2025generalized}. RH analysis was also adapted to discrete orthogonal polynomials \cite{Baik_Kriecherbauer_McLaughlin_Miller_2003}. Borodin’s work \cite{Borodin_Boyarchenko_2003} is also important in this branch because it connected discrete orthogonal polynomials, RH-type formulations, and discrete Painlevé structures. A separate and now very large branch is the RH analysis of MOPs. In this regard, strong asymptotics for multiple Laguerre polynomials \cite{Lysov_Wielonsky_2008}, strong asymptotics for Cauchy biorthogonal polynomials \cite{Bertola_Gekhtman_Szmigielski_2013}, etc., have been studied in the past two decades. Bertola and Bothner \cite{Bertola_Bothner_2015} analyzed a model of several coupled positive-definite matrices via Cauchy biorthogonal polynomials and a $(p+1)\times(p+1)$ Riemann--Hilbert problem, obtaining strong asymptotics and new universality classes \cite{Bertola_Bothner_2015}. In the Muttalib--Borodin setting, Kuijlaars and Molag established local universality for $\theta=\tfrac{1}{2}$ \cite{Kuijlaars_Molag_2019}, and Molag later extended the analysis to $\theta=1/r$ \cite{Kuijlaars_Molag_2019}, both through higher-order Riemann--Hilbert problems and Meijer-$G$ parametrices \cite{Kuijlaars_Molag_2019,Molag_Nonlinearity_2021}. Silva and Zhang \cite{Silva_Zhang_2020} extended this circle of ideas to the large
$n$ analysis of the product of two coupled random matrices In recent years, the RH approach has been extended beyond the real-line and circle settings to the matrix orthogonal polynomials \cite{Cassatella_Manas_2012, Branquinho_Moreno_Manuel_2024}. There are also newer extensions to $q-$orthogonal polynomials and planar orthogonal polynomials. Joshi and Latimer \cite{Joshi_Latimer_2021} formulated and solved a $q-$Riemann–Hilbert problem for a class of $q-$orthogonal polynomials. Hedenmalm and Wennman \cite{Hedenmalm_2024, Hedenmalm_Wennman_2024} developed RH methods for planar orthogonal polynomials. 

Despite the extensive study of various structural and asymptotic properties of  JP polynomials in the literature, the derivation of their strong  asymptotics has remained an open problem. Motivated by the success of the  Riemann--Hilbert method in analyzing related orthogonal polynomial ensembles, the present work aims to fill this gap by establishing strong asymptotics with respect to a system of two Jacobi weights. The remainder of this paper is organized as follows:

Section 2 formulates the RHP that characterizes the Jacobi–Piñeiro polynomials. Section 3 analyzes the weak asymptotics of the underlying measures, which leads to a vector equilibrium problem with a Nikishin interaction matrix. Section 4 states that the solution of this equilibrium problem can be expressed explicitly in terms of an algebraic function, providing the necessary ingredients for the subsequent steepest-descent analysis. Section 5 describes the auxiliary $H$-functions for later use.
Section 6 constitutes the core of the analysis and is divided into several subsections. In Section 6.1, we perform the first transformation which introduces a jump on the negative real axis. Section 6.2 introduces the second transformation, which normalizes the problem at infinity. Section 6.3 performs the second transformation, opening the lenses to convert oscillatory behavior into exponential decay. Section 6.4 constructs the outer (global) parametrix that captures the leading-order behavior away from the endpoints. Section 6.5 builds the local parametrix at the endpoint 
$0$ using Meijer-G function. To match the local parametrix with the global one, we employ a double-matching technique introduced in \cite{Kuijlaars_Molag_2019}. Moreover, we go a step further by re-establishing and explicitly verifying this matching through the determination of the appropriate constant values required for Theorem 1.2 in \cite{Molag_2021} to hold. Section 6.6 constructs the local parametrix at the endpoint $1$ with the help of Bessel functions. Section 6.7 constructs the local parametrix at the endpoint $x^*$ with the help of Airy functions. Section 6.8 applies the final / ratio transformation. The strong asymptotics for the Jacobi–Piñeiro polynomials in all the regions are computed in Section 7.


\section{Riemann-Hilbert Problem}
Let $\vec{n} = (n_1, n_2) \in \mathbb{N}^2$ be a multi-index of size $|\vec{n}| = n_1+n_2$, where $n_1+1\ge n_2$.  \emph{The Jacobi-Pi\~neiro polynomials} $P_{\vec{n}}$, with parameters $\alpha_1, \alpha_2$, and $\beta$ are monic MOPs of degree $|\vec{n}|$ on $\Delta = [0,1]$ with weights 
\begin{align*}
    w_j(x) = x^{\alpha_j}(1-x)^{\beta},
\end{align*}
and satisfy orthogonality condition
\begin{equation}\label{Jacobi Pineiro}
\int_0^1 P_{\vec{n}}(x)x^k x^{\alpha_j}(1-x)^{\beta} dx = 0, \ \ k = 0,1, \ldots, n_j-1, 
\end{equation}
for $j=1,2$, where $\alpha_1,\alpha_2, \beta>-1$, and $\alpha_1-\alpha_2 \notin \mathbb{Z}$. 
Associated to the Jacobi-Pi\~neiro polynomials $P_{\vec{n}}$, there are second kind polynomials denoted by $\mathcal{R}_{\vec{n},j}$ and defined as
\begin{align}\label{Rn1 and 2}
\mathcal{R}_{\vec{n},j}(z)
&=\frac{1}{2\pi i}\widehat{P_{\vec{n}}w_j}(z) = \frac{-1}{2\pi i} \int_0^1 \frac{P_{\vec{n}}(x)w_j(x)}{z-x}dx.
\end{align}

The problem is to find a $3 \times 3$ matrix-valued function $Y: \mathbb{C} \backslash \Delta \rightarrow \mathbb{C}^{3 \times 3}$ such that the following three properties are satisfied:
\begin{enumerate}
\item $Y$ is analytic in $\mathbb{C} \backslash \Delta $ where $\Delta = [0,1]$.
\item $Y$ has boundary values $Y_+$ and $Y_-$ such that $Y_+ = Y_- \mathcal{J}_Y$, where
\begin{align}\label{RH Condition 2}
 \mathcal{J}_Y= \begin{pmatrix}
1 & w_1 & w_2 \\
0 & 1 & 0 \\
0 & 0 & 1 
\end{pmatrix} 
\end{align}
on $\Delta$.
\item As $z \rightarrow \infty$, we have
\begin{align}\label{RH Condition 3}
Y(z)=\left(\mathbb{I}+O\left(\frac{1}{z}\right) \right) \begin{pmatrix}
z^{n_1+n_2} & 0 & 0 \\
0 & z^{-n_1} & 0 \\
0 & 0 & z^{-n_2}
\end{pmatrix} ,
\end{align}
where $n_1+n_2=2n$.
\item $Y$ has the following behavior at the endpoints $0$ and $1$:
\begin{align*}
&Y(z)= O\begin{pmatrix}
1 & t(z) & t(z) \\
1 & t(z) & t(z) \\
1 & t(z) & t(z)
\end{pmatrix} ~ z \rightarrow 1,~ {\rm where}~ t(z) = \begin{cases}
  |z-1|^{\beta},  & \beta < 0,\\
   \log|z-1|, & \beta = 0,\\
    1, &\beta > 0.
\end{cases} \\
&Y(z)= O\begin{pmatrix}
1 & u(z) & v(z) \\
1 & u(z) & v(z) \\
1 & u(z) & v(z)
\end{pmatrix} ~ z \rightarrow 0, ~{\rm where}~ u(z) = \begin{cases}
  |z|^{\alpha_1},  & \alpha_1< 0,\\
   \log|z|, & \alpha_1= 0,\\
    1, &\alpha_1 > 0,
\end{cases}  \nonumber\\
& {\rm and}~ v(z) = \begin{cases}
  |z|^{\alpha_2},  & \alpha_2< 0,\\
   \log|z|, & \alpha_2= 0,\\
    1, &\alpha_2 > 0.
\end{cases}  
\end{align*}

\end{enumerate} 
\begin{theorem}
Let $P_{n_1,n_2}(z)$, $\mathcal{R}_{\vec{n},1}(z)$ and $\mathcal{R}_{\vec{n},2}(z)$ be given by \eqref{Jacobi Pineiro} and \eqref{Rn1 and 2}, respectively. Then the above Riemann-Hilbert problem for $Y$ has a unique solution given by
\begin{align*}
Y(z) = \begin{pmatrix}
P_{n_1,n_2}(z) & \mathcal{R}_{\vec{n},1}(z) & \mathcal{R}_{\vec{n},2}(z) \\
\widehat{c_1} P_{n_1-1,n_2}(z) &\widehat{c_1} \mathcal{R}_{\vec{n}-e_1,1}(z) &  \widehat{c_1} \mathcal{R}_{\vec{n}-e_1,2}(z) \\
\widehat{c_2} P_{n_1,n_2-1}(z) &\widehat{c_2}  \mathcal{R}_{\vec{n}-e_2,1}(z) &\widehat{c_2}  \mathcal{R}_{\vec{n}-e_2,2}(z)
\end{pmatrix},
\end{align*}
\end{theorem}
where $e_1=(1,0)$ and $e_2=(0,1)$.
\begin{proof}
From the conditions \eqref{RH Condition 2} and \eqref{RH Condition 3} for the entry $(1,1)$, we have
\begin{align*}
(Y_{11})_+=(Y_{11})_- \quad {\rm and} \quad Y_{11}(z)=z^{n_1+n_2}+O(z^{n_1+n_2-1}).
\end{align*}
These conditions are satisfied for $P_{n_1,n_2}(z)$ since it is a monic polynomial of degree $n_1+n_2$.

For the entry $(1,2)$, we get
\begin{align*}
(Y_{12})_+=(Y_{12})_- +(Y_{11})_-w_1\quad {\rm and} \quad Y_{12}(z)=O(z^{-n_1-1}).
\end{align*}
But, as shown in the previous step, $Y_{11}(z)=P_{n_1,n_2}(z)$. Using this fact along with Sokhotskii-Plemelj formula and the formula \eqref{Rn1 and 2} for $\mathcal{R}_{\vec{n},1}(z)$ suggests $Y_{12}(z)=\mathcal{R}_{\vec{n},1}(z)$.

For the entry $(1,3)$, we obtain
\begin{align*}
(Y_{13})_+=(Y_{13})_- +(Y_{11})_- w_2\quad {\rm and} \quad Y_{13}(z)=O(z^{-n_2-1}).
\end{align*}
Again using the fact from the previous step that $Y_{11}(z)=P_{n_1,n_2}(z)$, Sokhotskii-Plemelj formula and the formula \eqref{Rn1 and 2} for $\mathcal{R}_{\vec{n},2}(z)$ yields $Y_{13}(z)=\mathcal{R}_{\vec{n},2}(z)$.

To verify the asymptotic conditions $Y_{1k}(z)=O(z^{-n_{k-1}-1})$ for $k=2,3$, first we use
\begin{align}\label{1 over x-z}
    \frac{1}{x-z} = -\sum_{i=0}^{n_{k-1}-1}\frac{x^i}{z^{i+1}}+\frac{x^{n_{k-1}}}{z^{n_{k-1}}} \frac{1}{x-z} 
\end{align}
for $k=2$ in \eqref{Rn1 and 2} to observe
\begin{align*}
 &\mathcal{R}_{\vec{n},1}(z)  =\frac{1}{2\pi i} \int_0^1 \frac{P_{\vec{n}}w_1(x)}{x-z}dx \\
 & = -\sum_{i=0}^{n_{1}-1} \left(\frac{1}{2\pi i} \int_0^1 P_{\vec{n}} x^i w_1(x)\right) z^{-i-1}+\left( \frac{1}{2\pi i}  \int_0^1 \frac{P_{\vec{n}} x^{n_{1}} w_1(x)}{x-z}dx\right) z^{-n_{1}}.
\end{align*}
By the virtue of \eqref{Jacobi Pineiro}, the above expression for $\mathcal{R}_{\vec{n},1}(z)$ reduces to
\begin{align*}
\mathcal{R}_{\vec{n},1}(z)  = \left( \frac{1}{2\pi i}  \int_0^1 \frac{P_{\vec{n}} x^{n_{1}} w_1(x)}{x-z}dx\right) z^{-n_{1}}
\end{align*}
which shows that the asymptotic condition $Y_{12}(z)=O(z^{-n_{1}-1})$ is satisfied. If we start from $\mathcal{R}_{\vec{n},2}(z)$ and \eqref{1 over x-z} for $k=3$, then a similar analysis verifies the asymptotic condition for $Y_{13}(z)$.
For the entry $(2,1)$, we arrive at
\begin{align*}
(Y_{21})_+=(Y_{21})_- \quad {\rm and} \quad Y_{21}(z)=O(z^{n_1+n_2-1})
\end{align*}
which is satisfied for a polynomial of degree $n_1+n_2-1$, i.e. $Y_{21}(z)=\widehat{c_1}P_{n_1-1,n_2}(z)$. Note that the asymptotic condition in this step is different from the one used for obtaining the entry $Y_{11}(z)$. Hence, the polynomials obtained in this step need not be monic. It necessitates the multiplication by a suitable factor $\widehat{c_1}$. 

For the entry, $(2,2)$, we get
\begin{align*}
(Y_{22})_+=(Y_{22})_- +(Y_{21})_-w_1\quad {\rm and} \quad Y_{22}(z)=z^{-n_1}+O(z^{-n_1-1}),
\end{align*}
where $Y_{21}(z)=\widehat{c_1}P_{n_1-1,n_2}(z)$. Hence, Sokhotskii-Plemelj formula and the formula \eqref{Rn1 and 2} for $\mathcal{R}_{\vec{n},1}(z)$ imply $Y_{22}(z)=\widehat{c_1}\mathcal{R}_{(n_1-1,n_2),1}(z) =\widehat{c_1}\mathcal{R}_{\vec{n}-e_1,1}(z)$. Further, expanding $\mathcal{R}_{\vec{n}-e_1,1}(z)$ using \eqref{1 over x-z} leads to 
\begin{align*}
&\mathcal{R}_{\vec{n}-e_1,1}(z)=\frac{1}{2\pi i} \int_0^1 \frac{P_{(n_1-1,n_2)}w_1(x)}{x-z}dx= -\sum_{i=0}^{n_{1}-2} \left(\frac{1}{2\pi i} \int_0^1 P_{(n_1-1,n_2)} x^i w_1(x)\right) z^{-i-1} \\
 &-\left(\frac{1}{2\pi i} \int_0^1 P_{(n_1-1,n_2)} x^{n_1-1} w_1(x)\right) z^{-n_1} +\left( \frac{1}{2\pi i}  \int_0^1 \frac{P_{\vec{n}} x^{n_{1}} w_1(x)}{x-z}dx\right) z^{-n_{1}}.
\end{align*}
The orthogonality condition \eqref{Jacobi Pineiro} suggests
\begin{equation*}
\int_0^1 P_{(n_1-1,n_2)}(x)x^i w_1dx = 0, \ \ k = 0,1, \ldots, n_1-2.
\end{equation*}
Further, to make the leading coefficient of $Y_{22}(z)$ monic, we must have
\begin{align*}
\int_0^1 P_{(n_1-1,n_2)}(x)x^{n_1-1} w_1dx = -2 \pi i,
\end{align*}
which suggests the choice $\widehat{c_1} = -1/ 2 \pi i$, and the asymptotic condition for $Y_{22}(z)$ is satisfied. 

A similar line of arguments can be used to find $Y_{23}(z)$ and the third row of the matrix $Y(z)$.
\end{proof}

\section{Vector equilibrium problem of logarithmic potential}\label{Vector equilibrium problem}
For a closed set $K$ in $\mathbb{C}$, the class $\mu_t(K)$ represents the set of finite Borel measures $\mu$ on $K$ having finite energy, i.e.,
\begin{align*}
    \mu_t(K)=\{\mu: \mu \text{ is a positive measure on $K$},~\mu(K)=t,~I(\mu) < \infty \},
\end{align*}
where
\begin{align*}
I(\mu) = \int V^\mu(x) d\mu(x) \quad ~ {\rm and} \quad~ V^\mu(x) = \int \ln \frac{1}{|x-t|}d\mu(t)
\end{align*}
is the logarithmic energy and logarithmic potential, respectively (see e.g. \cite{Saff_Totik_1997}). Let $\mu_2(\Delta)$ and $\mu_\theta(\Gamma)$ denote the subsets in $\mu_t(K)$ of measures of mass $2$ and $\theta$, respectively. For a pair of closed intervals $\Delta$ and $\Gamma$, the problem is to find unique extremal measure $\vec{\lambda}=(\lambda_1,\lambda_2)$ minimizing the energy functional $J$ given as 
\begin{align*}
J(\mu_1,\mu_2)= I(\mu_1)+I(\mu_2)+I(\mu_1-\mu_2).
\end{align*}
Due to the convexity of the problem, there exists a unique extremal measure $\vec{\lambda}=(\lambda_1,\lambda_2)$ such that $J(\vec{\lambda}) < J(\vec{\mu})$, $\forall \mu \in \mu_2(\Delta)\times \mu_\theta(\Gamma)$. By Kuhn–Tucker theorem \cite{boyd2004convex}, the vector measure $\vec{\lambda}$ is completely characterized by the following equilibrium relations 
\begin{align}
&(2V^{\lambda_1}-V^{\lambda_2})(x) \begin{cases} =\omega_1, & x \in S(\lambda_1),\\
\geq \omega_2, & x \in \Delta,
\end{cases} \label{equilibrium problem}\\
&(2V^{\lambda_2}-V^{\lambda_1})(x) \begin{cases} =\omega_1, & x \in S(\lambda_2),\\
\geq \omega_2, & x \in \Gamma,
\end{cases} \label{equilibrium problem 2}
\end{align}
where $\omega_1$ and $\omega_2$ are some equilibrium constants. Note that the problem of finding the equilibrium measure $\vec{\lambda}$ that minimizes the energy functional $J$ is equivalent to the vector problem of logarithmic potentials, $V^\mu$, with Nikishin interaction matrix $A_N$ given by
\begin{align*}
    A_N:= \begin{pmatrix}
       2 & -1\\
        -1& 2
    \end{pmatrix}.
\end{align*}
A solution to the above-mentioned problem will be provided in the next section. For a recent relevant studies in this direction, we refer to \cite{Aptekarev_Lagomasino_Andre_2017,Lapik_2012} and references therein.
The dependence of supports of the components of the equilibrium measure $\vec{\lambda}$ on the value of $\theta$ is as follows \cite{Bogolyubskii_Lysov_2020}:
\begin{align*}
&\theta \in (0,1) \Rightarrow S(\lambda_1)=\Delta, \quad S(\lambda_2)=\Gamma^*=[x^*,0], \quad  x^*<0, \quad  \frac{d}{d\theta} x^*<0;\\
&\theta = 1 \Rightarrow S(\lambda_1)=\Delta, \quad S(\lambda_2)=\Gamma, \quad \omega_2 =0,
\end{align*}
where
\begin{align}\label{eq_x*}
x^*(\theta) = \frac{-27\theta^2(2-\theta)^2}{(2+\theta)^2(4-\theta)^2(1-\theta)^2}.
\end{align}

It can be observed that for all $\theta \in (0,1]$, the support of $\lambda_1$ coincides with $\Delta$ and the support of $\lambda_2$, which we denote by $\Gamma^*$, monotonically increases with $\theta$. Thus, the supports always have a point of tangency at the origin, and for $\theta = 1$, one of them is not compact. This case is the least studied, and strong asymptotics of Hermite-Pade approximants for this case will be obtained in an upcoming work of the same authors.

\section{Weak Asymptotics}
A method for studying asymptotics of Hermite-Pade approximations based on the vector equilibrium problem was developed in \cite{Gonchar_Rakhmanov_1987}. It follows from \cite{Gonchar_Rakhmanov_Sorokin_1997} that the weak asymptotics of Hermite-Pade approximations for the Nikishin system in the case where $\Delta$ and $\Gamma$ are non-overlapping segments of the real line is described by the solution of a certain equilibrium problem. For the problem under consideration, the capacitors $\Delta$ and $\Gamma$ are touching. However, it turns out that the result of \cite{Gonchar_Rakhmanov_Sorokin_1997} is still valid. 
It is convenient to write the weak asymptotics in terms of the equilibrium potential for a given sequence of orthogonal polynomials. The weak asymptotics of Hermite-Pade approximants under consideration is described as follows:
\begin{theorem}\cite[Theorem 1]{Lys17}
If $\dfrac{n_2}{n_1+n_2} \rightarrow \dfrac{\theta}{2} \in (0,1/2]$, then the following asymptotic formulas hold:
\begin{align*}
&\frac{2}{|\vec{n}|}\log |P_{\vec{n}}(x)| \rightarrow -V^{\lambda_1}(x), \quad x \in \mathbb{C} \backslash \Delta, \\
& \frac{2}{|\vec{n}|}\log |\mathcal{R}_{\vec{n},1}(x)| \rightarrow V^{\lambda_1}(x)-V^{\lambda_2}(x)-\omega_1, \quad x \in \mathbb{C} \backslash (\Delta \cup \Gamma^*),\\
& \frac{2}{|\vec{n}|} \log|\tilde{\mathcal{R}}_{\vec{n},2}(x)| \rightarrow V^{\lambda_2}(x)-\omega_1-\omega_2, \quad x \in \mathbb{C} \backslash \Gamma^*.
\end{align*}
\end{theorem}

\subsection{Explicit solution of the vector equilibrium problem in terms of algebraic function}
The solution to the problem discussed in section \ref{Vector equilibrium problem} can be given in terms of algebraic function. In this section, first we construct the algebraic function and then the solution to the equilibrium problem is formulated as proposition \ref{prop_equi_algebraic}. For the corresponding construction, we need a three sheeted Riemann surface $\mathcal{R}$ of genus $0$ and its conformal mapping onto the sphere $\bar{\mathbb{C}}_s$. This Riemann surface is realized by gluing three copies of the complex sphere along the cuts $\Delta$ and $\Gamma^*$ (see Figure \ref{Fig_Riemann surface}):
\begin{align*}
\mathcal{R}_0:= \bar{\mathbb{C}}\backslash \Delta, \quad \mathcal{R}_1:=\bar{\mathbb{C}}\backslash (\Delta \cup \Gamma^*),\quad \mathcal{R}_2:=\bar{\mathbb{C}}\backslash \Gamma^*.
\end{align*} 

\begin{figure}[htbp!]
\begin{center}
\begin{tikzpicture}
\filldraw (-2,-4)node[below=5pt,red]{$x^*$} circle [radius=2pt];
\filldraw (3,-4)node[below=5pt,red]{$0$} circle [radius=2pt];
\filldraw (3,2)node[above=5pt,red]{$0$} circle [radius=2pt];
\filldraw (5,2)node[above=5pt,red]{$1$} circle [radius=2pt];
\coordinate (A) at (-3,3);
\coordinate (B) at (7,3);
\coordinate (C) at (6,1);
\coordinate (D) at (-4,1);
\draw[thick] (A)node[below=15pt]{$\mathcal{R}_0$} -- (B) -- (C) -- (D) -- cycle;
\coordinate (E) at (-3,0);
\coordinate (F) at (7,0);
\coordinate (G) at (6,-2);
\coordinate (H) at (-4,-2);
\draw[thick] (E)node[below=15pt]{$\mathcal{R}_1$} -- (F) -- (G) -- (H) -- cycle;
\coordinate (I) at (-3,-3);
\coordinate (J) at (7,-3);
\coordinate (K) at (6,-5);
\coordinate (L) at (-4,-5);
\draw[thick] (I)node[below=15pt]{$\mathcal{R}_2$} -- (J) -- (K) -- (L) -- cycle;
\draw[thick] (-2,-4) -- (3,-4);    
    \draw[thick] (-2,-1) -- (5,-1);      
    \draw[thick] (3,2) -- (5,2);       
\draw[dashed] (-2,-4) -- (-2,-1);
\draw[dashed] (3,-4) -- (3,-1) -- (3,2);
\draw[dashed] (5,-1) -- (5,2);
\end{tikzpicture}
\end{center}
\caption{The three-sheeted Riemann surface $\mathcal{R}$}
\label{Fig_Riemann surface}
\end{figure}
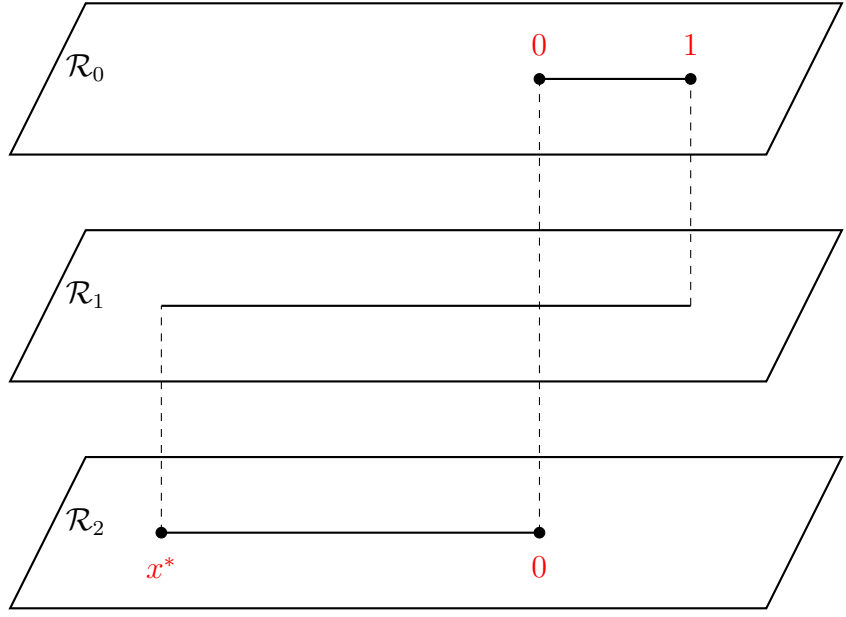

Then, for $\theta_1 =\frac{2-\theta}{4-\theta}$, $\theta_2 = \frac{\theta}{\theta+2}$, and $c_0= \frac{4}{(2+\theta)(4-\theta)}$, the following cubic equation with respect to the variable $s$
\begin{align}\label{Mapping R(s)}
z=\pi \circ \phi(s)=R(s)=\frac{4s^3}{((2+\theta)s-\theta)((4-\theta)s-2+\theta)}= \frac{c_0s^3}{(s-\theta_1)(s-\theta_2)}
\end{align}
defines an algebraic function of $3^{rd}$ degree, which is single-valued on some compact three-sheet Riemann surface $\mathcal{R}$. The rational function $R$ is a conformal mapping of the complex sphere $\bar{\mathbb{C}}_s$ onto the Riemann surface $\mathcal{R}$. Precisely, $R=\pi \circ \phi$ where $\pi: \mathcal{R} \rightarrow \bar{\mathbb{C}}_z$ is the canonical projection and $\phi:\bar{\mathbb{C}}_s \rightarrow \mathcal{R}$ is the conformal mapping. The inverse images of the three sheets $\mathcal{R}_j$ under the mapping $R$ are shown in Fig. \ref{Fig_Inverse of R}. 

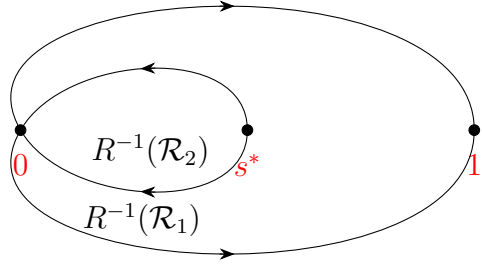
\begin{figure}[htbp!]
\begin{center}
  \begin{tikzpicture}
\filldraw (0,0)node[below=5pt,red]{$0$} circle [radius=2pt];
\filldraw (3,0)node[below=5pt,red]{$s^*$} circle [radius=2pt];
\filldraw (6,0)node[below=5pt,red]{$1$} circle [radius=2pt];
\draw[arc arrow=to pos .5 with length 2mm] (3,0) to[out=90,in=60] (0,0);
\draw[arc arrow=to pos .5 with length 2mm] (3,0) to[out=-90,in=-60] node[midway,above=5pt]{$R^{-1}(\mathcal{R}_2)$} (0,0);
\draw[arc arrow=to pos .5 with length 2mm] (0,0) to[out=-120,in=-90]node[pos=0.4, above]{$R^{-1}(\mathcal{R}_1)$} (6,0);
\draw[arc arrow=to pos .5 with length 2mm] (0,0) to[out=120,in=90] (6,0);
\draw[arc arrow=to pos .5 with length 2mm] (0,0) to[out=-90,in=-120] (0,0);
\end{tikzpicture}
\end{center} 
\caption{Inverse images of sheets $\mathcal{R}_j$ under the mapping $R$}
\label{Fig_Inverse of R}
\end{figure}

The critical points of $\mathcal{R}$ are $0,0,1,s^*$ and values at these points are $0,0,1,x^*$, where
\begin{align*}
s^* =\frac{3\theta(2-\theta)}{(2+\theta)(4-\theta)}=3\theta_1\theta_2,
\end{align*}
and $x^*=-3\left(\frac{s^\ast}{1-\theta}\right)^2$ as defined by \eqref{eq_x*}. It is easy to check that
\[
\theta\in(0,1):\quad \frac12>\theta_1>\frac13>s^\ast>\theta_2>0,\quad \frac12>c_0>\frac49.
\]
As $z \rightarrow \infty$, the three branches $s_0(z), s_1(z), s_2(z)$ to \eqref{Mapping R(s)} which we choose are given as
\begin{align*}
& s_0(z)= \frac{z}{c_0} -(\theta_1+\theta_2)-\frac{c_0 (\theta_1^2+\theta_1\theta_2+\theta_2^2)}{z}+O\left(\frac{1}{z^2}\right),\\
& s_1(z)=\theta _1
+\frac{c_1}{ z}+
\frac{c^2 \left(2 \theta _1-3 \theta _2\right) \theta _1^5}{\left(\theta _1-\theta _2\right){}^3 z^2}
+O\left(z^{-3}\right),\\
& s_2(z)=\theta _2-\frac{c_2}{ z}+
\frac{c^2 \left(2 \theta _2-3 \theta _1\right) \theta _2^5}{\left(\theta _2-\theta _1\right){}^3 z^2}+O\left(z^{-3}\right)
\end{align*} 
with
$$
c_1=\frac{c_0 \theta _1^3}{\theta _1-\theta _2}=
\frac{(2-\theta)^3}{(4-\theta)^3(1-\theta)},\quad c_2=
\frac{c_0 \theta _2^3}{\theta _1-\theta _2}=\frac{\theta^3}{(2+\theta)^3(1-\theta)}.
$$
Similarly, as $z \to 0$, we have
\[
s_{0,1,2}(z)=
\left(\frac{\theta _1 \theta _2}{c_0}\right)^{1/3} z^{1/3} -
\frac{\theta _1 +\theta _2}{3c_0^{2/3}(\theta _1 \theta _2)^{1/3}} z^{2/3} + \frac{z}{3c_0}
+O\left(z^{4/3}\right),\quad z\to0,
\]
where we define the branches of \( z^{1/3} \) as follows:
\[
\begin{aligned}
z^{1/3} &:= |z|^{1/3} e^{i(\arg z - 2\pi)/3}, && \arg z \in (0,\pi), \quad \text{for } s_1;\\[4pt]
z^{1/3} &:= |z|^{1/3} e^{i(\arg z + 2\pi)/3}, && \arg z \in (-\pi,0), \quad \text{for } s_1;\\[4pt]
z^{1/3} &:= |z|^{1/3} e^{i(\arg z + 2\pi)/3}, && \arg z \in (0,2\pi), \quad \text{for } s_0, \text{ so that } s_0(x)<0 \text{ for } x<0;\\[4pt]
z^{1/3} &:= |z|^{1/3} e^{i\arg z/3}, && \arg z \in (-\pi,\pi), \quad \text{for } s_2, \text{ so that } s_2(x)>0 \text{ for } x>0.
\end{aligned}
\]
For $z \to 1$, we get
\begin{align*}
&s_{0,1}(z)=1+
\frac{2(z-1)^{1/2}}{(\theta^2-2\theta+4)^{1/2}}  +
\frac{6(\theta^2-2\theta+8/3) }{(\theta^2-2\theta+4)^{2}}(z-1)+O((z-1)^{3/2}) ,\quad z\to1,\\
&s_{2}(z)=\frac{\theta (2-\theta)}{4}+
	\frac{\theta^3 (2-\theta)^3}{4(\theta^2-2\theta+4)^{2}}(z-1)+
	O\left((z-1)^{2}\right),\quad z\to1, 
\end{align*}
where $(z-1)^{1/2}:=(-1)^j|z-1|e^{i \arg(z-1)/2}$, $\arg(z-1)\in (-\pi,\pi)$ for $s_{j}$, $j=0,1$. 

Further, as $z \to x^\ast$:
\begin{align*}
&s_{0}(z)=-\frac{3 \theta (2-\theta) }{4 (1-\theta )^2}
-\frac{(2+\theta)^4(4-\theta)^4  }{324\, \theta (2-\theta)  (\theta^2 -2 \theta +4)^2}(x^\ast-z)+O\left((x^\ast-z)^{2}\right),\quad z\to x^\ast,\\
&s_{1,2}(z)=s^\ast+
	\frac{2(1-\theta)^2}{3(\theta^2-2\theta+4)^{1/2}}(x^\ast-z)^{1/2}
	+
	O\left(x^\ast-z\right),\quad z\to x^\ast,
\end{align*}
where $(x^\ast-z)^{1/2}:=(-1)^{j-1}|x^\ast-z|e^{i \arg(x^\ast-z)/2}$, $\arg(x^\ast-z)\in (-\pi,\pi)$ for $s_{j}$, $j=1,2$. 
The critical points on the sheets of Riemann surface $\mathcal{R}$ and their inverse images on $s$-plane are tabulated in Table \ref{tab:z-s-correspondence}.
\begin{table}[h]
\centering
\caption{Correspondence between $z$ and $s$ values for branch points of $\mathcal{R}$}
\label{tab:z-s-correspondence}
\begin{tabular}{c c}
\toprule
\textbf{Branch points of $\mathcal{R}$ ($z$-values)} & \textbf{Points on $s$-sphere} \\
\midrule
$1_{01}$      & $1$        \\
$1_2$         & $\theta(2-\theta)/4$      \\
$0$           & $0$        \\
$\infty_{1}$ & $\theta_1$      \\
$\infty_{2}$ & $\theta_2$      \\
$\infty_0$    & $\infty$   \\
\bottomrule
\end{tabular} 
\end{table}

\noindent Define a meromorphic function $h$ on $\mathcal{R}$ by the divisor 
\begin{align*}
    \frac{\infty_0, \infty_1, \infty_2}{1_{01}, (0_{012})^2}
\end{align*}
and normalize it by 
\begin{align*}
    zh_0(z) \rightarrow 2, \quad z\rightarrow \infty.
\end{align*}
The motivation of this function is given by  Proposition~\ref{prop_equi_algebraic} below. By $h_j$ for $j=0,1,2$, we will mean the restriction of the function $h$ to the sheet $\mathcal{R}_j$ without the corresponding cut.
Then, there exist unique rational functions (up to a constant factor) $h$ on $\mathcal{R}$ given by
\begin{align}\label{function h}
h_j(z)  = \frac{((2+\theta)s_j(z)-\theta)((4-\theta)s_j(z)-2+\theta)}{2s_j(z)^2(s_j(z)-1)} = \frac{2s_j(z)}{z(s_j(z)-1)}.   
\end{align}
They are the solutions of the cubic equation:
$$
h^3+\frac{\left(\theta ^2-2 \theta +4\right) h}{(1-z)z}+\frac{2 \theta(2-\theta)  }{(1-z) z^2}=0.
$$
The three branches at $\infty$ are given by
\begin{align*}
&    h_0(z)=\frac{2}{z}+\frac{8}{(4-\theta) (2+\theta) z^2}+O\left(z^{-3}\right), \quad z \rightarrow \infty,\\
& h_1(z)=-\frac{2-\theta}{z}-\frac{(2-\theta)^3}{2 (4-\theta ) (1-\theta ) z^2}+O\left(z^{-3}\right), \quad z \rightarrow \infty,\\
& h_2(z)=-\frac{\theta}{z}-\frac{\theta ^3}{2 (\theta +2) (1-\theta)z^2}+O\left(z^{-3}\right), \quad z \rightarrow \infty.
\end{align*}
At $z \to 0$, we obtain
\begin{align}\label{eq: h0 h1 h2 near 0}
&h_{0,1,2}(z)=\frac{e_2(z)}{z^{2/3}}+\frac{e_1(z)}{z^{1/3}},\quad
e_j\in\mathcal O(0),\notag \\
& e_2(0)=-(2\theta(2-\theta))^{1/3},\quad e_1(0)=\frac{\theta^2 -2\theta +4}{3(2\theta(2-\theta))^{1/3}}.
\end{align}
For $z \to 1$, we have
\begin{align}
&h_{0,1}(z)=
	\frac{(\theta ^2-2 \theta +4)^{1/2}}{(z-1)^{1/2}}+
	\frac{\theta(2 -\theta)}{\theta ^2-2 \theta +4}+
		O\left((z-1)^{1/2}\right),\quad z\to1, \label{eq: h0 h1 near 1}\\
&	h_{2}(z)=\frac{-2\theta(2 -\theta)}{\theta ^2-2 \theta +4}+
	O\left(z-1\right),\quad z\to1. \notag
\end{align}
Further, as $z \to x^\ast$, we get
\begin{equation}
	h_{1,2}(z)=
	\frac{(4-\theta)^2 (1-\theta )^2 (2+\theta )^2}{9 \theta  (2-\theta) (\theta ^2-2 \theta +4)}+
	\frac{(4-\theta)^4 (1-\theta )^4 (2+\theta )^4 }{81\theta ^2  (2-\theta )^2 (\theta ^2-2 \theta +4)^{5/2}}(x^\ast-z)^{1/2}
	+
	O(x^\ast-z). \label{eq: h1 h2 near x*}
\end{equation}
We now formulate the main result of this section.
\begin{proposition}\label{prop_equi_algebraic}
Let $h$ be the function defined by \eqref{function h} and let $h_0$ and $h_2$ be its holomorphic branches on $\mathbb{C} \backslash \Delta$ and $\mathbb{C} \backslash \Gamma^*$, respectively. Then the solution to the equilibrium problem \eqref{equilibrium problem} can be expressed in terms of $h_0$ and $h_2$ as follows:
\begin{align} 
&d\lambda_1(x) = -\frac{1}{2\pi i} (h_0^+-h_0^-)(x), \quad x \in \Delta, \label{eq: D lambda1 in h0}\\
& d\lambda_2(x) = \frac{1}{2\pi i} (h_2^+-h_2^-)(x), \quad x \in \Gamma^*. \label{eq: D lambda2 in h2}
\end{align}
\end{proposition}
\begin{proof}
For $\theta \in (0,1)$, the function $h_0$ and $h_2$ are holomorphic in the neighbourhood of $\infty$ and have residues $-2$ and $\theta$. Then, for mass $\lambda_1$, we have 
\begin{align*}
|\lambda_1| = -\frac{1}{2\pi i} \int_\Delta (h_0^+-h_0^-)(x) dx= \frac{1}{2\pi i} \int_{\gamma_\Delta} h_0(z) dz= 2,
\end{align*}
where $\gamma_\Delta$ is a contour that goes around $\Delta$ counterclockwise. Similarly, $|\lambda_2| = \theta $. The case for $\theta = 1$ can be proved by passing the limit. Thus, the measures $\lambda_1$ and $\lambda_2$ are positive. Next, note that by the Sokhotskii-Plemelj formula, the functions $h_0$ and $h_2$ are the Cauchy transforms of the measures $\lambda_1$ and $\lambda_2$:
\begin{align*}
h_0(z):= -\int \frac{d\lambda_1(x)}{x-z}, \quad z \in \mathbb{C}\backslash\Delta, \qquad h_2(z):= \int \frac{d\lambda_2(x)}{x-z}, \quad z \in \mathbb{C}\backslash\Gamma^*.
\end{align*}
Representing the Cauchy transform of the measures $\lambda_j$, $j=1,2$ as $\widehat{\lambda}_j$, $j=1,2$, i.e.,
\begin{align*}
    \widehat{\lambda}_j(z) = \int \frac{d\lambda_j(x)}{x-z}, \quad j=1,2
\end{align*}
implies
\begin{align*}
\widehat{\lambda}_1(z)=-h_0(z), \quad z \in \mathbb{C} \backslash \Delta, \qquad
\widehat{\lambda}_2(z)=h_2(z), \quad z \in \mathbb{C} \backslash \Gamma^*.
\end{align*}
The complex potentials $U^{\lambda_j}$, $j=1,2$ are defined as 
\begin{align}\label{eq: U complex potential}
U^{{\lambda_j}}(z) = -\int \ln (z-x) d\lambda_j(x). 
\end{align}
Further, the real and the complex potentials are related by the identity $V^{\lambda_j}(x) = \frac{1}{2}(U_+^{\lambda_j}+U_-^{\lambda_j})$, $j=1,2$. Thus, for the equality relation to hold in \eqref{equilibrium problem}, we must have
\begin{align*}
&2V^{\lambda_1}-V^{\lambda_2} = \omega_1, \quad~ {\rm on}~ (0,1) \\
\Rightarrow& U_+^{\lambda_1}+U_-^{\lambda_1}-V^{\lambda_2} = \omega_1 \\
\Rightarrow&\frac{d}{dx}(U_+^{\lambda_1}+U_-^{\lambda_1}-V^{\lambda_2}) = 0 \\
\Rightarrow& \widehat{\lambda}_1^+ + \widehat{\lambda}_1^- - \widehat{\lambda}_2 = 0 \\
\Rightarrow& -h_0^+ - h_0^- -h_2 = 0 \\
\Rightarrow& h_0^+ + h_0^- +h_2 = 0 
\end{align*}
which holds true since the boundary values of the branches $h_0$, $h_1$, and $h_2$ satisfy the relations
\begin{align*}
h_0^\pm (x) =  h_1^\mp (x), \quad x \in (0,1), \qquad h_1^\pm (x) =  h_2^\mp (x), \quad x \in (x^*,0)
\end{align*}
and $h_0+h_1+h_2 = 0$. Similarly, one can prove the equality relation in \eqref{equilibrium problem 2} by showing $h_2^+ + h_2^- +h_0 = 0 $ also holds in view of the above relations.
\end{proof}

\begin{proposition}
The measures $\lambda_1$ and $\lambda_2$ defined on $\Delta$ and $\Gamma^*$ have the following asymptotic behavior near the end points:
\begin{align}
&\lim_{x\to 1-}(1-x)^{1/2}\lambda'_1(x)=\frac{(\theta ^2-2 \theta +4)^{1/2}}{\pi}, \label{eq: measure behave at 1}\\
&\lim_{x\to x^\ast+}(x-x^\ast)^{-1/2}\lambda'_2(x)=\frac{(4-\theta)^4 (1-\theta )^4 (2+\theta )^4 }{81\pi\theta ^2  (2-\theta )^2 (\theta ^2-2 \theta +4)^{5/2}}, \label{eq: measure behave at x*}\\
&\lim_{x\to 0-}x^{2/3}\lambda'_2(x)=\lim_{x\to 0+}x^{2/3}\lambda'_1(x)=\frac{-\sin\frac{2\pi}{3}e_2(0)}{\pi }=\frac{{3}^{1/2}(2\theta(2-\theta))^{1/3}}{2\pi } \label{eq: measure behave at 0}.
\end{align}
\end{proposition}
\begin{proof}
Using \eqref{eq: h0 h1 near 1}, \eqref{eq: D lambda1 in h0} implies
\begin{align*}
&d\lambda_1(x) = -\frac{1}{2\pi i} (h_0^+-h_0^-)(x)=-\frac{1}{2\pi i}\bigg[\frac{(\theta ^2-2 \theta +4)^{1/2}}{|z-1|^{1/2}e^{i\pi/2}}-\frac{(\theta ^2-2 \theta +4)^{1/2}}{|z-1|^{1/2}e^{-i\pi/2}} \bigg]\\
&=-\frac{1}{2\pi i}\frac{(\theta ^2-2 \theta +4)^{1/2}}{|z-1|^{1/2}} (e^{-i\pi/2}-e^{i\pi/2}) = \frac{(\theta ^2-2 \theta +4)^{1/2}}{\pi |z-1|^{1/2}} .
\end{align*}
Hence, the behavior $\eqref{eq: measure behave at 1}$ follows. Similarly, using \eqref{eq: h1 h2 near x*}, \eqref{eq: D lambda2 in h2} becomes
\begin{multline*}
d\lambda_2(x) = \frac{1}{2\pi i} (h_2^+-h_2^-)(x) = \frac{1}{2\pi i}\frac{(4-\theta)^4 (1-\theta )^4 (2+\theta )^4 }{81\theta ^2  (2-\theta )^2 (\theta ^2-2 \theta +4)^{5/2}}|x^\ast-z|^{1/2}(e^{i\pi/2} -  e^{-i\pi/2}) \\= \frac{(4-\theta)^4 (1-\theta )^4 (2+\theta )^4 }{81 \pi \theta ^2  (2-\theta )^2 (\theta ^2-2 \theta +4)^{5/2}} |x^\ast-z|^{1/2}
\end{multline*}
which implies the behavior \eqref{eq: measure behave at x*}. Now, using \eqref{eq: h0 h1 h2 near 0} in \eqref{eq: D lambda1 in h0} and \eqref{eq: D lambda2 in h2}, we have
\begin{align*}
&d\lambda_1(x)=-\frac{e_2(0)}{2\pi i } \bigg[ \bigg(\frac{1}{z_+^{1/3}} \bigg)^2- \bigg(\frac{1}{z_-^{1/3}} \bigg)^2\bigg]=-\frac{e_2(0)}{2\pi i |z|^{2/3}}\bigg[ \bigg(\frac{1}{|z|^{1/3}e^{2i\pi/3}} \bigg)^2- \bigg(\frac{1}{|z|^{1/3}e^{4i\pi/3}} \bigg)^2\bigg]\\
&= -\frac{e_2(0)}{2\pi i |z|^{2/3}} (e^{-4i\pi/3}-e^{-8i\pi/3}) = -\frac{e_2(0)} {2\pi i |z|^{2/3}} \times i \sqrt{3} = \frac{{3}^{1/2}(2\theta(2-\theta))^{1/3}}{2\pi }.
\end{align*}
Similarly, one can find $d\lambda_1(x)$ near $0$ and hence, we arrive at \eqref{eq: measure behave at 0}.
\end{proof}


\section{The $H$ functions} 
We define three functions $H_j$ such that
\begin{enumerate}
\item $H_0,H_1$ are analytic in $\mathbb C\setminus (-\infty,1]$, $H_2$ is analytic in $\mathbb C\setminus (-\infty,0]$; 
\item $H'_j(z)=h_j(z)$; 
\item $H_j$ are real on $(1,\infty)$;
\item $\mathrm{Re}\, H_j(0)$=0.
\end{enumerate}
Then 
\begin{multline*}
	H_j(z)=\int^z \frac{2s_j(z)d\zeta}{(s_j(z)-1)\zeta}
	=\int^{s_j(z)} \frac{2sR'(s)ds}{(s-1)R(s)}\\=
	\int^{s_j(z)} \left(\frac3{s}-\frac1{s-\theta_1}-\frac1{s-\theta_2}\right)\frac{2sds}{(s-1)}\\=(2-\theta)\ln(s_j(z)-\theta_1)+
	\theta\ln(s_j(z)-\theta_2)-\mathrm{const}_j,
\end{multline*} 
$$\mathrm{Res}_{s=1}\left(\frac3{s}-\frac1{s-\theta_1}-\frac1{s-\theta_2}\right)\frac{2sds}{(s-1)} = 6-(4-\theta)-(2+\theta)=0.
$$
If we introduce a constant 
$$
C=(2-\theta)\ln\frac1\theta_1+
\theta\ln\frac1\theta_2>0, 
$$
then we get the following explicit formulae for $H_j$:
\begin{align*}
&H_j(z)=(2-\theta)\ln(s_j(z)-\theta_1)+
	\theta\ln(s_j(z)-\theta_2)+C; \quad  j=0,1, 	\\
&	H_2(z)=(2-\theta)\ln(\theta_1-s_2(z))+
	\theta\ln(\theta_2-s_2(z))+C.
\end{align*}
The expansions at $\infty$ are the following:
\begin{equation}\label{eq: H012 at infty}
\left.
\begin{aligned}
&H_0(z)=2\ln z+C_0-
\frac{8}{(4-\theta) (2+\theta) z}
+O\left(z^{-2}\right),\\
& H_1(z)=(\theta -2)\ln z+C_1
+
\frac{(2-\theta)^3}{2 (4-\theta ) (1-\theta ) z}
+O\left(z^{-2}\right),\\
& H_2(z)=-\theta \ln z+C_2
-\frac{\theta ^3}{2 (\theta +2) (1-\theta)z}
+O\left(z^{-2}\right),
\end{aligned}
\right\}
\end{equation}
with 
\begin{align*}
&C_0=C+2\ln \frac{1}{c_0}=
\theta\ln\frac{2+\theta}{\theta}-(2-\theta)\ln\frac{2-\theta}{4-\theta}+2\ln\frac{(2+\theta)(4-\theta)}{4},\\
&C_1=C+(2-\theta)\ln c_1+\theta\ln(\theta_1-\theta_2)=
\theta\ln\frac{4(1-\theta)}{\theta(4-\theta)}-(2-\theta)\ln\frac{(1-\theta)(4-\theta)^2}{(2-\theta)^2},\\
&C_2=C+(2-\theta)\ln (\theta_1-\theta_2)+\theta\ln c_2=\theta\ln\frac{\theta^2}{(2+\theta)^2(1-\theta)}-
(2-\theta)\ln\frac{(2+\theta)(2-\theta)}{4(1-\theta)}.
\end{align*}
Since $h_0+h_1+h_2\equiv0$,  it follows that $H_0+H_1+H_2 \equiv {\rm Const}$. It can be verified that $C_0+C_1+C_2=0$ so we have
$$(H_0+H_1+H_2)(z)=\lim_{z\to\infty}(H_0+H_1+H_2)(z)=C_0+C_1+C_2=0.$$ 
It is easy to see that
$$
H_j(z)=\pm 2\pi i (-1)^j+\int_{0\pm i0}^z h_j(\zeta)d\zeta=H_{01}(1)+\int_{1}^z h_j(\zeta)d\zeta,\quad z\in \mathbb C\setminus (-\infty,1], \quad j=0,1, 
$$
$$
H_2(z)=\int_{0}^z h_2(\zeta)d\zeta,\quad z\in \mathbb C\setminus (-\infty,0],
$$
where
$$
H_{01}(1)=(2-\theta)\ln\left(\frac{1}{\theta_1}-1\right)+
\theta\ln\left(\frac{1}{\theta_2}-1\right).
$$
Further, the expansions of $H$-functions as $z \to 1$ are as follows:
\begin{equation}\label{eq: H012 at 1}
H_{01}(z)=H_{01}(1)+
	2b_1(z-1)^{1/2}+
	\frac{\theta(2 -\theta)}{\theta ^2-2 \theta +4}(z-1)+
	O\left((z-1)^{3/2}\right),\quad z\to1.
\nonumber
\end{equation}
Similarly, we have the following as $z \to x^\ast$:
\begin{equation}\label{eq: H012 at x*}
\left.
\begin{aligned}
&	H_{1}(z)=\mathrm{Re}\,H_{12}(x^\ast)\mp\pi i(2-\theta)+
\frac{(4-\theta)^2 (1-\theta )^2 (2+\theta )^2}{9 \theta  (2-\theta) (\theta ^2-2 \theta +4)}(z-x^\ast)\\
&-\frac23 b_\ast(x^\ast-z)^{3/2}
+
O(z-x^\ast)^2,\quad z\to x^\ast,\quad \pm\mathrm{Im}\,z>0,\\
& 	H_{2}(z)=\mathrm{Re}\,H_{12}(x^\ast)\mp\pi i\theta+
	\frac{(4-\theta)^2 (1-\theta )^2 (2+\theta )^2}{9 \theta  (2-\theta) (\theta ^2-2 \theta +4)}(z-x^\ast)\\
&-\frac23 b_\ast(x^\ast-z)^{3/2}+
	O(z-x^\ast)^2,\quad z\to x^\ast,\quad \pm\mathrm{Im}\,z>0,
\end{aligned}
\right\}
\end{equation}
where
$$
\mathrm{Re}\,H_{12}(x^\ast)=C+(2-\theta)\ln(\theta_1-s^\ast)+
\theta\ln(s^\ast-\theta_2).
$$
We also get the following expansions as $z \to 0$: 
\begin{align*}
&H_{j}(z)=\pm 2\pi i (-1)^j+3e_2(0)z^{1/3}+\frac32e_1(0)z^{2/3}+O(z^{4/3}),\quad j=0,1,\quad z\to 0,\quad\pm\mathrm{Im}\,z>0,\\
&H_{2}(z)=3e_2(0)z^{1/3}+\frac32e_1(0)z^{2/3}+O(z^{4/3}),\quad z\to 0.
\end{align*}
If $\tilde{H}_j:=H_j-C_0$, then, in the view of \eqref{eq: H012 at infty}, we have
\begin{align*}
&e^{\tilde{H}_0(z)} = z^2 (1+O(1)), \quad z \to \infty, \\
&e^{\tilde{H}_1(z)} = e^{C_1-C_0}z^{\theta-2}  (1+O(1)), \quad z \to \infty,\\
& e^{\tilde{H}_2(z)} = e^{C_2-C_0} z^{-\theta}  (1+O(1)), \quad z \to \infty.
\end{align*}
\begin{lemma} The following relations between the $H$-functions and complex vector potentials exist:
\begin{align*}
\tilde{H}_0 = -U^{\lambda_1},
    \quad\tilde{H}_1 = U^{\lambda_1}-U^{\lambda_2},
    \quad\tilde{H}_2 = U^{\lambda_2}.
\end{align*}
\end{lemma} 
\begin{proof}
Note that $\tilde{H}'_0=H'_0=h_0=-\widehat{\lambda}_1=(-U^{\lambda_1})'$, where $U^{\lambda_1}$ is defined by \eqref{eq: U complex potential}. From this, we achieve $\tilde{H}_0 = -U^{\lambda_1}$. Similarly, $\tilde{H}'_1=H'_1=h_1=-h_0-h_2=\widehat{\lambda}_1-\widehat{\lambda}_2=(U^{\lambda_1})'-(U^{\lambda_2})' = (U^{\lambda_1}-U^{\lambda_2})'$. Similarly, $\tilde{H}'_2=H'_2=h_2=\widehat{\lambda}_2=(U^{\lambda_2})'$.
\end{proof}
Further, the boundary conditions for $H$-functions are given by
\begin{equation}\label{eq: Boundary conditions H}
\left.
\begin{aligned}
&H_0^\pm(x)=H_1^\mp(x), \quad x\in (0,1),\\
&
H_2^\pm(x)-H_1^\mp(x)=\mp2\pi i,\quad x\in (x^\ast,0),\\
&
H_0^+-H_0^-=4\pi i,\,\,
H_1^+-H_1^-=-2\pi i(2-\theta),\,\,
H_2^+-H_2^-=-2\pi i\theta,\quad x\in (-\infty,x^\ast).
\end{aligned}
\right\}
\end{equation}
Moreover on $(-\infty,0)$ we have
$$
\mathrm{Im}\,H_0^\pm=\pm2\pi i,
$$
and on $(-\infty,x^\ast)$ we have
$$
\mathrm{Im}\,H_1^\pm=\mp\pi i(2-\theta), \quad
\mathrm{Im}\,H_2^\pm=\mp\pi i\theta. 
$$


\begin{remark}
The equilibrium constants $\omega_1$ and $\omega_2$ present in \eqref{equilibrium problem} and \eqref{equilibrium problem 2} can be represented in terms of the constants $C_0$, $C_1$, and $C_2$. Since on $\Delta $, we have
\begin{multline*}
0=2V^{\lambda_1}-V^{\lambda_2}-\omega_1=
U_+^{\lambda_1}+U_-^{\lambda_1}-U^{\lambda_2}-\omega_1\\=
-(H_0^++H_0^-+H_2)+2C_0+C_2-\omega_1=-(H_0^++H_1^++H_2)+2C_0+C_2-\omega_1=2C_0+C_2-\omega_1,
\end{multline*}
then
$$
\omega_1=2C_0+C_2=C_0-C_1.
$$
At the same time on $\Gamma^\ast$, we have
 \begin{multline*}
 	0=2V^{\lambda_2}-V^{\lambda_1}-\omega_2=
 	U_+^{\lambda_2}+U_-^{\lambda_2}-\mathrm{Re}\,U^{\lambda_1}-\omega_2=
 	(H_2^+ +H_2^-+ \mathrm{Re}\, H_0)-2C_2-C_0-\omega_2 \\=(H_2^++H_1^++2\pi i+H_0^+-2\pi i)-2C_2-C_0-\omega_2=-2C_2-C_0-\omega_2,
 \end{multline*}
 so,
 $$
 \omega_2=-2C_2-C_0=C_1-C_2.
 $$
 Moreover on $(-\infty,x^\ast)$ we have
 \begin{equation}\label{eq: Real H2-H1 +ve}
	0<2V^{\lambda_2}-V^{\lambda_1}-\omega_2=
\mathrm{Re}\,(H_2-H_1).
\end{equation}
\end{remark}

\section{Steepest Descent Analysis of RH Problem}
In the upcoming sections, we apply the Deift-Zhou steepest descent method to the RH problem for $Y$. The method consists of a series of explicit transformation $Y \to X \to U \to T \to R$. We will do it for the case when $\theta \in \mathbb{Q}$ and $\theta < 1$. 
\subsection{The first transformation}
Define the first transformation $Y \rightarrow X$ as
\begin{align*}
    X(z) = Y(z) \begin{pmatrix}
        1 &0 &0\\
        0 &1 & \dfrac{-z^{\alpha_2-\alpha_1}}{1-e^{2 \pi i (\alpha_2-\alpha_1)}}\\ 
         0& 0& 1 
    \end{pmatrix}.
\end{align*}
This transformation introduces a jump on the negative real axis, and one can easily compute the jump matrices as
\begin{align}\label{eq: Jump matrices for X}
 \mathcal{J}_X= \begin{pmatrix}
1 & w_1 & 0 \\
0 & 1 & 0 \\
0 & 0 & 1 
\end{pmatrix} ~~{\rm and}~ \mathcal{J}_X= \begin{pmatrix}
1 & 0 & 0 \\
0 & 1 & c_\alpha\rho \\
0 & 0 & 1 
\end{pmatrix} 
\end{align}
on $\Delta=(0,1) $ and $\Gamma=(-\infty,0) $, respectively, and where $\rho = (-z)^{\alpha_2-\alpha_1} > 0$, $\arg z =\pi$ and $c_\alpha = e^{ \pi i (\alpha_2-\alpha_1)}\neq 0$. Thus, $X$ satisfies the following RH problem:
\begin{itemize}
\item $X$ is analytic in $\mathbb{C} \backslash (\Delta \cup \Gamma)$.
\item $X$ satisfies jump conditions $X_+ = X_- J_X$ where $J_X$ is given by \eqref{eq: Jump matrices for X}.
\item $X $ has the following behavior at infinity:
\begin{align*}
&    X(z) =\bigg(\mathbb{I}+O\bigg(\frac{1}{z} \bigg) \bigg) \begin{pmatrix}
z^{n_1+n_2} & 0 & 0 \\
0 & z^{-n_1} & -z^{-n_1+\alpha_2-\alpha_1} \\
0 & 0 & z^{-n_2}
\end{pmatrix} , \quad z \rightarrow \infty.\\
&=\bigg(\mathbb{I}+O\bigg(\frac{1}{z} \bigg) \bigg) \begin{pmatrix}
1 & 0 & 0 \\
0 & 1 & -z^{n_2-n_1+\alpha_2-\alpha_1} \\
0 & 0 & 1
\end{pmatrix} \begin{pmatrix}
z^{n_1+n_2} & 0 & 0 \\
0 & z^{-n_1} & 0 \\
0 & 0 & z^{-n_2}
\end{pmatrix} \\
&=\bigg(\mathbb{I}+O\bigg(\frac{1}{z} \bigg) \bigg) \begin{pmatrix}
z^{n_1+n_2} & 0 & 0 \\
0 & z^{-n_1} & 0 \\
0 & 0 & z^{-n_2}
\end{pmatrix}
\end{align*}
since it is possible to choose $n_1$ and $n_2$ such that $n_2-n_1+\alpha_2-\alpha_1 \leq -1$ and $n_2-n_1 \to -\infty$ since $\dfrac{n_2}{n_1+n_2} \rightarrow \dfrac{\theta}{2} \in (0,1/2)$.
\item The behavior of $X$ at the endpoints 0 and 1 is as follows:
\begin{align*}
&X(z)= O\begin{pmatrix}
1 & t(z) & t(z) \\
1 & t(z) & t(z) \\
1 & t(z) & t(z)
\end{pmatrix} ~ z \rightarrow 1,~ {\rm where}~ t(z) = \begin{cases}
  |z-1|^{\beta},  & \beta < 0\\
   \log|z-1|, & \beta = 0\\
    1, &\beta > 0.
\end{cases} 
\end{align*}
If $\alpha_1 < 0$,
\begin{align*}
&X(z)= O\begin{pmatrix}
1 & |z|^{\alpha_1} & w(z) \\
1 & |z|^{\alpha_1} & w(z) \\
1 & |z|^{\alpha_1} & w(z)
\end{pmatrix} ~ z \rightarrow 0, ~{\rm where}~ w(z) = \begin{cases}
  |z|^{\alpha_2},  & \alpha_2< 0\\
   \log|z|, & \alpha_2= 0\\
    1, &\alpha_2 > 0.
\end{cases} 
\end{align*}
If $\alpha_1 = 0$,
\begin{align*}
&X(z)= O\begin{pmatrix}
1 & \log|z| & w(z) \\
1 & \log|z| & w(z) \\
1 & \log|z| & w(z)
\end{pmatrix} ~ z \rightarrow 0, ~{\rm where}~ w(z) = \begin{cases}
  |z|^{\alpha_2},  & \alpha_2< 0\\
    1, &\alpha_2 > 0.
\end{cases} 
\end{align*}
If $\alpha_1 > 0$,
\begin{align*}
&X(z)= O\begin{pmatrix}
1 & 1 & w(z) \\
1 & 1 & w(z) \\
1 & 1 & w(z)
\end{pmatrix} ~ z \rightarrow 0, ~{\rm where}~ w(z) = \begin{cases}
  |z|^{\alpha_2-\alpha_1},  & \alpha_2< 0~or~ 0<\alpha_2 < \alpha_1\\
  |z|^{-\alpha_1},  & \alpha_2= 0\\
    1, &\alpha_2 > 0, \alpha_2 > \alpha_1.
\end{cases} 
\end{align*}

\end{itemize}

\subsection{The Second Transformation}
Define the second transformation $X \rightarrow U$ as 
\begin{align}\label{Transformation U}
  U(z)=  \begin{pmatrix}
        1 &0 &0\\
        0 &e^{-n\omega_1} & 0\\ 
         0& 0& e^{-n\omega_1-n\omega_2} 
    \end{pmatrix} X(z)     \begin{pmatrix}
        e^{-n\tilde{H}_0} &0 &0\\
        0 &e^{-n\tilde{H}_1} & 0\\
         0& 0& e^{-n\tilde{H}_2} \\
    \end{pmatrix}.
\end{align}
This transformation normalizes the RH problem at $\infty$. Now, using \eqref{RH Condition 2} and \eqref{Transformation U}, one can easily obtain the jump matrices
\begin{align}
   & J_U(z) = \begin{pmatrix}
        e^{n(H_0^- - H_0^+)} & w_1e^{n(\tilde{H}_0^- - \tilde{H}_1^+)} &0\\
        0 &e^{n(H_1^- - H_1^+)} & 0\\
         0& 0& e^{n(H_2^- - H_2^+)} \\
    \end{pmatrix}, \quad z \in \Delta, \label{1st Jump Matrix for U}\\
&J_U(z) = \begin{pmatrix}
         e^{n(H_0^- - H_0^+)}& 0 &0\\
        0 &e^{n(H_1^- - H_1^+)} & c_\alpha\rho(z) e^{n(\tilde{H}_1^- - \tilde{H}_2^+)}\\
         0& 0& e^{n(H_2^- - H_2^+)} \\
    \end{pmatrix} , \quad z \in \Gamma^* , \label{2nd Jump Matrix for U}\\
&J_U(z) = \begin{pmatrix}
         1& 0 &0\\
        0 &e^{2 \pi i(2-\theta)n} & c_\alpha\rho(z) e^{n(\tilde{H}_1^- - \tilde{H}_2^+)}\\
         0& 0& e^{2 \pi i \theta n} \\
    \end{pmatrix}, \quad z \in (-\infty,x^*) \notag.
\end{align}
We assume that $\theta n$ and $(2-\theta)n$ are integers.
Thus, $U$ satisfies the following RH problem:
\begin{itemize}
\item $U$ is analytic in $\mathbb{C} \backslash (\Delta \cup \Gamma^*)$.
\item $U$ satisfies jump conditions $U_+ = U_- J_U$ where $J_U$ is given by \eqref{1st Jump Matrix for U} and \eqref{2nd Jump Matrix for U}.
\item $U$ is normalized at infinity:
\begin{align*}
&  U(z) =\mathbb{I}+O(1/z), \quad z \rightarrow \infty. 
\end{align*}
\item The behavior of $U$ at the endpoints 0 and 1 is the same as that of $X$. 
\end{itemize}
\subsection{The Third Transformation: Opening of the lenses} The third transformation is based on the factorization of jump matrices given by \eqref{1st Jump Matrix for U} and \eqref{2nd Jump Matrix for U} as the product of three matrices. By the virtute of \eqref{eq: Boundary conditions H}, we get
\begin{align*}
&\begin{pmatrix}
   e^{n(H_0^- - H_0^+)} & w_1\\
   0 & e^{n(H_1^- - H_1^+)}\\
\end{pmatrix} 
= \begin{pmatrix}
    1& 0\\
   \frac{1}{w_1}e^{n(H_1^- - H_0^-)} & 1\\
\end{pmatrix} \begin{pmatrix}
    0& w_1\\
   -\frac{1}{w_1} & 0\\
\end{pmatrix} \begin{pmatrix}
    1& 0\\
   \frac{1}{w_1}e^{n(H_1^+ - H_0^+)} & 1\\
\end{pmatrix}, \\
&\begin{pmatrix}
   e^{n(H_1^- - H_1^+)} & c_\alpha\rho\\
   0 & e^{n(H_2^- - H_2^+)}\\
\end{pmatrix} 
= \begin{pmatrix}
    1& 0\\
   \frac{1}{c_\alpha\rho}e^{n(H_2^- - H_1^-)} & 1\\
\end{pmatrix} \begin{pmatrix}
    0& c_\alpha\rho\\
   -\frac{1}{c_\alpha\rho} & 0\\
\end{pmatrix} \begin{pmatrix}
    1& 0\\
   \frac{1}{c_\alpha\rho}e^{n(H_2^+ - H_1^+)} & 1\\
\end{pmatrix}.
 \end{align*}

\begin{center}
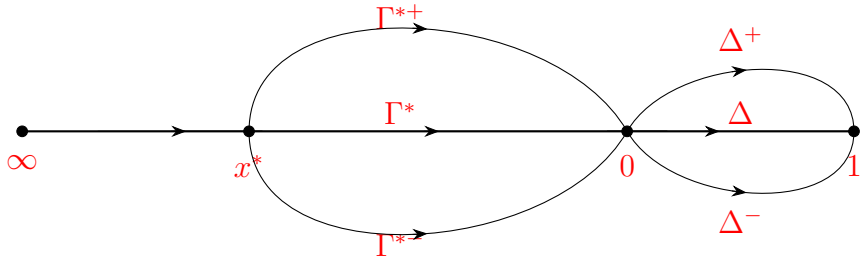
\begin{figure}[htbp!]
\begin{tikzpicture}
\filldraw (1.5,1.7)node[below=5pt,red]{$\Delta^+$};
\filldraw (1.5,-0.7)node[below=5pt,red]{$\Delta^-$};
\filldraw (-3,2)node[below=5pt,red]{$\Gamma^{*+}$};
\filldraw (-3,-1)node[below=5pt,red]{$\Gamma^{*-}$};
\filldraw (-3,0.7)node[below=5pt,red]{$\Gamma^*$};
\filldraw (1.5,0.7)node[below=5pt,red]{$\Delta$};
\filldraw (0,0)node[below=5pt,red]{$0$} circle [radius=2pt];
\filldraw (3,0)node[below=5pt,red]{$1$} circle [radius=2pt];
\filldraw (-8,0)node[below=5pt,red]{$\infty$} circle [radius=2pt];
\filldraw (-5,0)node[below=5pt,red]{$x^*$} circle [radius=2pt];
\draw[thick] (0,0) -- (3,0);    
    \draw[thick] (0,0) -- (-8,0);
\draw[arc arrow=to pos .5 with length 2mm] (-7,0) to[out=0,in=0] (-5,0) ;
\draw[arc arrow=to pos .5 with length 2mm] (-5,0) to[out=-270,in=120] (0,0);
\draw[arc arrow=to pos .5 with length 2mm] (-5,0) to[out=270,in=-120] (0,0);
\draw[arc arrow=to pos .5 with length 2mm] (-6,0) to[out=0,in=0] (0,0);
\draw[arc arrow=to pos .5 with length 2mm] (0,0) to[out=60,in=90] (3,0);
\draw[arc arrow=to pos .5 with length 2mm] (0,0) to[out=-60,in=-90] (3,0);
\draw[arc arrow=to pos .5 with length 2mm] (0,0) to[out=180,in=180] (3,0);
\end{tikzpicture}
\caption{Open lenses}
\label{Fig: Open lenses}
\end{figure}
\end{center} 

Choose smooth paths $\Delta^{\pm}$ connecting $0$ and $1$ with $\Delta^+$ in the upper half plane and $\Delta^-$ in the lower half plane. Similarly, choose $\Gamma^{*+}$ and $\Gamma^{*-}$ enclosing 0 and $x^*$. The paths $\Delta^{\pm}$ and $\Gamma^{*\pm}$ around the intervals $[0,1]$ and $[x^*,0]$, respectively, constitute four bounded regions called the lenses around $[0,1]$, $[x^*,0]$ as shown in Figure \ref{Fig: Open lenses}. We define the next transformation $U \rightarrow T$ as

 \begin{align}\label{Transformtion T}
T:= U \cdot \begin{cases}
    \begin{pmatrix}
      1& 0 &0\\
        -w_1^{-1}e^{n(H_1 - H_0)} &1 & 0\\
         0& 0& 1 \\  
    \end{pmatrix}, & {\rm between}~ \Delta^+ ~{\rm and}~ \Delta, \\
    \begin{pmatrix}
      1& 0 &0\\
        w_1^{-1}e^{n(H_1 - H_0)} &1 & 0\\
         0& 0& 1 \\  
    \end{pmatrix}, & {\rm between}~ \Delta^- ~{\rm and}~ \Delta ,\\
    \begin{pmatrix}
      1& 0 &0\\
        0 &1 & 0\\
         0& -(c_\alpha\rho)^{-1}e^{n(H_2 - H_1)}& 1 \\  
    \end{pmatrix}, & {\rm between}~ \Gamma^* ~{\rm and}~ \Gamma^{*+}, \\
    \begin{pmatrix}
      1& 0 &0\\
        0 &1 & 0\\
         0& (c_\alpha\rho)^{-1}e^{n(H_2 - H_1)}& 1 \\  
    \end{pmatrix}, & {\rm between}~ \Gamma^* ~{\rm and}~ \Gamma^{*-}, \\
    \hspace{2cm}I, & {\rm elsewhere}.
\end{cases}
 \end{align}
Note that this transformation also does not affect the behavior at infinity. The transformation induces following jumps $J_T$ on $[0,1]$, $[x^*,0]$, $\Delta^{\pm}$ and $\Gamma^{^*\pm}$.
 \begin{align}
&J_T= \begin{pmatrix}
      0& w_1 &0\\
        -\frac{1}{w_1} &0 & 0\\
         0& 0& 1 \\  
    \end{pmatrix} {\rm on}~\Delta, \quad J_T= \begin{pmatrix}
      1& 0 &0\\
        \frac{1}{w_1}e^{n(H_1 - H_0)} &1 & 0\\
         0& 0& 1 \\  
    \end{pmatrix} {\rm on}~\Delta^+ ~{\rm and}~ \Delta^- ,\label{First Jump for T}\\
&J_T= \begin{pmatrix}
      1& 0 &0\\
        0 &0 & c_\alpha\rho\\
         0&-\frac{1}{c_\alpha\rho} & 0 \\  
    \end{pmatrix} {\rm on}~\Gamma^*, \quad J_T= \begin{pmatrix}
      1& 0 &0\\
        0 &1 & 0\\
         0& \frac{1}{c_\alpha\rho}e^{n(H_2 - H_1)}& 1 \\  
    \end{pmatrix} {\rm on}~\Gamma^{*+} ~{\rm and}~ \Gamma^{*-}. \label{Second Jump for T}
 \end{align}
 The behavior at end points $0$ and $1$ also changes. The RH problem for $T$ can be written as
 \begin{itemize}
\item $T$ is analytic in $\mathbb{C} \backslash (\Delta \cup \Gamma^* \cup \Gamma^{*\pm} \cup \Delta^\pm)$.
\item $T$ has jumps $T_+ = T_- J_T$, where $J_T$ is given by \eqref{First Jump for T} and \eqref{Second Jump for T}.
\item $T(z) =\mathbb{I}+O(1/z)$, $\quad z \rightarrow \infty$.
\item The behavior of $T$ near the end points $0$ and $1$ is as follows:
\begin{align*}
&T(z)= O\begin{pmatrix}
t_1(z) & t_2(z) & t_2(z) \\
t_1(z) & t_2(z) & t_2(z) \\
t_1(z) & t_2(z) & t_2(z)
\end{pmatrix} ~ z \rightarrow 1,
\end{align*}
where
\begin{align*}
t_1(z),t_2(z) = \begin{cases}
  1,|z-1|^{\beta},  & \beta < 0\\
   \log|z-1|,\log|z-1|, & \beta = 0\\
    |z-1|^{-\beta},1, &\beta > 0, z~{\rm inside~ the ~lens}\\
     1,1 & \beta > 0, z~{\rm outside~ the ~lens}.
\end{cases} 
\end{align*}
If $\alpha_1 < 0$ and $z \rightarrow 0$ from the positive real axis 
\begin{align*}
&T(z)= O\begin{pmatrix}
1 & |z|^{\alpha_1} & w(z) \\
1 & |z|^{\alpha_1} & w(z) \\
1 & |z|^{\alpha_1} & w(z)
\end{pmatrix} , ~{\rm where}~ w(z) = \begin{cases}
  |z|^{\alpha_2},  & \alpha_2< 0\\
   \log|z|, & \alpha_2= 0\\
    1, &\alpha_2 > 0.
\end{cases} 
\end{align*}
If $\alpha_1 < 0$ and $z \rightarrow 0$ from the negative real axis
\begin{align*}
&T(z)= O\begin{pmatrix}
1 & w_1(z) & w(z) \\
1 & w_1(z) & w(z) \\
1 & w_1(z) & w(z)
\end{pmatrix} , ~{\rm where}~ w_1(z),w(z) = \begin{cases}
  |z|^{\alpha_1},|z|^{\alpha_2},  & \alpha_2< 0\\
   |z|^{\alpha_1},\log|z|, & \alpha_2= 0\\
    |z|^{\alpha_1-\alpha_2}, 1, &\alpha_2 > 0.
\end{cases}
\end{align*}
If $\alpha_1 = 0$ and $z \rightarrow 0$ from positive real axis
\begin{align*}
&T(z)= O\begin{pmatrix}
\log|z| & \log|z| & w(z) \\
\log|z| & \log|z| & w(z) \\
\log|z| & \log|z| & w(z)
\end{pmatrix} ~ z \rightarrow 0, ~{\rm where}~ w(z) = \begin{cases}
  |z|^{\alpha_2},  & \alpha_2< 0\\
    1, &\alpha_2 > 0.
\end{cases} 
\end{align*}
If $\alpha_1 = 0$ and $z \rightarrow 0$ from negative real axis
\begin{align*}
&T(z)= O\begin{pmatrix}
1 & w_2(z) & w(z) \\
1 & w_2(z) & w(z) \\
1 & w_2(z) & w(z)
\end{pmatrix} ~ z \rightarrow 0, ~{\rm where}~ w_2(z),w(z) = \begin{cases}
  \log|z|,|z|^{\alpha_2},  & \alpha_2< 0\\
    |z|^{-\alpha_2}, 1, &\alpha_2 > 0.
\end{cases} 
\end{align*}
If $\alpha_1 > 0$ and $z \rightarrow 0$ from positive real axis
\begin{align*}
&T(z)= O\begin{pmatrix}
|z|^{-\alpha_1} & 1 & w(z) \\
|z|^{-\alpha_1} & 1 & w(z) \\
|z|^{-\alpha_1} & 1 & w(z)
\end{pmatrix} ~ z \rightarrow 0, ~{\rm where}~ w(z) = \begin{cases}
  |z|^{\alpha_2-\alpha_1},  & \alpha_2< 0~or~ 0<\alpha_2 < \alpha_1\\
  |z|^{-\alpha_1},  & \alpha_2= 0\\
    1, &\alpha_2 > 0, \alpha_2 > \alpha_1. 
\end{cases} 
\end{align*}
If $\alpha_1 > 0$ and $z \rightarrow 0$ from negative real axis
\begin{align*}
&T(z)= O\begin{pmatrix}
1 & w_3(z) & w(z) \\
1 & w_3(z) & w(z) \\
1 & w_3(z) & w(z)
\end{pmatrix} ~ z \rightarrow 0, ~{\rm where}~ w_3(z) = \begin{cases}
  1,  & \alpha_2< \alpha_1\\
    |z|^{\alpha_1-\alpha_2}, &\alpha_2> \alpha_1.
\end{cases} \\
&w(z) = \begin{cases}
  |z|^{\alpha_2-\alpha_1},  & \alpha_2< 0~or~ 0<\alpha_2 < \alpha_1\\
  |z|^{-\alpha_1},  & \alpha_2= 0\\
    1, &\alpha_2 > 0, \alpha_2 > \alpha_1 .
\end{cases}
\end{align*}
\end{itemize}
\subsection{The Riemann Hilbert Problem for $N$: Outer Parametrix} This parametrix gives approximation away from the end points, $x^*$, $0$ and $1$. It should satisfy the following RH problem:
\begin{itemize}
\item $N$ is analytic in $\mathbb{C} \backslash (\Delta \cup \Gamma^*)$.
\item $N$ has jumps given by 
\begin{align}
&N_+ = N_- \begin{pmatrix}
      0& w_1 &0\\
        -\frac{1}{w_1} &0 & 0\\
         0& 0& 1 \\  
    \end{pmatrix}, \quad z \in \Delta, \label{1st jump matrix for N}\\
&N_+ = N_- \begin{pmatrix}
      1& 0 &0\\
        0 &0 & c_\alpha\rho\\
         0& -\frac{1}{c_\alpha\rho}& 0 \\  
    \end{pmatrix}, \quad z \in \Gamma^*. \label{2nd jump matrix for N}
\end{align}
\item $N(z) =\mathbb{I}+O(1/z)$, $\quad z \rightarrow \infty$. 
\end{itemize}
\begin{theorem}
Let $\theta_1 =\frac{2-\theta}{4-\theta}$, $\theta_2 = \frac{\theta}{\theta+2}$ and $\sqrt{\mathcal{D}(s)} = \sqrt{s^2(s-1)(s-3\theta_1\theta_2)}$ being positive for large $s$. If 
\begin{align}
&\mathcal{F}_0(s)= \bigg(\frac{4}{(2+\theta)(4-\theta)}\bigg)^{\alpha_1+\beta}(s-\theta_1)(s-\theta_2) \frac{\mathcal{G}(s)}{\sqrt{\mathcal{D}(s)}}, \label{Function F1 general case}\\
&\mathcal{F}_1(s)=\dfrac{(s-\theta_2)\sqrt{\mathcal{D}(\theta_1)}}{(\theta_1-\theta_2)^{\alpha_2-\alpha_1+1} \theta_1^{2\alpha_1-\alpha_2}(1-\theta_1)^{\beta}}  \frac{\mathcal{G}(s)}{\sqrt{\mathcal{D}(s)}},\label{Function F2 general case}\\
&\mathcal{F}_2(s)=\dfrac{(\theta_2-\theta_1)^{\alpha_2-\alpha_1} \sqrt{\mathcal{D}(\theta_2)} [(2+\theta)(4-\theta)]^{\alpha_2-\alpha_1} }{\theta_2^{-\alpha_1+2\alpha_2+1}e^{(\alpha_2-\alpha_1)i\pi}4^{\alpha_2-\alpha_1}(1-\theta_2)^{\beta} } (s-\theta_1) \frac{\mathcal{G}(s)}{\sqrt{\mathcal{D}(s)}}, \label{Function F3 general case}
\end{align}
is defined and analytic in $\mathbb{C} \backslash (\varphi_1^-(\Delta) \cup \varphi_1^+(\Gamma)$. Further, let the function $\mathcal{G}(s)$ is given by
\begin{align}\label{Function G general case}
\mathcal{G}(s) = \begin{cases}
  \dfrac{(s-\theta_2)^{\alpha_2-\alpha_1}s^{2\alpha_1-\alpha_2}(1-s)^{\beta}}{z^{\alpha_1}(1-z)^\beta},  & s \in \varphi(\mathcal{R}_0),\\
   (s-\theta_2)^{\alpha_2-\alpha_1} s^{2\alpha_1-\alpha_2}(1-s)^{\beta},& s \in \varphi(\mathcal{R}_1),\\
   \dfrac{e^{(\alpha_2-\alpha_1)i\pi}4^{\alpha_2-\alpha_1}  s^{-\alpha_1+2\alpha_2}(1-s)^{\beta}}{[(2+\theta)(4-\theta)]^{\alpha_2-\alpha_1}(s-\theta_1)^{\alpha_2-\alpha_1} },& s \in \varphi(\mathcal{R}_2),
\end{cases}
\end{align}
then, a solution to the Riemann-Hilbert problem for N is given by
\begin{align}\label{Solution matrix N general case}
N(z) = \begin{pmatrix}
    \mathcal{F}_0(\varphi_0(z)) & \mathcal{F}_0(\varphi_1(z)) & \mathcal{F}_0(\varphi_2(z))\\
    \mathcal{F}_1(\varphi_0(z)) & \mathcal{F}_1(\varphi_1(z)) & \mathcal{F}_1(\varphi_2(z))\\
    \mathcal{F}_2(\varphi_0(z)) & \mathcal{F}_2(\varphi_1(z)) & \mathcal{F}_2(\varphi_2(z))
\end{pmatrix}.
\end{align}
\end{theorem}


\begin{proof}
Considering the rows of \eqref{1st jump matrix for N} and \eqref{2nd jump matrix for N}, we get
\begin{align}
&N_{k0}^+ = -w_1^{-1} N_{k1}^-, \nonumber\\
&  N_{k1}^+ = w_1 N_{k0}^-,  \label{N on delta}\\
&  N_{k2}^+ = N_{k2}^- \nonumber
\end{align}
on $\Delta$ and 
\begin{align}
&N_{k0}^+ =  N_{k0}^-, \nonumber\\
&  N_{k1}^+ = -(c_\alpha\rho)^{-1} N_{k2}^-,  \label{N on gamma}\\
&  N_{k2}^+ =  c_\alpha\rho N_{k1}^- \nonumber
\end{align}
on $\Gamma^*$ for $k=0,1,2$. Clearly, $N_{k2}$ and $N_{k0}$ are analytic on $\Delta$ and $\Gamma^*$, respectively. Thus, $N_{k0}$, $N_{k1}$ and $N_{k2}$ can be seen as functions on $\mathcal{R}_0$, $\mathcal{R}_1$ and $\mathcal{R}_2$, respectively. We then use the mapping $\varphi: \mathcal{R} \rightarrow \mathbb{C}$ to transform the problem from the Riemann surface to the $s$-plane. For $\theta < 1$, define $\mathcal{F}_k$ as 
\begin{align}\label{Function F}
\mathcal{F}_k(s) = \begin{cases}
N_{k0} \bigg[\dfrac{4s^3}{((2+\theta)s-\theta)((4-\theta)s-2+\theta)}\bigg],& s \in \varphi(\mathcal{R}_0),\\
N_{k1} \bigg[\dfrac{4s^3}{((2+\theta)s-\theta)((4-\theta)s-2+\theta)}\bigg],& s \in \varphi(\mathcal{R}_1),\\
N_{k2} \bigg[\dfrac{4s^3}{((2+\theta)s-\theta)((4-\theta)s-2+\theta)}\bigg],& s \in \varphi(\mathcal{R}_2).
\end{cases}
\end{align}
Then, $\mathcal{F}_k$ is analytic in $\mathbb{C} \backslash (\varphi_1^-(\Delta) \cup \varphi_1^+(\Gamma))$ and the jumps of $\mathcal{F}_k$ can be determined using \eqref{N on delta} and \eqref{N on gamma}. They are given as
\begin{align}\label{Jumps of F}
&\begin{cases}
  \mathcal{F}^+_k(s)= -w_1^{-1} \mathcal{F}^-_k(s),\quad &s \in \varphi_1^-(\Delta),\\
\mathcal{F}^+_k(s)= w_1 \mathcal{F}^-_k(s),\quad &s \in \varphi_1^+(\Delta), \\
\mathcal{F}^+_k(s)= -c_\alpha\rho \mathcal{F}^-_k(s),\quad &s \in \varphi_1^+(\Gamma), \\
\mathcal{F}^+_k(s)= (c_\alpha\rho)^{-1} \mathcal{F}^-_k(s),\quad &s \in \varphi_1^-(\Gamma).  
\end{cases}
\end{align}

The behavior of N as $z \rightarrow \infty$ imposes the following asymptotic conditions for $\mathcal{F}_0$:
\begin{align}\label{Normalization F}
\mathcal{F}_0(\infty) = 1, \quad \mathcal{F}_0(\theta_1)=0, \quad \mathcal{F}_0(\theta_2) = 0.
\end{align}
This enforces $\mathcal{F}_1$ to have the following form \eqref{Function F1 general case}:
\begin{align*}
\mathcal{F}_0(s)= (s-\theta_1)(s-\theta_2) \frac{\mathcal{G}(s)}{\sqrt{\mathcal{D}(s)}} , 
\end{align*}
where $\mathcal{D}(s) = s^2(s-1)(s-3\theta_1\theta_2)$ and $\mathcal{G}$ analytic in $\mathbb{C} \backslash (\varphi_1^-(\Delta) \cup \varphi_1^+(\Gamma))$ with jumps:
\begin{align}\label{Jumps of G}
&\begin{cases}
  \mathcal{G}^+(s)= w_1^{-1} \mathcal{G}^-(s),\quad &s \in \varphi_1^-(\Delta),\\
\mathcal{G}^+(s)= w_1 \mathcal{G}^-(s),\quad &s \in \varphi_1^+(\Delta), \\
\mathcal{G}^+(s)= (c_\alpha\rho)^{-1} \mathcal{G}^-(s),\quad &s \in \varphi_1^-(\Gamma), \\
\mathcal{G}^+(s)= c_\alpha\rho \mathcal{G}^-(s),\quad &s \in \varphi_1^+(\Gamma).  
\end{cases}
\end{align}
The normalization for $\mathcal{G}$ is $\mathcal{G}(\infty) = 1$. Then, it is easy to check that $\mathcal{G}$ given by \eqref{Function G general case} satisfies the jump conditions \eqref{Jumps of G}. From this, it is straightforward to verify the jump conditions \eqref{Jumps of F} and normalization \eqref{Normalization F} for $\mathcal{F}_1$. Using \eqref{Function F}, we recover the following representation for $N_{00}$, $N_{01}$, and $N_{02}$:
\begin{align*}
N_{00}(z) = \mathcal{F}_0(\varphi_0(z)), \quad N_{01}(z) = \mathcal{F}_0(\varphi_1(z)), \quad N_{02}(z) = \mathcal{F}_0(\varphi_2(z)),
\end{align*}
where $\varphi_0(z)$, $\varphi_1(z)$, and $\varphi_2(z)$ are the three solutions of \eqref{Mapping R(s)}. Thus, the jumps \eqref{N on delta} and \eqref{N on gamma} are satisfied, and the asymptotic conditions for $N$ are correct. Thus, we have obtained the first row of \eqref{Solution matrix N general case}.

To obtain the second and third row of \eqref{Solution matrix N general case}, we use the following normalization at infinity:
\begin{align*}
\mathcal{F}_1(\infty) = 0, \quad \mathcal{F}_1(\theta_1)=1, \quad \mathcal{F}_1(\theta_2) = 0  ,
\end{align*}
and 
\begin{align*}
\mathcal{F}_2(\infty) = 0, \quad \mathcal{F}_2 (\theta_1)=0, \quad \mathcal{F}_2(\theta_2) = 1 .
\end{align*}
This leads to the construction of two new functions $\mathcal{F}_1$ and $\mathcal{F}_2$ satisfying the same jump conditions as \eqref{Jumps of F}. These functions, which can also be expressed in terms of $\mathcal{G}$, are given by \eqref{Function F2 general case} and \eqref{Function F3 general case}. 
\end{proof}
We have the following behavior as $z \rightarrow 0$ (since $s \sim  z^{\frac{1}{3}}$ in view of \eqref{Mapping R(s)}) :
\begin{align}\label{Nz behavior at z=0}
&\mathcal{F}_j(s) [\mathcal{G}(s)]^{-1} = O(1/s)= O(z^{-\frac{1}{3}}) , \quad j=0,1,2 \notag\\
\Rightarrow & \quad N(z) = O (z^{-\frac{1}{3}} z^\frac{A}{3}),  
\end{align}
where $A = \diag[A_1,A_2,A_3]= \diag [-\alpha_1-\alpha_2,2\alpha_1-\alpha_2,-\alpha_1+2\alpha_2]$.

To carry out the local analysis at $0$, we have assumed $\beta=0$, i.e., $w_1 = z^{\alpha_1}$. Hence, a more detailed representation of $N(z)$ as $z \rightarrow 0$ is given by the following proposition.
\begin{proposition}\label{Prop: N hat}
The matrix
\begin{align}\label{N hat def}
\widehat{N(z)} =   N(z)   z^{-\frac{A}{3}} [\mathcal{U}^\pm]^{-1} z^{\frac{M}{3}},
\end{align}
where
\begin{align}
& M= \diag [1,0,-1],\quad \mathcal{W}= \diag[\omega^{-A_1}, \omega^{A_1}, 1],\quad \mathcal{U}^+=\widehat{\mathcal{U}} \mathcal{W}, \quad \mathcal{U}^- = \mathcal{U}^+ J, \quad \omega = e^{\frac{i\pi}{3}}, \notag\\
&\widehat{\mathcal{U}}  = \begin{pmatrix}
      \omega^{-2}& \omega^{2} & 1 \\
        -1 & -1 & -1\\
        \omega^{2} & \omega^{-2} & 1 \\  
    \end{pmatrix}, \quad [\widehat{\mathcal{U}}]^{-1} =\frac{1}{3}\begin{pmatrix}
      \omega^{2}& -1 & \omega^{-2} \\
        \omega^{-2} & -1 & \omega^{2}\\
        1 & -1 & 1 \\  
    \end{pmatrix}  , \quad J = \begin{pmatrix}
      0& -1 & 0 \\
        1 & 0 & 0\\
        0 & 0 & 1 \\  
    \end{pmatrix}, \label{U plus and U minus} 
\end{align}
is analytic in the neighbourhood of $z=0$ and has an analytic inverse as well. Further, $z=0$ is a removable singularity of $\widehat{N(z)}$.
\end{proposition}
\begin{proof}
It is clear that we need to check the analyticity across $\Delta$ and $\Gamma^*$ in the neighbourhood of $z=0$. For $x \in \Delta$, it follows from \eqref{1st jump matrix for N} that
\begin{align*}
&[\widehat{N_-(x)}]^{-1} \widehat{N_+(x)} = x^{-\frac{M}{3}} \mathcal{U}^-  x^\frac{A}{3} J_N(x) x^{-\frac{A}{3}} \mathcal{(U^+)}^{-1} x^{\frac{M}{3}}\\
&=x^{-\frac{M}{3}} \mathcal{U}^- \begin{pmatrix}
      0& 1 & 0 \\
        -1 & 0 & 0\\
        0 & 0 & 1 \\  
    \end{pmatrix} \mathcal{(U^+)}^{-1} x^{\frac{M}{3}}\\
&=x^{-\frac{M}{3}} \mathcal{U}^- J^{-1} \mathcal{(U^+)}^{-1} x^{\frac{M}{3}}=x^{-\frac{M}{3}} \mathcal{U}^+ \mathcal{(U^+)}^{-1} x^{\frac{M}{3}} = I.
\end{align*}
Similarly, when $x \in \Gamma^*$, it follows from \eqref{2nd jump matrix for N} that 
\begin{align*}
&[\widehat{N_-(x)}]^{-1} \widehat{N_+(x)} = x_-^{-\frac{M}{3}} \mathcal{U}^- (x) x_-^\frac{A}{3} J_N(x) x_+^{-\frac{A}{3}} \mathcal{(U^+)}^{-1}x_+^{\frac{M}{3}}\\
&=x_-^{-\frac{M}{3}} \mathcal{U}^-  \begin{pmatrix}
      \omega^{-2A_1}&0 & 0 \\
        0 & 0 & \omega^{A_1}\\
        0 & -\omega^{A_1} & 0 \\  
    \end{pmatrix}  \mathcal{(U^+)}^{-1} x_+^{\frac{M}{3}}\\
&=x_-^{-\frac{M}{3}}\widehat{\mathcal{U}}  \mathcal{W}\begin{pmatrix}
      0& -1 & 0 \\
        1 & 0 & 0\\
        0 & 0 & 1 \\  
    \end{pmatrix} \begin{pmatrix}
      \omega^{-2A_1}&0 & 0 \\
        0 & 0 & -\omega^{A_1}\\
        0 & \omega^{A_1} & 0 \\  
    \end{pmatrix}  \mathcal{W}^{-1}\widehat{\mathcal{U}}^{-1} x_+^{\frac{M}{3}}\\
&=x_-^{-\frac{M}{3}} \widehat{\mathcal{U}}  \begin{pmatrix}
      0&0 & 1\\
        1 & 0 & 0\\
        0 & 1 & 0 \\  
    \end{pmatrix} \widehat{\mathcal{U}}^{-1} x_+^{\frac{M}{3}}= x_-^{-\frac{M}{3}} \diag[\omega^{-2},1, \omega^{2} ]  x_+^{\frac{M}{3}} = I.
\end{align*}
Thus, we conclude that $\widehat{N(z)}$ is analytic in a neighborhood of the origin. Moreover, $z=0$ is an isolated singularity. Further, it can be inferred immediately from \eqref{Nz behavior at z=0} and \eqref{N hat def} that singularity in $\widehat{N(z)}$ can be, atmost, of the order $z^{-2/3}$, i.e,
\begin{align*}
\widehat{N(z)} = O(z^{-2/3}), \quad z \to 0,
\end{align*}
which shows that this singularity is indeed removable. At last, the existence of inverse of $\widehat{N(z)}$ follow from the fact that the determinants of $N$, $\mathcal{U}^\pm$, $z^{\frac{M}{3}}$, and $z^{-\frac{A}{3}}$ are all constant and non zero in the neighborhood of $z=0$.
\end{proof}



\subsection{Local parametrix at the edge $0$}\label{subsec:Local parametrix near 0}
We consider a disk $D(0,r_0)$ around $z=0$ of radius $r_0$. The {\it initial local parametrix problem} is as follows: We look for a matrix-valued function $\mathring{P}$ that satisfies the following RH problem:
\begin{itemize}
\item $\mathring{P}$ is analytic in $D(0,r_0) \backslash (\Delta \cup \Gamma^* \cup \Gamma^{*\pm} \cup \Delta^\pm)$.
\item $\mathring{P}$ satisfies the following jump conditions:
\begin{align*}
\mathring{P}_+ = \mathring{P}_-J_{\mathring{P}}, \quad z \in (\Delta \cup \Gamma^* \cup \Gamma^{*\pm} \cup \Delta^\pm) \cap D(0,r_0),
\end{align*}
where
\begin{align*}
&J_{\mathring{P}} = \begin{pmatrix}
      0& w_1 &0\\
        -\frac{1}{w_1} &0 & 0\\
         0& 0& 1 \\  
    \end{pmatrix}, \quad z \in \Delta \cap D(0,r_0), \\
&J_{\mathring{P}} = \begin{pmatrix}
      1& 0 &0\\
        \frac{1}{w_1}e^{n(H_1 - H_0)} &1 & 0\\
         0& 0& 1 \\  
    \end{pmatrix}, \quad z \in \Delta^\pm \cap D(0,r_0), \\
&J_{\mathring{P}} = \begin{pmatrix}
      1& 0 &0\\
        0 &0 & c_\alpha\rho\\
         0&-\frac{1}{c_\alpha\rho} & 0 \\  
    \end{pmatrix}, \quad z \in \Gamma^* \cap D(0,r_0), \\
&J_{\mathring{P}} = \begin{pmatrix}
      1& 0 &0\\
        0 &1 & 0\\
         0& \frac{1}{c_\alpha\rho}e^{n(H_2 - H_1)}& 1 \\  
    \end{pmatrix}, \quad z \in \Gamma^{*\pm} \cap D(0,r_0).
\end{align*}
\item $\mathring{P}(z)$ behaves similar to $T(z)$ near the origin.
\end{itemize}
The matching condition is usually given as: On $\partial D(0,r_0)$ we have that as $n \rightarrow \infty$, $\mathring{P}$ matches N in the sense
\begin{align*}
\mathring{P}(z) = (\mathbb{I}+O (1/n)) N(z), \quad {\rm uniformly ~for}  z \in \partial D(0,r_0).
\end{align*}
However, in larger size RHPs obtaining such a matching condition is a major technical issue, and ours is not exception. We will therefore use double matching instead of ordinary matching. Then, there is also a jump on a shrinking circle inside $D(0,r_0)$. In our case, this circle is $\partial D(0,r_n)$ where
\begin{align*}
    r_n = n^{-3/2}, \quad n=1,2,\ldots
\end{align*}
On the other hand, in the annulus $r_n<|z|<r_0$, denoted by $A(0,r_n,r_0)$, the local parametrix will have no jump on the lips of the lens, i.e., the actual local parametrix $P$ that we will use in the final transformation, satisfies an altered version of RHP for $\mathring{P}$. Indeed this is why we called the local parametrix problem for $\mathring{P}$ the initial local parametrix problem. We will first find out a solution $\mathring{P}$ to the RHP for $\mathring{P}$ and then, using the double matching approach, we will construct $P$ as 
\begin{align}\label{definition P(z)}
P(z)= \begin{cases}
    E_n^0(z) \mathring{P}(z), & z \in D(0,r_n),\\
   E_n^\infty(z) N(z), & z \in A(0,r_n,r_0).
\end{cases}
\end{align}
where $E_n^0(z)$ and $E_n^\infty(z)$ are analytic prefactors. Then, it will turn out that $P$ satisfies a double matching of the form
\begin{align}
  &  P(z)N(z)^{-1} = \mathbb{I}+O\left(\frac{1}{n} \right),  \quad uniformly~for ~ z \in \partial D(0,r_0), \label{double matching of P 1}\\
  &P_+(z) P_-(z)^{-1} = \mathbb{I}+O\left(\frac{1}{n} \right),  \quad uniformly~for ~ z \in \partial D(0,r_n), \label{double matching of P 2}
\end{align}
as $n \to \infty$ .

\subsubsection{A local parametrix problem}\label{subsec: making n-independent}
The first step to solving a local parametrix problem is to transform it to a problem that does not depend on $n$. To that end, we define
\begin{align}\label{eq: def Dn z}
    D_n(z) = \begin{pmatrix}
      e^{-nH_0}  & 0 & 0 \\
      0  & e^{-nH_1} & 0 \\
      0  & 0 & e^{-nH_2} 
    \end{pmatrix}, \quad z \in D(0,r_0).
\end{align}

\begin{proposition}
Suppose that
\begin{align}\label{eq: hat P equal Dn}
\widehat{\mathbf{P}(z)} = \mathring{P}(z)[D_n(z)]^{-1}.
\end{align}
Then $\mathring{P}$ satisfies the RHP for $\mathring{P}$ iff $\widehat{\mathbf{P}(z)}$ satisfies RHP for $\widehat{\mathbf{P}}$ as defined below.
\end{proposition}
\begin{itemize}
\item $\widehat{\mathbf{P}}$ is analytic on $D(0,r_0) \backslash (\Delta \cup \Gamma^* \cup \Gamma^{*\pm} \cup \Delta^\pm)$.
\item The jumps for $\widehat{\mathbf{P}}$ are $\widehat{\mathbf{P}}_+ =\widehat{\mathbf{P}}_- J_{\widehat{\mathbf{P}}}$ where
\begin{align}\label{eq: jump condition P hat}
&J_{\widehat{\mathbf{P}}} =  \begin{cases}
    \begin{pmatrix}
      0& w_1 &0\\
        -\frac{1}{w_1} &0 & 0\\
         0& 0& 1 \\  
    \end{pmatrix}, \quad z \in \Delta \cap D(0,r_0), \\
    \begin{pmatrix}
      1& 0 &0\\
        \frac{1}{w_1} &1 & 0\\
         0& 0& 1 \\  
    \end{pmatrix}, \quad z \in \Delta^\pm \cap D(0,r_0), \\
    \begin{pmatrix}
      1& 0 &0\\
        0 &0 & c_\alpha\rho\\
         0&-\frac{1}{c_\alpha\rho} & 0 \\  
    \end{pmatrix}, \quad z \in \Gamma^* \cap D(0,r_0), \\
    \begin{pmatrix}
      1& 0 &0\\
        0 &1 & 0\\
         0& \frac{1}{c_\alpha\rho}& 1 \\  
    \end{pmatrix}, \quad z \in \Gamma^{*\pm} \cap D(0,r_0).
\end{cases}
\end{align}
\item Near the origin, $\widehat{\mathbf{P}}$ has the similar behavior as in RHP for $T(z)$.
\end{itemize}
We will eventually find a solution to RHP for $\widehat{\mathbf{P}}$ in the form $\widehat{\mathbf{P}(z)} = \Psi(n^3 f(z))$. Here, $\Psi$ is a solution to so called bare-parametrix problem that we define explicitly in section \ref{Bare Parametrix}. This problem has the same behavior as in RHP for $\widehat{\mathbf{P}}$, except that the jump contours are extended to infinity, i.e., we have jumps on the positive and negative real axis and on the rays inclined at $\pm \pi/4$ and $\pm 3\pi/4$ (see Figure \ref{Fig: 6 rays}). We denote these by $\mathbb{R}_+$, $-\mathbb{R}_+$, $e^{\pm \frac{i\pi}{4}} \mathbb{R}_+$, and $e^{\pm \frac{3i\pi}{4}} \mathbb{R}_+$. The function $f(z)$ is a conformal map with $f(0)=0$. We are allowed to slightly deform the lips of the lens so as to match the aforementioned six rays.

\begin{center}
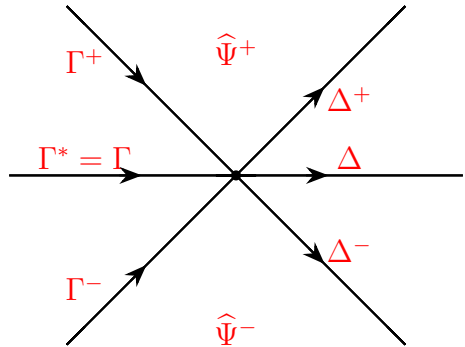
\begin{figure}[htbp!]
\begin{tikzpicture}[thick]
\filldraw (1.5,1.5)node[below=5pt,red]{$\Delta^+$};
\filldraw (1.5,-0.5)node[below=5pt,red]{$\Delta^-$};
\filldraw (-2,2)node[below=5pt,red]{$\Gamma^+$};
\filldraw (-2,-1)node[below=5pt,red]{$\Gamma^-$};
\filldraw (-2,0.7)node[below=5pt,red]{$\Gamma^*=\Gamma$};
\filldraw (1.5,0.7)node[below=5pt,red]{$\Delta$};
\filldraw (0,2.2)node[below=5pt,red]{$\widehat{\Psi}^{+}$};
\filldraw (0,-1.5)node[below=5pt,red]{$\widehat{\Psi}^{-}$};

\foreach \angle in {0,45,135,180,225,315} {
    \draw (0,0) -- (\angle:3);  
}

\filldraw (0,0) circle (1.5pt);
\draw[arc arrow=to pos .5 with length 3mm] (-2,2) to[out=135,in=135] (0,0);
\draw[arc arrow=to pos .5 with length 3mm] (-2,-2) to[out=225,in=225] (0,0);
\draw[arc arrow=to pos .5 with length 3mm] (-3,0) to[out=0,in=0] (0,0);
\draw[arc arrow=to pos .5 with length 3mm] (0,0) to[out=45,in=45] (2,2);
\draw[arc arrow=to pos .5 with length 3mm] (0,0) to[out=315,in=315] (2,-2);
\draw[arc arrow=to pos .5 with length 3mm] (0,0) to[out=180,in=180] (3,0);
\end{tikzpicture}
\caption{The jump contours for the model RH problem for $\Psi$}
\label{Fig: 6 rays}
\end{figure}
\end{center}

\subsubsection{Bare local parametrix problem}\label{Bare Parametrix}
The bare parametrix problem has the following form:
\begin{itemize}
    \item $\Psi$ is analytic on $\mathbb{C} \backslash \{\mathbb{R} \cup e^{\pm \frac{i\pi}{4}} \mathbb{R}_+ \cup e^{\pm \frac{3i\pi}{4}} \mathbb{R}_+ \}$.
    \item $\Psi$ has the following jump behavior: $\Psi_+ = \Psi_- J_\Psi$ where
 \begin{align}\label{eq: Jump condition for psi}
 &J_\Psi= \begin{cases}
     \begin{pmatrix}
      0& w_1 &0\\
        -\frac{1}{w_1} &0 & 0\\
         0& 0& 1 \\  
    \end{pmatrix}, \quad \xi \in \mathbb{R}_+, \\
    \begin{pmatrix}
      1& 0 &0\\
        0 &0 & c_\alpha\rho\\
         0& -\frac{1}{c_\alpha\rho}& 0 \\  
    \end{pmatrix}, \quad \xi \in -\mathbb{R}_+, \\
    \begin{pmatrix}
      1& 0 &0\\
        \frac{1}{w_1} &1 & 0\\
         0& 0& 1 \\  
    \end{pmatrix}, \quad \xi \in e^{\pm \frac{i\pi}{4}} \mathbb{R}_+, \\
    \begin{pmatrix}
      1& 0 &0\\
        0 &1 & 0\\
         0& \frac{1}{c_\alpha\rho}& 1 \\  
    \end{pmatrix}, \quad \xi \in e^{\pm \frac{3i\pi}{4}} \mathbb{R}_+.
 \end{cases}
\end{align}   
    \item $\Psi$ has a specific asymptotic behavior at $\infty$.
    \item  $\Psi$ has similar behavior as $T$ at the origin.
\end{itemize}
Next, we construct the solution. The RHP for $\Psi$ shows great similarity with the bare Meijer-G parametrix for $p$-chain \cite[p. 47]{Bertola_Bothner_2015}. Its solution is constructed using Meijer-G functions and ours will also be constructed using Meijer-G functions as we will see in the theorem \ref{G function theorem}. See \cite{Beals_Szmigielski_2013} for an introduction on Meijer-G functions. They can be defined through the Mellin-Barnes integral formula
\begin{align*}
G^{m,n}_{p,q}\left(\begin{array}{c} a_1, a_2, \ldots, a_p \\ b_1, b_2, \ldots, b_q \end{array} \middle| z\right) = \frac{1}{2\pi i} \int_{\mathcal{L}} \frac{\prod\limits_{j=1}^m \Gamma(b_j + s) \prod\limits_{j=1}^n \Gamma(1 - a_j - s)}{\prod\limits_{j=m+1}^q \Gamma(1 - b_j - s) \prod\limits_{j=n+1}^p \Gamma(a_j + s)} z^{-s} \, ds,
\end{align*}
where $\Gamma$ denotes the gamma function and $\mathcal{L}$ is a suitable integration contour separating the poles of $\Gamma(b_j - s)$ from those of $\Gamma(1 - a_j + s)$.

\begin{theorem}\label{G function theorem}
Let 
\begin{align*}
    \widehat{\Psi}(\xi) = \begin{pmatrix}
    g_1(\xi) & g^{\pm}_2(\xi) & g_3(\xi)  \\
    z \partial_z g_1(\xi) & \big(z \partial_z -\alpha_1 \big)g^{\pm}_2(\xi) &\big(z \partial_z -\alpha_2 \big)g_3(\xi)  \\
    \big(z \partial_z\big)^2 g_1(\xi) &\big(z \partial_z -\alpha_1 \big)^2 g^{\pm}_2(\xi) & \big(z \partial_z -\alpha_2 \big)^2 g_3(\xi)
\end{pmatrix}, \quad \pm \im \xi > 0 ,
\end{align*}
with 
\begin{align*}
g_j^{\pm}(\xi) = \frac{c_j}{2\pi i} \int_L \frac{\prod_{\ell=1}^j \Gamma(s+a_{\ell,j-1})}{\prod_{\ell=j}^3 \Gamma(1+a_{j\ell}-s)} e^{\pm i\pi s \Sigma_j} \xi^{-s} ds, \quad c_j=(2\pi i)^{3-j}\sqrt{\frac{3}{(2\pi)^2}} \quad 1 \leq j\leq 3,
\end{align*}
where $\Sigma_j =(j+1)$ mod $2$.
With this, the solution $\Psi(\xi)$ to the RH problem for $\Psi$ is given by
\begin{align}\label{G 2 xi jump relations}
\Psi(\xi) = \begin{cases}
\widehat{\Psi}(\xi),   & \arg \xi \in (-\frac{3\pi}{4},-\frac{\pi}{4}) \cup (\frac{\pi}{4},\frac{3\pi}{4}),\\
\widehat{\Psi}(\xi) \left[\begin{pmatrix}
      1& 0 \\
        -\frac{1}{w_1} &1 
    \end{pmatrix}\oplus 1\right],& \arg \xi \in (0,\frac{\pi}{4}),\\
\widehat{\Psi}(\xi) \left[\begin{pmatrix}
      1& 0 \\
        \frac{1}{w_1} &1 
    \end{pmatrix}\oplus 1\right],& \arg \xi \in (-\frac{\pi}{4},0),\\
\widehat{\Psi}(\xi)\left[1\oplus\begin{pmatrix}
      1& 0 \\
        -\frac{1}{c_\alpha\rho} &1 
    \end{pmatrix}\right], & \arg \xi \in (\frac{3\pi}{4},\pi),\\
\widehat{\Psi}(\xi)\left[1\oplus \begin{pmatrix}
      1& 0 \\
        \frac{1}{c_\alpha\rho} &1 
    \end{pmatrix}\right], & \arg \xi \in (-\pi,-\frac{3\pi}{4}).
\end{cases}
\end{align}
Moreover, as $\xi \rightarrow \infty$ with $\epsilon > 0$ fixed,
\begin{align}\label{G 2 xi}
\Psi(\xi) = \xi^{-\frac{M}{3}} \mathcal{U}^\pm \tilde{K}(\xi)   \xi^\frac{A}{3} \exp \big({ -3\xi^{\frac{1}{3}} \Omega_\pm} \big), \quad \pm \im \xi > 0 ,
\end{align}
where $\mathcal{U}^\pm(z)$, $M$ and A have already appeared in Proposition \ref{Prop: N hat}, $K$ admits an expansion of the form
\begin{align*}
    \tilde{K}(\xi) = I+\sum_{j=1}^{\infty} \frac{\tilde{K}_j}{\xi^{\frac{j}{3}}}, 
\end{align*}
where $\tilde{K}_j$'s are the $3 \times 3$ matrix valued coefficients and
\begin{align}\label{Omega and its tilde}
\Omega_\pm = \diag [\omega^{\pm 2}, \omega^{\mp 2},1].
\end{align}
\end{theorem}
\begin{remark}
The functions $g_m^{\pm}(\xi)$, $m=1,2,3$, involved in the above construction are all Meijer G-functions and are defined as
\begin{align}
&g_1^{\pm}(\xi) = c_1  G_{0,3}^{1,0} \left(\begin{array}{c}
 - \\ 0,-\alpha_1,-\alpha_2
\end{array} \middle| \xi \right) = \frac{c_1}{2\pi i} \int_L \frac{\Gamma(s)}{\Gamma(1+\alpha_1-s)\Gamma(1+\alpha_2-s)} \xi^{-s} ds, \label{def g_1 xi}\\
&g_2^{\pm}(\xi) = c_2  G_{0,3}^{2,0} \left(\begin{array}{c}
 - \\ 0,\alpha_1,\alpha_1-\alpha_2
\end{array} \middle| e^{\mp i \pi }\xi \right)=\frac{c_2}{2\pi i} \int_L \frac{\Gamma(s)\Gamma(s+\alpha_1) }{\Gamma(1+\alpha_2-\alpha_1-s)} e^{\pm i \pi s}\xi^{-s} ds, \label{def g_2 xi}\\
&g_3^{\pm}(\xi) = c_3 G_{0,3}^{3,0} \left(\begin{array}{c}
 - \\ 0,\alpha_2,\alpha_2-\alpha_1
\end{array} \middle| \xi \right) =\frac{c_3}{2\pi i} \int_L \Gamma(s)\Gamma(s+\alpha_2)\Gamma(s+\alpha_2-\alpha_1) \xi^{-s} ds, \label{def g_3 xi}
\end{align}
where $c_1 = -2\pi \sqrt{3}$, $c_2 = i\sqrt{3}$, and $c_3 = \frac{\sqrt{3}}{2\pi}$.
\end{remark}

\begin{proof}
Recall that $w_1 = z^{\alpha_1}$ and $\rho= (-z)^{\alpha_2-\alpha_1} = z^{\alpha_2-\alpha_1}e^{\pm i\pi ({\alpha_2-\alpha_1})}$. For the jump relations \eqref{G 2 xi jump relations} to be true, the following monodromy relations should hold on the negative side:
\begin{align}
&g_1^{+}(\xi) = g_1^{-}(\xi), \label{negative side 1}\\
& g_3^{+}(\xi) =\rho g_2^{+}(\xi)+ g_3^{-}(\xi), \label{negative side 2}\\
& g_3^{+}(\xi) = \rho g_2^{-}(\xi)+ g_3^{-}(\xi), \label{negative side 3}
\end{align}
and the following should hold on the positive real axis:
\begin{align}
&g_2^{+}(\xi) = w_1g_1^{-}(\xi)+ g_2^{-}(\xi), \label{positive side 1}\\
& g_2^{+}(\xi) =w_1g_1^{+}(\xi)+ g_2^{-}(\xi), \label{positive side 2} \\
& g_3^{+}(\xi) = g_3^{-}(\xi). \label{positive side 3}
\end{align}
The relations \eqref{negative side 2} and \eqref{negative side 3} hold if $g_2^{+}(\xi) = g_2^{-}(\xi)$ and \eqref{negative side 3} hold on negative real axis and the relations \eqref{positive side 1} and \eqref{positive side 2} hold if $g_1^{+}(\xi) = g_1^{-}(\xi)$ and \eqref{positive side 2} hold on positive real axis.
Since the Bare parametrix problem is similar to the one considered in \cite{Bertola_Bothner_2015}, we can make use of the relations (4.82) for our purpose. We rewrite the relations for the purpose of completeness:
\begin{align*}
 g_j^{+}(\xi e^{2\pi i})  - g_j^{+}(\xi) = -\xi^{a_{j-1}} e^{i\pi a_{j-1}\Sigma_{j-1}}g_{j-1}^{+}(\xi e^{2\pi i \Sigma_{j-1}}), \quad j=2,3.
\end{align*}
In our case, $a_1 = \alpha_1$ and $a_2 = \alpha_2-\alpha_1$. Now, for $j=2$, we have
\begin{align*}
\Rightarrow& g_2^{+}(\xi e^{2\pi i})  - g_2^{+}(\xi) = -\xi^{a_{1}} e^{i\pi a_{1}\Sigma_{1}}g_{1}^{+}(\xi e^{2\pi i \Sigma_{1}}), \quad \Sigma_{1} = 2 ~mod ~2 =0\notag\\
\Rightarrow& g_2^{+}(\xi e^{2\pi i})  - g_2^{+}(\xi) = -\xi^{\alpha_1}  g_{1}^{+}(\xi), \notag \\
\Rightarrow& g_2^{+}(\xi e^{2\pi i})  - g_2^{+}(\xi) = -w_1g_{1}^{+}(\xi),
\end{align*}
which implies \eqref{positive side 2}.
For $j=3$, we have
\begin{align*}
\Rightarrow& g_3^{+}(\xi e^{2\pi i})  - g_3^{+}(\xi) = -\xi^{a_{2}}e^{i\pi a_{2}\Sigma_{2}}g_{2}^{+}(\xi e^{2\pi i \Sigma_{2}}), \quad \Sigma_{2} = 3 ~mod ~2 = ~1\notag\\
\Rightarrow& g_3^{+}(\xi e^{2\pi i})  - g_3^{+}(\xi) = -\xi^{\alpha_2-\alpha_1}e^{ i\pi (\alpha_2-\alpha_1))}g_{2}^{+}(\xi e^{\pm 2\pi i }),   \notag\\
\Rightarrow& g_3^{+}(\xi e^{2\pi i})  - g_3^{+}(\xi) = -\rho g_{2}^{+}(\xi e^{ 2\pi i}),
\end{align*}
which implies \eqref{negative side 3}.
Relations \eqref{negative side 1} and \eqref{positive side 3} hold immediately from the definitions \eqref{def g_1 xi} and \eqref{def g_3 xi}, respectively. It remains to show $g_2^{+}(\xi) = g_2^{-}(\xi)$ on the negative real axis, which holds when we use the relation $\xi_+^{-s} = \xi_-^{-s} e^{-2\pi i s}$ in \eqref{def g_2 xi}. The asymptotic behavior at infinity \eqref{G 2 xi} is a consequence of \cite[eq. 4.77]{Bertola_Bothner_2015}. Thus, all the assertions made in the theorem hold true.
\end{proof}

\subsubsection{The conformal map and definition of initial local parametrix problem}
We use the following locally conformal change of variables 
\begin{align}\label{Conformal map}
\xi = \xi(z) = \frac{1}{27}n^3z (f_0(z))^3, \quad -\pi <\arg z < \pi, \quad z \in U_\delta \Rightarrow \xi^\frac{1}{3}(z) = \frac{n}{3} z^\frac{1}{3} f_0(z),
\end{align}
where $f_0(z)$ is obtained on integrating \eqref{eq: h0 h1 h2 near 0} and collecting the polynomial coefficient of $z^{1/3}$. Now, define 
\begin{align}\label{eq: def P hat}
\widehat{\mathbf{P}(z)} = \widehat{N(z)} \bigg(\frac{\xi}{z}\bigg)^{M/3}  \Psi(\xi) \bigg(\frac{\xi}{z}\bigg)^{-A/3}.
\end{align}
\begin{proposition}
The matrix-valued function $\widehat{\mathbf{P}(z)}$ solves the RH problem proposed in Section \ref{subsec: making n-independent}.
\end{proposition}
\begin{proof}
Clearly, $\widehat{\mathbf{P}(z)}$ is analytic on $D(0,r_0) \backslash (\Delta \cup \Gamma^* \cup \Gamma^{*\pm} \cup \Delta^\pm)$. Next, we check the jump condition. From \eqref{eq: Jump condition for psi}, if $z \in \Delta$, we have
\begin{align*}
&J_{\widehat{\mathbf{P}}}(z) = [\widehat{\mathbf{P}_-(z)}]^{-1}\widehat{\mathbf{P}_+(z)}=\bigg(\frac{\xi}{z}\bigg)^{A/3} \begin{pmatrix}
      0& \xi^{\alpha_1}&0\\
        -\xi^{-\alpha_1} &0 & 0\\
         0& 0& 1 \\  
    \end{pmatrix} \bigg(\frac{\xi}{z}\bigg)^{-A/3}=\begin{pmatrix}
      0& z^{\alpha_1} &0\\
        -z^{-\alpha_1}&0 & 0\\
         0& 0& 1 \\  
    \end{pmatrix}.
\end{align*}
Similarly, the other jump matrices of $\widehat{\mathbf{P}(z)}$ as in \eqref{eq: jump condition P hat} can be verified. Finally, since all other terms in \eqref{eq: def P hat} remain bounded as $z \to 0$, the behavior of $\widehat{\mathbf{P}}$ near the origin will be same as that of $\Psi$.
\end{proof}
We further set  
\begin{align}\label{First guess of local parametrix}
&\mathbf{P}(z)= \widehat{N(z)} \bigg(\frac{\xi}{z}\bigg)^{M/3}  \Psi(\xi) \bigg(\frac{\xi}{z}\bigg)^{-A/3} \exp \big({ 3\xi^{\frac{1}{3}} \Omega_\pm } \big) D_n(z).
\end{align}
By virtue of \eqref{eq: hat P equal Dn}, it can be easily verified that $\mathbf{P}(z)$ satisfies the same jump conditions as $\mathring{P}(z)$. As will be shown, we will solve the RH problem for $\mathring{P}(z)$ with the help of $\mathbf{P}(z)$. To this end, using the expression \eqref{G 2 xi} for $\Psi(\xi)$, we can write
\begin{align}
&\mathbf{P}(z)=\widehat{N(z)} \bigg(\frac{\xi}{z}\bigg)^{\frac{M}{3}} \xi^{-\frac{M}{3}} \mathcal{U}^\pm  \tilde{K}(\xi)  \xi^\frac{A}{3} \exp \big({ -3\xi^{\frac{1}{3}} \Omega_\pm} \big) \bigg(\frac{\xi}{z}\bigg)^{-\frac{A}{3}} \exp \big({ 3\xi^{\frac{1}{3}} \Omega_\pm } \big) D_n(z) \notag \\
&=\widehat{N(z)} z^{-\frac{M}{3}} \mathcal{U}^\pm  \tilde{K}(\xi) z^\frac{A}{3} D_n(z). \notag
\end{align}
Further, as $n \rightarrow \infty$ for $0 < |z| < r$ with $r$ sufficiently small,
\begin{align}\label{asymptotic expansion Kj}
\mathbf{P}(z)(N(z))^{-1}& =\widehat{N(z)} z^{-\frac{M}{3}} \mathcal{U}^\pm  \tilde{K}(\xi) z^\frac{A}{3} D_n(z)z^{-\frac{A}{3}} [\mathcal{U}^\pm]^{-1} z^{\frac{M}{3}}[\widehat{N(z)}]^{-1} \notag\\
&=\widehat{N(z)} z^{-\frac{M}{3}} \mathcal{U}^\pm  \tilde{K}(\xi) [\mathcal{U}^\pm]^{-1}\mathcal{U}^\pm D_n(z) [\mathcal{U}^\pm]^{-1} z^{\frac{M}{3}}[\widehat{N(z)}]^{-1} \notag\\
&=\widehat{N(z)}z^{-\frac{M}{3}} \mathcal{K}(z) H(z) z^{\frac{M}{3}} [\widehat{N(z)}]^{-1}, \quad \pm \im z > 0 ,
\end{align}
where the function $\mathcal{K}(z)$ and $H(z)$ are given by 
\begin{align}
&\mathcal{K}(z)=\mathcal{U}^\pm \tilde{K}(\xi) [\mathcal{U}^\pm]^{-1}, \quad H(z)=  \mathcal{U}^\pm D_n(z)[\mathcal{U}^\pm]^{-1}, \quad \pm \im z > 0, \label{HU = UH}
\end{align}
with $\mathcal{U}^\pm$ as in \eqref{U plus and U minus}.

\begin{proposition}\label{Prop_conjugate H entire}
Let $z^\gamma$ be defined for $-\pi <\arg z < \pi$ such that $z^\gamma > 0$ for $z>0$. 
Then $\mathbb{H}(z)=z^{-\frac{M}{3}}H(z)z^{\frac{M}{3}}$ is analytic and non-singular in the neighborhood of $z=0$.
\end{proposition}
\begin{proof}
We prove it by showing that the function $\mathbb{H}(z)$ has no jumps on $\Delta$ and $\Gamma^*$. If $z \in \Delta$, then
\begin{align*}
&\mathbb{H}^+(z)=z^{-\frac{M}{3}}H^+(z)z^{\frac{M}{3}}=z^{-\frac{M}{3}} \mathcal{U}^+ D^+_n(z)[\mathcal{U}^+]^{-1} z^{\frac{M}{3}}=z^{-\frac{M}{3}} \widehat{\mathcal{U}} D^+_n(z)[\widehat{\mathcal{U}}]^{-1} z^{\frac{M}{3}}\\
&\mathbb{H}^-(z)=z^{-\frac{M}{3}} \mathcal{U}^- D^-_n(z)[\mathcal{U}^-]^{-1} z^{\frac{M}{3}}= z^{-\frac{M}{3}} \mathcal{U}^+J D^-_n(z)J^{-1} [\mathcal{U}^+]^{-1}z^{\frac{M}{3}}.
\end{align*}
To prove, $[\mathbb{H}^-(z)]^{-1} \mathbb{H}^+(z) = \mathbb{I}$, it is sufficient to verify 
\begin{align*}
JD^-_n(z)J^{-1} = D^+_n(z)
\end{align*}
which, on using \eqref{eq: def Dn z} simplifies to 
\begin{align*}
&\begin{pmatrix}
      0& -1 & 0 \\
        1 & 0 & 0\\
        0 & 0 & 1 \\  
    \end{pmatrix} \begin{pmatrix}
      e^{-nH_0^-}  & 0 & 0 \\
      0  & e^{-nH_1^-} & 0 \\
      0  & 0 & e^{-nH_2^-} 
    \end{pmatrix}\begin{pmatrix}
      0& 1 & 0 \\
        -1 & 0 & 0\\
        0 & 0 & 1 \\  
    \end{pmatrix} = \begin{pmatrix}
      e^{-nH_0^+}  & 0 & 0 \\
      0  & e^{-nH_1^+} & 0 \\
      0  & 0 & e^{-nH_2^+} 
    \end{pmatrix}\\
\Longrightarrow & \begin{pmatrix}
      e^{-nH_1^-}  & 0 & 0 \\
      0  & e^{-nH_0^-} & 0 \\
      0  & 0 & e^{-nH_2^-} 
    \end{pmatrix} = \begin{pmatrix}
      e^{-nH_0^+}  & 0 & 0 \\
      0  & e^{-nH_1^+} & 0 \\
      0  & 0 & e^{-nH_2^+} 
    \end{pmatrix}
\end{align*}
which hold true since $H_0^\pm = H_1^\mp$ and $H_2^+=H_2^-$ on $\Delta$. Next, if $z \in \Gamma^*$ and $\widehat{\Omega} = \diag [\omega^2,1, \omega^{-2}]$, then
\begin{align*}
&\mathbb{H}^+(z)=z_+^{-\frac{M}{3}}H^+(z)z_+^{\frac{M}{3}}=z_+^{-\frac{M}{3}} \mathcal{U}^+ D^+_n(z)[\mathcal{U}^+]^{-1} z_+^{\frac{M}{3}}=z_+^{-\frac{M}{3}} \widehat{\mathcal{U}} D^+_n(z)[\widehat{\mathcal{U}}]^{-1} z_+^{\frac{M}{3}}.\\
&\mathbb{H}^-(z)=z_-^{-\frac{M}{3}} \mathcal{U}^- D^-_n(z)[\mathcal{U}^-]^{-1} z_-^{\frac{M}{3}}= z_+^{-\frac{M}{3}} \widehat{\Omega} \mathcal{U}^+J D^-_n(z) J^{-1}  [\mathcal{U}^+]^{-1} \widehat{\Omega}^{-1} z_+^{\frac{M}{3}}\\
&=z_+^{-\frac{M}{3}} \widehat{\Omega} \widehat{\mathcal{U}} \mathcal{W} J D^-_n(z) J^{-1} \mathcal{W}^{-1} \widehat{\mathcal{U}}^{-1} \widehat{\Omega}^{-1} z_+^{\frac{M}{3}} = z_+^{-\frac{M}{3}} \widehat{\Omega} \widehat{\mathcal{U}} J D^-_n(z) J^{-1}  \widehat{\mathcal{U}}^{-1} \widehat{\Omega}^{-1} z_+^{\frac{M}{3}}.
\end{align*}
The last equality follows since it can be easily verified that $J D^-_n(z) J^{-1}$ is a diagonal matrix. Next, we compute
\begin{align*}
&\widehat{\mathcal{U}}^{-1}\widehat{\Omega} \widehat{\mathcal{U}} J D^-_n(z) J^{-1}  \widehat{\mathcal{U}}^{-1} \widehat{\Omega}^{-1}\widehat{\mathcal{U}}\\
=& \begin{pmatrix}
      0& 1 & 0 \\
        0 & 0 & 1\\
        1 & 0 & 0 \\  
    \end{pmatrix}\begin{pmatrix}
      e^{-nH_1^-}  & 0 & 0 \\
      0  & e^{-nH_0^-} & 0 \\
      0  & 0 & e^{-nH_2^-} 
    \end{pmatrix} \begin{pmatrix}
      0& 0 & 1 \\
        1 & 0 & 0\\
        0 & 1 & 0 \\  
    \end{pmatrix}=\begin{pmatrix}
      e^{-nH_0^-}& 0 & 0 \\
        0 & e^{-nH_2^-} & 0\\
        0 & 0 & e^{-nH_1^-} \\  
    \end{pmatrix}.
\end{align*}
Again, to show that $[\mathbb{H}^-(z)]^{-1} \mathbb{H}^+(z) = \mathbb{I}$, it is sufficient to verify
\begin{align*}
&\widehat{\Omega} \widehat{\mathcal{U}} J D^-_n(z) J^{-1}  \widehat{\mathcal{U}}^{-1} \widehat{\Omega}^{-1} = \widehat{\mathcal{U}} D^+_n(z)[\widehat{\mathcal{U}}]^{-1}\\
{\rm or,~ equivalently,} \quad & \widehat{\mathcal{U}}^{-1}\widehat{\Omega} \widehat{\mathcal{U}} J D^-_n(z) J^{-1}  \widehat{\mathcal{U}}^{-1} \widehat{\Omega}^{-1}\widehat{\mathcal{U}} = D^+_n(z)
\end{align*}
which holds since $H_2^\pm - H_1^\mp= \mp 2 \pi i$ and $H_0^+=H_0^-$ on $\Gamma^*$ and we achieve the desired result. The non-singularity of $\mathbb{H}(z)$ follows from the fact that the determinant of $H(z)$ given by \eqref{HU = UH} is $1$ and hence $\det \mathbb{H}(z) = 1$.
\end{proof}
\begin{proposition}\label{Remark: H_1 entire}
Let $z^\gamma$ be defined for $-\pi <\arg z < \pi$ such that $z^\gamma > 0$ for $z>0$. 
Then 
\begin{align*}
\mathbb{H}_1(z) = z^{-\frac{M}{3}} \mathcal{U}^\pm  \exp \big({ -3\xi^{\frac{1}{3}} \Omega_\pm} \big) [\mathcal{U}^\pm]^{-1} z^{\frac{M}{3}}
\end{align*}
is also entire.
\end{proposition}
\begin{proof}
Similar to the proof of Proposition \ref{Prop_conjugate H entire}, we show that $\mathbb{H}_1(z)$ has no jumps on $\Delta$ and $\Gamma^*$. On $\Delta$, we have
\begin{align*}
[\mathbb{H}_1^-(z)]^{-1} \mathbb{H}_1^+(z) &= z^{-\frac{M}{3}}\mathcal{U}^- \exp \big({ 3\xi^{\frac{1}{3}} \Omega_-} \big) [\mathcal{U}^-]^{-1} \mathcal{U}^+ \exp \big({- 3\xi^{\frac{1}{3}} \Omega_+}\big)[\mathcal{U}^+]^{-1} z^{\frac{M}{3}}\\
&=z^{-\frac{M}{3}}\mathcal{U}^- \exp \big({ 3\xi^{\frac{1}{3}} \Omega_-} \big) J^{-1} \exp \big({- 3\xi^{\frac{1}{3}} \Omega_+}\big)[\mathcal{U}^+]^{-1} z^{\frac{M}{3}} \\
&= z^{-\frac{M}{3}}\mathcal{U}^- J^{-1} [\mathcal{U}^+]^{-1} z^{\frac{M}{3}}\\
&=z^{-\frac{M}{3}}\mathcal{U}^+ [\mathcal{U}^+]^{-1} z^{\frac{M}{3}} = \mathbb{I}.
\end{align*}
On $\Gamma^*$, we have
\begin{align*}
&\mathbb{H}_1^+(z)=z_+^{-\frac{M}{3}} \mathcal{U}^+ \exp\big({- 3\xi_+^{\frac{1}{3}} \Omega_+}\big)[\mathcal{U}^+]^{-1} z_+^{\frac{M}{3}}=z_+^{-\frac{M}{3}} \widehat{\mathcal{U}} \exp\big({- 3\xi_+^{\frac{1}{3}} \Omega_+}\big)[\widehat{\mathcal{U}}]^{-1} z_+^{\frac{M}{3}}.\\
&\mathbb{H}_1^-(z)=z_-^{-\frac{M}{3}} \mathcal{U}^- \exp\big({- 3\xi_-^{\frac{1}{3}} \Omega_-}\big)[\mathcal{U}^-]^{-1} z_-^{\frac{M}{3}}= z_+^{-\frac{M}{3}} \widehat{\Omega} \mathcal{U}^+J \exp\big({- 3\xi_+^{\frac{1}{3}} \omega^{-2} \Omega_-}\big) J^{-1}  [\mathcal{U}^+]^{-1} \widehat{\Omega}^{-1} z_+^{\frac{M}{3}}\\
&=z_+^{-\frac{M}{3}} \widehat{\Omega} \widehat{\mathcal{U}} \mathcal{W} J \exp\big({- 3\xi_+^{\frac{1}{3}} \widehat{\Omega}}\big) J^{-1} \mathcal{W}^{-1} \widehat{\mathcal{U}}^{-1} \widehat{\Omega}^{-1} z_+^{\frac{M}{3}} = z_+^{-\frac{M}{3}} \widehat{\Omega} \widehat{\mathcal{U}} J \exp\big({- 3\xi_+^{\frac{1}{3}} \widehat{\Omega}}\big) J^{-1}  \widehat{\mathcal{U}}^{-1} \widehat{\Omega}^{-1} z_+^{\frac{M}{3}}.
\end{align*}
To show that $[\mathbb{H}_1^-(z)]^{-1} \mathbb{H}_1^+(z) = \mathbb{I}$, it is sufficient to verify
\begin{align*}
&\widehat{\Omega} \widehat{\mathcal{U}} J \exp\big({- 3\xi_+^{\frac{1}{3}} \widehat{\Omega}}\big) J^{-1}  \widehat{\mathcal{U}}^{-1} \widehat{\Omega}^{-1} = \widehat{\mathcal{U}} \exp\big({- 3\xi_+^{\frac{1}{3}} \Omega_+}\big)[\widehat{\mathcal{U}}]^{-1}\\
{\rm or,~ equivalently,} \quad & [\widehat{\mathcal{U}}]^{-1}\widehat{\Omega} \widehat{\mathcal{U}} J \exp\big({- 3\xi_+^{\frac{1}{3}} \widehat{\Omega}}\big) J^{-1}  \widehat{\mathcal{U}}^{-1} \widehat{\Omega}^{-1} \widehat{\mathcal{U}}= \exp\big({- 3\xi_+^{\frac{1}{3}} \Omega_+}\big)
\end{align*}
Now, we compute
\begin{align*}
&[\widehat{\mathcal{U}}]^{-1}\widehat{\Omega} \widehat{\mathcal{U}} J \exp\big({- 3\xi^{\frac{1}{3}} \widehat{\Omega}}\big) J^{-1}  \widehat{\mathcal{U}}^{-1} \widehat{\Omega}^{-1} \widehat{\mathcal{U}} \\
=& \begin{pmatrix}
      0& 1 & 0 \\
        0 & 0 & 1\\
        1 & 0 & 0 \\  
    \end{pmatrix} J\begin{pmatrix}
      e^{- 3\xi_+^{\frac{1}{3}} \omega^2} & 0 & 0 \\
        0 & e^{- 3\xi_+^{\frac{1}{3}} } & 0\\
        0 & 0 & e^{- 3\xi_+^{\frac{1}{3}} \omega^{-2}} \\  
    \end{pmatrix}J^{-1} \begin{pmatrix}
      0& 0 & 1 \\
        1 & 0 & 0\\
        0 & 1 & 0 \\  
    \end{pmatrix}\\
=& \begin{pmatrix}
      0& 1 & 0 \\
        0 & 0 & 1\\
        1 & 0 & 0 \\  
    \end{pmatrix} \begin{pmatrix}
      e^{- 3\xi_+^{\frac{1}{3}} }  & 0 & 0 \\
        0 & e^{- 3\xi_+^{\frac{1}{3}} \omega^2} & 0\\
        0 & 0 & e^{- 3\xi_+^{\frac{1}{3}} \omega^{-2}} \\  
    \end{pmatrix} \begin{pmatrix}
      0& 0 & 1 \\
        1 & 0 & 0\\
        0 & 1 & 0 \\  
    \end{pmatrix}=\begin{pmatrix}
      e^{- 3\xi_+^{\frac{1}{3}} \omega^2}  & 0 & 0 \\
        0 & e^{- 3\xi_+^{\frac{1}{3}} \omega^{-2}} & 0\\
        0 & 0 & e^{- 3\xi_+^{\frac{1}{3}} }  \\  
    \end{pmatrix}
\end{align*}
which is clearly equal to $\exp\big({- 3\xi_+^{\frac{1}{3}} \Omega_+}\big)$. This completes the proof.
\end{proof}

\begin{proposition}
For the following structure of the matrix coefficients $\{\mathcal{K}_j \}_{j=1}^\infty$, appearing in the expansion \eqref{asymptotic expansion Kj}, 
\begin{align}\label{K structure}
\mathcal{K}_j =\begin{cases}
 \begin{pmatrix}
    0 &0 & *\\
    * & 0& 0 \\
    0 & *& 0
 \end{pmatrix}, & j = 1 \mod 3,\\
 \begin{pmatrix}
    0 &* & 0\\
    0 & 0& * \\
    * & 0& 0
 \end{pmatrix}, & j = 2 \mod 3,  \\ 
 \begin{pmatrix}
    * &0 & 0\\
     0& * & 0 \\
    0 & 0& *
 \end{pmatrix},&  j = 3 \mod 3,
\end{cases}
\end{align}
the function $z^{-\frac{M}{3}}\mathcal{K}(z)z^{\frac{M}{3}} = z^{-\frac{M}{3}} \bigg[I+\sum_{j=1}^{\infty}\mathcal{K}_j \xi^{-\frac{j}{3}}\bigg]  z^{\frac{M}{3}} $ admits a formal asymptotic expansion of the form 
\begin{align}\label{K formal series in prop}
z^{-\frac{M}{3}}\mathcal{K}(z)z^{\frac{M}{3}} = I+ \sum_{j=1}^{\infty} z^{-\frac{M}{3}} \mathcal{K}_j z^{\frac{M}{3}} \frac{1}{n^j z^{j/3}},
\end{align}
and contains only integer powers of $z$.
\end{proposition}
\begin{proof}
Proceeding similar to the proof of Proposition \ref{Prop: N hat}, we can observe that $\Psi(\xi) \xi^\frac{A}{3} [\mathcal{U}^\pm ]^{-1} \xi^{\frac{M}{3}}$ has no jump on $\mathbb{R}\backslash \{ 0\}$. Further, from \eqref{G 2 xi} and Remark \ref{Remark: H_1 entire}, we derive that
\begin{align*}
&\Psi(\xi) \xi^{-\frac{A}{3}} [\mathcal{U}^\pm ]^{-1} \xi^{\frac{M}{3}} = \xi^{-\frac{M}{3}} \mathcal{U}^\pm  \tilde{K}(\xi) \exp \big({ -3\xi^{\frac{1}{3}} \Omega_\pm} \big) [\mathcal{U}^\pm ]^{-1} \xi^{\frac{M}{3}}  \\
&=\xi^{-\frac{M}{3}} \mathcal{U}^\pm  \tilde{K}(\xi)[\mathcal{U}^\pm ]^{-1}\xi^{\frac{M}{3}}\xi^{-\frac{M}{3}}\mathcal{U}^\pm \exp \big({ -3\xi^{\frac{1}{3}} \Omega_\pm} \big) [\mathcal{U}^\pm ]^{-1} \xi^{\frac{M}{3}}\\
&=\xi^{-\frac{M}{3}} \mathcal{U}^\pm  \tilde{K}(\xi)[\mathcal{U}^\pm ]^{-1}\xi^{\frac{M}{3}}\mathbb{H}_1(\xi), \quad \xi \to \infty, \quad \pm ~Im(\xi) > 0.\\
\Longrightarrow&\bigg(\frac{n}{3}  f_0(z)\bigg)^{M} \Psi(\xi) \xi^{-\frac{A}{3}}  [\mathcal{U}^\pm ]^{-1} \xi^{\frac{M}{3}} [\mathbb{H}_1(\xi) ]^{-1} \bigg(\frac{n}{3}  f_0(z)\bigg)^{-M}=  z^{-\frac{M}{3}} \mathcal{K}(z) z^{\frac{M}{3}} .
\end{align*}
Now, the functions appearing on the LHS have no jump on the real axis. Thus, the LHS admits an asymptotic expansion in integer powers of $z$. Since $\xi^{-1/3} = 3[n f_0(z)]^{-1} z^{-\frac{1}{3}} $, the RHS admits an asymptotic expansion in powers of $z^{-\frac{1}{3}}$. But since the LHS has no jumps $\mathbb{R}\backslash \{ 0\}$, the expansion on the RHS must involve only inverse integer powers of $z$. Now, \eqref{K formal series in prop} follows from the expression of $\xi$ in \eqref{Conformal map}. Further, to identify the structure of $\mathcal{K}_j$'s in \eqref{K structure}, we calculate 
\begin{align*}
z^{-\frac{1}{6}\lambda_3} \mathcal{K}_j z^{\frac{1}{6}\lambda_3} z^{-j/3} = \begin{pmatrix}
   \mathcal{K}_{11}z^{-j/3} & \mathcal{K}_{12} z^{-(j+1)/3} & \mathcal{K}_{13}z^{-(j+2)/3} \\
   \mathcal{K}_{21} z^{(1-j)/3} & \mathcal{K}_{22}z^{-j/3} & \mathcal{K}_{23}z^{-(j+1)/3}  \\
   \mathcal{K}_{31} z^{(2-j)/3} & \mathcal{K}_{32} z^{(1-j)/3} &  \mathcal{K}_{33} z^{-j/3}
\end{pmatrix}.
\end{align*}
If $j \equiv 1 \pmod{3}$, the monomials corresponding to the entries 
$\mathcal{K}_{13}$, $\mathcal{K}_{21}$, and $\mathcal{K}_{32}$ have inverse integer powers of~$z$, 
while monomials for the remaining entries involve fractional powers of~$z$. 
Consequently, all coefficients except those of $\mathcal{K}_{13}$, $\mathcal{K}_{21}$, and $\mathcal{K}_{32}$ 
must vanish. An analogous argument determines the structure of $\mathcal{K}_j$ for 
$j \equiv 2 \pmod{3}$ and $j \equiv 0 \pmod{3}$.
\end{proof}

\subsubsection{Determination of $\mathcal{K}_1$} From \eqref{G 2 xi}, it is evident that
\begin{align}\label{eq: K z to Psi}
    \tilde{K}(z) = [\mathcal{U}^\pm]^{-1} z^{\frac{M}{3}} \Psi(z) \exp \big({ 3z^{\frac{1}{3}} \Omega_\pm} \big) z^{-\frac{A}{3}}. 
\end{align}

We have the following asymptotic expansions using \cite[Lemma 4.25]{Bertola_Bothner_2015}:
\begin{align*}
& g_1(z) = z^{-\gamma_1}\omega^{-2+\alpha_1+\alpha_2}e^{-3z^{\frac{1}{3}}\omega^2}[1+t_{12}\omega^{-2} z^{-\frac{1}{3}}+t_{13}\omega^{2} z^{-\frac{2}{3}}+O(z^{-1})],\\
& g_2^+(z) = z^{-\gamma_2}\omega^{2-\alpha_1-\alpha_2}e^{-3z^{\frac{1}{3}}\omega^{-2}}[1+t_{12}\omega^2 z^{-\frac{1}{3}}+t_{13}\omega^{-2} z^{-\frac{2}{3}}+O(z^{-1})],\\
&g_3(z) = z^{-\gamma_3}e^{-3z^{\frac{1}{3}}}[1+t_{12} z^{-\frac{1}{3}}+t_{13} z^{-\frac{2}{3}}+O(z^{-1})],
\end{align*}
where $\gamma_1 = \frac{1+\alpha_1+\alpha_2}{3}$, $\gamma_2= \frac{1-2\alpha_1+\alpha_2}{3}$, $\gamma_3= \frac{1-2\alpha_2+\alpha_1}{3}$, and $t_{12}=M_1 = \frac{\alpha_1^2+\alpha_2^2-\alpha_1 \alpha_2}{3} - \frac{1}{9}$. Differentiating the above equations to obtain the entries of the matrix $\widehat{\Psi}(z)$, we get
\begin{align*}
&z \partial_z g_1(z)= -z^{-\gamma_1}\omega^{-2+\alpha_1+\alpha_2}e^{-3z^{\frac{1}{3}}\omega^2}[z^{\frac{1}{3}}\omega^2+t_{22}+t_{23}z^{-\frac{1}{3}}\omega^{-2}+O(z^{-\frac{2}{3}})],\\
&(z \partial_z)^2 g_1(z)=z^{-\gamma_1}\omega^{-2+\alpha_1+\alpha_2}e^{-3z^{\frac{1}{3}}\omega^2} [z^{\frac{2}{3}}\omega^{-2}+t_{32}z^{\frac{1}{3}}\omega^{2}+t_{33}+O(z^{-\frac{1}{3}})],\\
&\big(z \partial_z -\alpha_1 \big)g^{+}_2(z) = -z^{-\gamma_2}\omega^{2-\alpha_1-\alpha_2}  e^{-3\omega^{-2}z^{\frac{1}{3}}}  [z^{\frac{1}{3}}\omega^{-2}+t_{22}+t_{23}z^{-\frac{1}{3}} \omega^2+O(z^{-\frac{2}{3}})],\\
&\big(z \partial_z -\alpha_1 \big)^2 g^{+}_2(z)=z^{-\gamma_2} \omega^{2-\alpha_1-\alpha_2}e^{-3z^{\frac{1}{3}}\omega^{-2}}  [z^{\frac{2}{3}} \omega^{2}+t_{32}z^{\frac{1}{3}}\omega^{-2}+t_{33}+O(z^{-\frac{1}{3}})],\\ 
&\big(z \partial_z -\alpha_2 \big)g_3(z)=-z^{-\gamma_3}e^{-3z^{\frac{1}{3}}} [z^{\frac{1}{3}}+t_{22}+t_{23}z^{-\frac{1}{3}}+O(z^{-\frac{2}{3}})],\\
&\big(z \partial_z -\alpha_2 \big)^2 g_3(z)= z^{-\gamma_3}e^{-3z^{\frac{1}{3}}} [z^{\frac{2}{3}}+t_{32}z^{\frac{1}{3}}+ t_{33}+O(z^{-\frac{1}{3}})],
\end{align*}
where $t_{22}=M_1+\gamma_1$, $t_{23}=M_2+M_1 (\gamma_1+\frac{1}{3})$, $t_{32}=M_1+2\gamma_1-\frac{1}{3}$, and $t_{33}=M_2+M_1(2\gamma_1 +\frac{1}{3})+\gamma_1^2$. Substituting all the entries of the matrix \eqref{G 2 xi} and then writing a matrix representation, we get
\begin{align*}
&z^{\frac{M}{3}} \Psi(z)=\Bigg[\begin{pmatrix}
\omega^{-2} & \omega^{2} & 1 \\
-1 & -1 & -1 \\
\omega^{2} & \omega^{-2} & 1
\end{pmatrix}+\begin{pmatrix}
t_{12}\omega^{2} & t_{12}\omega^{-2} & t_{12} \\
-t_{22} \omega^{-2}& -t_{22}\omega^{2} & -t_{22} \\
t_{32} & t_{32} & t_{32}
\end{pmatrix}z^{-\frac{1}{3}}+O(z^{-\frac{2}{3}}) \Bigg]\\
&\times\begin{pmatrix}
      \omega^{-A_1}& 0 &0\\
        0 &\omega^{A_1}  & 0\\
         0& 0& 1 \\  
    \end{pmatrix} \begin{pmatrix}
      z^{A_1/3}& 0 &0\\
        0 &z^{A_2/3}  & 0\\
         0& 0& z^{A_3/3} \\  
    \end{pmatrix} \begin{pmatrix}
      e^{-3z^{\frac{1}{3}}\omega^{2}}& 0 &0\\
        0 &e^{-3z^{\frac{1}{3}}\omega^{-2}} & 0\\
         0& 0& e^{-3z^{\frac{1}{3}}} \\  
    \end{pmatrix},
\end{align*}
which implies (in view of \eqref{eq: K z to Psi})
\begin{align*}
&\tilde{K}(z) = [\mathcal{U}^+]^{-1} \Bigg[\widehat{\mathcal{U}}+\begin{pmatrix}
t_{12}\omega^{2} & t_{12}\omega^{-2} & t_{12} \\
-t_{22} \omega^{-2}& -t_{22}\omega^{2} & -t_{22} \\
t_{32} & t_{32} & t_{32}
\end{pmatrix}z^{-\frac{1}{3}}+O(z^{-\frac{2}{3}}) \Bigg] \mathcal{W}.
\end{align*}
Hence, using \eqref{HU = UH}, we have
\begin{align}\label{eq: Round K1}
&\mathcal{K}_1 =\mathcal{U}^+ [\mathcal{U}^+]^{-1} \begin{pmatrix}
t_{12}\omega^{2} & t_{12}\omega^{-2} & t_{12} \\
-t_{22} \omega^{-2}& -t_{22}\omega^{2} & -t_{22} \\
t_{32} & t_{32} & t_{32}
\end{pmatrix} \mathcal{W} [\mathcal{U}^+]^{-1}=\frac{1}{3}\begin{pmatrix}
0 & 0 & t_{12} \\
-t_{22} & 0& 0 \\
0& -t_{32} & 0
\end{pmatrix} .
\end{align}

\subsubsection{The double matching} We will apply the double matching procedure from \cite{Molag_2021}. In this section, we will show that the conditions of \cite[Theorem 2.1]{Molag_2021} can be met; we repeat this theorem for convenience.

\begin{theorem}\label{Molag theorem}
Let $\mathring{P}$ and $N$ be defined in the neighborhood of $\overline{D(0,\rho)}$ for some $\rho > 0$. These are matrix-valued functions of size $m \times m$ that may vary with $n$. Let $a,b,c,d,e \geq 0$ satisfy
\begin{align*}
    a \leq e < b, \quad and \quad d < \min(b,c).
\end{align*}
Suppose that uniformly for $z \in \partial D(0,n^{-a})$ as $n \to \infty$
\begin{align*}
  \mathring{P}(z)N(z)^{-1} E(z) = \mathbb{I}+\frac{C(z)}{n^b z} + O(n^{-c}) ,
\end{align*}
where $C$ and $E$ are $m \times m$ functions in a neighborhood of $\overline{D(0,\rho)}$ that may vary with $n$ and
\begin{enumerate}
    \item $C$ is meromorphic with only a possible pole at $z=0$, whose order is bounded by some non-negative integer $p$ for all $n$, and $C$ is uniformly bounded for $z \in \partial D(0,n^{-a})$ as $n \to \infty$,
    \item $E$ is non-singular, analytic, and uniformly for $z,w \in \partial D(0,n^{-a})$, we have as $n \to \infty$ ,
\begin{align*}
    E(z) = O(n^{d/2}), \quad E(z)^{-1} = O(n^{d/2}), \quad and \quad E(z)^{-1}E(w)= \mathbb{I}+O(n^e (z-w)).
\end{align*}
\end{enumerate}
Then there are non-singular analytic functions $E_n^0:\overline{D(0,n^{-a})} \to \mathbb{C}^{m \times m}$ and $E_n^\infty:\overline{A(0,n^{-a},\infty)} \to \mathbb{C}^{m \times m}$ such that as $n \to \infty$,
\begin{align*}
 & E_n^0(z)   \mathring{P}(z) = [\mathbb{I}+O(n^{d-c})] E_n^\infty(z) N(z), \quad uniformly~for ~ z \in \partial D(0,n^{-a}), \\
  & E_n^\infty(z) = \mathbb{I}+O(n^{d-b}), \quad uniformly~for ~ z \in \partial D(0,\rho).
\end{align*}

\end{theorem}
Our present situation requires that one takes $m=3$. The main objective of the current section is to show that assumptions of the theorem \ref{Molag theorem}, and in particular the estimates in (2), hold for some choice of the constants $a,b,c,d$ and $e$. We shall determine the explicit values of constants $a,b,c,d$ and $e$ as we go along. We want a double matching on the circle with radius $\rho=r_0$ and the circle of radius $r_n$ such that
\begin{align*}
    r_n = n^{-\frac{3}{2}}.
\end{align*}
As mentioned earlier, we can use the analytic prefactors $E_n^0$ and $E_n^\infty$ from the theorem \ref{Molag theorem} and construct local parametrix $P$ as
\begin{align*}
P(z)= \begin{cases}
    E_n^0(z) \mathring{P}(z), & z \in D(0,r_n),\\
   E_n^\infty(z) N(z), & z \in A(0,r_n,r_0),
\end{cases}
\end{align*}
and $P$ will then satisfy a matching condition on both the inner and outer circle.

\begin{definition}
With $\Psi$, $\xi(z)$, $\Omega$, $\widehat{N(z)}$, and $P(z)$ as in \eqref{First guess of local parametrix}, \eqref{G 2 xi}, \eqref{Conformal map} and \eqref{Omega and its tilde}, we define the initial local parametrix by 
\begin{align*}
\mathring{P}(z) = \left(\frac{z}{\xi}\right)^{\frac{M}{3}}\widehat{N(z)}^{-1}\mathbf{P}(z).
\end{align*}
\end{definition}

\begin{definition}\label{Def_E z}
We define the function
  \begin{align}\label{eq: E(z)}
    E(z) = \widehat{N(z)} z^{-\frac{M}{3}} H(z)^{-1} \xi^{\frac{M}{3}}.
\end{align}  
\end{definition}
We have to prove some properties of $E$.
\begin{proposition}
 E(z) is an analytic and non-singular function in the neighborhood of $z=0$.
\end{proposition}
\begin{proof}
Using \eqref{Conformal map} and \eqref{HU = UH}, we can rewrite $E(z)$ as
  \begin{align*}
    E(z) &= \widehat{N(z)} z^{-\frac{M}{3}} \Bigg[ \mathcal{U}^\pm  D_n(z) [\mathcal{U}^\pm]^{-1} \Bigg]^{-1} z^{\frac{M}{3}}  \bigg(\frac{n}{3}\bigg)^{M} [f_0(z)]^{M}.
\end{align*}
Since all the matrices in the above expression of $E(z)$ are non-singular, therefore, $E(z)$ is also non-singular. 
Now, $\widehat{N(z)}$ is analytic with $z=0$ as a removable singularity by Proposition \ref{Prop: N hat}, $z^{-\frac{M}{3}} H(z)^{-1} z^{\frac{M}{3}}$ is analytic and non-singular in the neighborhood of $z=0$ by Proposition \ref{Prop_conjugate H entire}, and $f_0(z)$ is a conformal mapping and hence analytic. Further, we know that the product of analytic functions is analytic and hence, the proposition follows. 
\end{proof}

\begin{lemma}\label{Lm: Ez approx}
Uniformly for $z \in \overline{D(0,r_n)}$, we have as $n \to \infty$
\begin{align*}
E(z) = O(n), \quad E(z)^{-1} = O(n).
\end{align*}
\end{lemma}
\begin{proof}
$\widehat{N(z)}$ is analytic at $z=0$ by Proposition \ref{Prop: N hat} and does not depend on $n$. Therefore, it is certainly uniformly bounded for $z \in \overline{D(0,r_n)}$. From \eqref{HU = UH}, it can easily be seen that $H$ and $H^{-1}$ are uniformly bounded for $z \in \overline{D(0,r_n)}$. Thus, each of the factors in $z^{-\frac{M}{3}} H(z)^{-1} z^{\frac{M}{3}} [f_0(z)]^{M}$ remains bounded for $|z|=r_n$, and since the product is analytic, the product remains uniformly bounded for $D(0,r_n)$ by maximum modulus principle. So the only factor that grows comes from the matrix $n^{M}$, which due to specific form of $M$ is $O(n)$, that leads to $E(z) = O(n)$. Since $\det E(z) \neq 0$ near $z=0$, $E(z)^{-1} = O(n)$.
\end{proof}

\begin{lemma}
Uniformly for $z,w \in \overline{D(0,r_n)}$, we have as $n \to \infty$
\begin{align*}
E(z)^{-1}E(w)  = \mathbb{I}+O(n^{5/2} (z-w)).
\end{align*}
\end{lemma}
\begin{proof}
We have $[\widehat{N(z)}]^{-1}\widehat{N(w)}  = \mathbb{I}+O(z-w) $ if $z,w \in \overline{D(0,r_n)}$. applying these estimates and using \eqref{Def_E z} implies for $z,w \in \overline{D(0,r_n)}$,
\begin{align*}
E(z)^{-1}E(w) = \xi(z)^{-\frac{M}{3}} H(z) z^{\frac{M}{3}}(\mathbb{I}+O(z-w))w^{-\frac{M}{3}}H(w)^{-1}\xi(w)^{\frac{M}{3}}.
\end{align*}
We first estimate $z^{\frac{M}{3}}(\mathbb{I}+O(z-w))w^{-\frac{M}{3}}$ as
\begin{align*}
&z^{\frac{M}{3}}(\mathbb{I}+O(z-w))w^{-\frac{M}{3}} = z^{\frac{M}{3}}(\mathbb{I}+O(z-w)) z^{-\frac{M}{3}} \bigg(\mathbb{I}+\frac{w-z}{z} \bigg)^{-\frac{M}{3}}\\
&=\diag[n^{-1/2},1,n^{1/2}](\mathbb{I}+O(z-w))\diag[n^{1/2},1,n^{-1/2}](\mathbb{I}+O(n^{3/2}(z-w)))\\
&=(\mathbb{I}+O(n(z-w)))(\mathbb{I}+O(n^{3/2}(z-w)))=\mathbb{I}+O(n^{3/2}(z-w)),
\end{align*}
where the last equality follows since $|z-w|\leq n^{-3/2}$. Further, owing to the boundedness of $H$ and $H^{-1}$, it remains to estimate $\xi(z)^{-\frac{M}{3}}(\mathbb{I}+O(n^{3/2}(z-w)))\xi(w)^{\frac{M}{3}}$ as
\begin{align*}
&\xi(z)^{-\frac{M}{3}}(\mathbb{I}+O(n^{3/2}(z-w)))\xi(w)^{\frac{M}{3}}=\xi(z)^{-\frac{M}{3}}(\mathbb{I}+O(n^{3/2}(z-w)))\xi(z)^{\frac{M}{3}}\bigg(\mathbb{I}+\frac{\xi(w)-\xi(z)}{\xi(z)} \bigg)^{\frac{M}{3}}\\
&=\diag[n^{-1/2},1,n^{1/2}](\mathbb{I}+O(n^{3/2}(z-w)))\diag[n^{1/2},1,n^{-1/2}](\mathbb{I}+O(n^{3/2}(z-w)))\\
&=(\mathbb{I}+O(n^{5/2}(z-w)))(\mathbb{I}+O(n^{3/2}(z-w)))=\mathbb{I}+O(n^{5/2}(z-w)),
\end{align*}
where the last equality follows since $|z-w|\leq n^{-5/2} < n^{-3/2} $ and the proof is complete.
\end{proof}

With the definitions described above, we obtain by \eqref{asymptotic expansion Kj} the asymptotic expansion
\begin{align}\label{asymptotic expansion in Kj}
\mathring{P}(z)(N(z))^{-1}E(z) = \xi^{-\frac{M}{3}} \bigg[I+\sum_{j=1}^{\infty}\mathcal{K}_j \xi^{-\frac{j}{3}}\bigg]  \xi^{\frac{M}{3}}.
\end{align}
In particular, taking the asymptotic series up to $j=7$, we find ($n$-independent) constants $\mathcal{K}_1$ and $\mathcal{K}_2$ such that 
\begin{align}\label{asymptotic expantion in tilde K}
\mathring{P}(z)(N(z))^{-1}E(z) = \mathbb{I}+\frac{\mathcal{K}_1}{\xi} +\frac{\mathcal{K}_2}{\xi^2}+O(\xi^{-3}) = \mathbb{I}+\frac{\mathcal{K}_1}{\xi} +O(\xi^{-2}).
\end{align}
In line with Theorem \ref{Molag theorem}, we define the constants 
\begin{align*}
a=\frac{3}{2}, \quad b=3, \quad d=2, \quad e=\frac{5}{2}.
\end{align*}
We still need to find a convenient value for $c$. If we take two terms of integer powers of $\xi$ in the above expansion, then this seems to be enough for the reasons that will become clearer soon, i.e., we identify
\begin{align}\label{eq: C(z)}
C(z) = \frac{n^3z}{\xi(z)}\mathcal{K}_1 = \frac{27 \mathcal{K}_1}{f_0(z)},
\end{align}
where $\mathcal{K}_1$ is given by \eqref{eq: Round K1}. Note that such a $C(z)$ is uniformly bounded on $\partial D(0,r_n)$ as $n \to \infty$. Infact, using \eqref{asymptotic expantion in tilde K}, we now have uniformly for $z \in \partial D(0,r_n)$ that
\begin{align}\label{asymptotic expansion in Cz}
\mathring{P}(z)(N(z))^{-1}E(z) = \mathbb{I}+\frac{C(z)}{n^3z}+O(n^{-3}),
\end{align}
as $n \to \infty$, and we identify $c=3 (<d)$. Now, all the requirements of the theorem \eqref{Molag theorem} are met. Hence, we obtain analytic prefactors $E_n^0:\overline{D(0,r_n)} \to \mathbb{C}$ and $E_n^\infty:\overline{A(0,r_n,r_0)} \to \mathbb{C}$ such that
\begin{align}
& E_n^0(z)   \mathring{P}(z) = [\mathbb{I}+O(n^{-1})] E_n^\infty(z) N(z), \quad uniformly~for ~ z \in \partial D(0,r_n), \label{PRefactor 1}\\
& E_n^\infty(z) = \mathbb{I}+O(n^{-1}), \quad uniformly~for ~ z \in \partial D(0,r_0). \label{PRefactor 2}
\end{align}
We are now ready to fix the definition of local parametrix $P$ at the origin as given by \eqref{definition P(z)} and based on our discussion in section \ref{Bare Parametrix}, \eqref{PRefactor 1} and \eqref{PRefactor 2}, we achieve the double matching as described in \eqref{double matching of P 1} and \eqref{double matching of P 2}.

\subsubsection{Construction of analytic prefactors $E_n^0(z)$ and $E_n^\infty(z)$}
In this section, we iteratively construct the analytic prefactors $E_n^0(z)$ and $E_n^\infty(z)$ as defined in \eqref{definition P(z)} and which follow the asymptotic condition as imposed by Theorem \ref{Molag theorem}.
\begin{definition}
We define
\begin{align*}
P^{(i)}(z) = \begin{cases}
    E_n^{(i)}(z) \mathring{P}(z), & z \in D(0,r_n),\\
   E_n^{\infty(i)}(z) N(z), & z \in A(0,r_n,r_0),
\end{cases}    
\end{align*}    
with $E_n^{(i)}(z) = E(z) $ and $E_n^{\infty(i)}(z) = \mathbb{I}$.
\end{definition}
\subsubsection{First Iteration}
\begin{lemma}\label{Lm: First iteration}
For $|z|=r_n$, 
\begin{align}\label{eq: First iteration}
P^{(i)}_+(z) = (\mathbb{I}+O(n^{1/2})) P^{(i)}_-(z).
\end{align}
\end{lemma}
\begin{proof}
Using \eqref{asymptotic expansion in Kj} and \eqref{asymptotic expansion in Cz} gives us
\begin{align*}
 & P^{(i)}_+(z) P^{(i)}_-(z)^{-1} = E_n^{(i)}(z) \mathring{P}(z) N(z)^{-1} \\
& = E(z) \xi^{-\frac{M}{3}} \bigg[I+\sum_{j=1}^{\infty}\mathcal{K}_j \xi^{-\frac{1}{3}}\bigg] \xi^{\frac{M}{3}} [E(z)]^{-1}  \\
& = E(z) \bigg[\mathbb{I}+\frac{C(z)}{n^3z}+O(n^{-3}) \bigg]  [E(z)]^{-1}  \\
& = \mathbb{I}+\frac{C_n^{(i)}(z)}{n^3z}+ E(z) O(n^{-3}) [E(z)]^{-1}, ~{\rm where}~ C_n^{(i)}(z)=E(z) C(z) [E(z)]^{-1}.
\end{align*}
Since $E(z)$ and $[E(z)]^{-1}$ are $O(n)$ as $n \to \infty$ as shown in Lemma \ref{Lm: Ez approx}, it follows that $C_n^{(i)}(z)= O(n^2)$ and consequently, $\frac{C_n^{(i)}(z)}{n^3z} = O(n^{1/2})$ as $n \to \infty$ with $|z|=n^{-3/2}$. This completes the proof of Lemma \ref{Lm: First iteration}.
\end{proof}
It can be observed that \eqref{eq: First iteration} does not obey the asymptotic rules as desired by Theorem \ref{Molag theorem}. Hence, we need to modify the analytic prefactors $E_n^{(i)}(z)$ and $E_n^{\infty(i)}(z)$. We record \cite[Proposition 5.15]{Kuijlaars_Molag_2019} adapted to our setting for later use.
\begin{proposition}
Let $k=1,2,3\ldots$
\begin{enumerate}
    \item We have uniformly for $z \in \overline{D(0,r_n)}$ that
    \begin{align}\label{eq: Arno eq 1}
        [C_n^{(i)}(z)]^k = O(n^2).
    \end{align}
    \item We have uniformly for $z_1,z_2, \ldots,z_k \in \overline{ D(0,r_n)}$ that
    \begin{align}\label{eq: Arno eq 2}
        C_n^{(i)}(z_1) \ldots C_n^{(i)}(z_k) = O(n^{k+1}).
    \end{align}
    When $\ell$ indices are such that $z_j=z_{j+1}$ for $j \in \{1,2,\ldots, k-1 \}$, then \eqref{eq: Arno eq 2} is replaced by
    \begin{align}\label{eq: Arno eq 3}
        C_n^{(i)}(z_1) \ldots C_n^{(i)}(z_k) = O(n^{k+1-\ell}).
    \end{align}
    \item We have uniformly for $z_1,z_2 \in \overline{ D(0,r_n)}$ that
    \begin{align}
     &C_n^{(i)}(z_1)E_n(z_2)   = O(n^2) , \notag \\ 
     &E_n^{-1}(z_1) C_n^{(i)}(z_2) = O(n^2).  \label{eq: Arno eq 5}
    \end{align}
\end{enumerate}
\end{proposition}

\subsubsection{Second Iteration} We eliminate the growing term in \eqref{eq: First iteration} by defining
\begin{definition}
 \begin{align*}
& E_n^{(ii)}(z) = \bigg(\mathbb{I}-\frac{C_n^{(i)}(z)-C_n^{(i)}(0)}{n^3 z} \bigg) E_n^{(i)}(z), \\
&E_n^{\infty(ii)}(z) = \bigg(\mathbb{I}-\frac{C_n^{(i)}(0)}{n^3 z} \bigg)^{-1} E_n^{\infty(i)}(z),
 \end{align*}   
and
\begin{align*}
P^{(ii)}(z) = \begin{cases}
    \bigg(\mathbb{I}-\dfrac{C_n^{(i)}(z)-C_n^{(i)}(0)}{n^3 z} \bigg) P^{(i)}(z), & z \in D(0,r_n),\\
   \bigg(\mathbb{I}-\dfrac{C_n^{(i)}(0)}{n^3 z} \bigg)^{-1} P^{(i)}(z), & z \in A(0,r_n,r_0).
\end{cases}    
\end{align*}
\end{definition}
For large $n$, the inverse $\bigg(\mathbb{I}-\dfrac{C_n^{(i)}(0)}{n^3 z} \bigg)^{-1}$ does exist since $\bigg(\dfrac{C_n^{(i)}(0)}{n^3 z} \bigg)^{2}= O(n^{-1})$ for $|z|> n^{-3/2}$ by \eqref{eq: Arno eq 1} which, in turn, implies that $\dfrac{C_n^{(i)}(0)}{n^3 z} $ does not have an eigenvalue equal to $1$.

\begin{lemma}\label{lm: Second iteration}
On $\partial D(0,r_n)$, we have the jump
\begin{align}\label{eq: Second iteration}
P^{(ii)}_+(z) = (\mathbb{I}+O(1)) P^{(ii)}_-(z).
\end{align}
\end{lemma}
\begin{proof}
First of all, we compute
\begin{align*}
&\bigg(\mathbb{I}+\frac{C_n^{(i)}(z)}{n^3z}+ E(z) O(n^{-3}) [E(z)]^{-1} \bigg) \bigg(\mathbb{I}-\dfrac{C_n^{(i)}(0)}{n^3 z} \bigg) \\
=& \mathbb{I}+\dfrac{C_n^{(i)}(z)-C_n^{(i)}(0)}{n^3 z}-\dfrac{C_n^{(i)}(z)C_n^{(i)}(0)}{n^6 z^2}+O(n^{-1})
\end{align*}
since the behavior of the last two terms is:
\begin{align*}
&E(z) O(n^{-3}) [E(z)]^{-1} = O(n^{-1})\\
&E(z) O(n^{-3}) [E(z)]^{-1} \dfrac{C_n^{(i)}(0)}{n^3 z} = \dfrac{E(z) O(n^{-3}) O(n^{-2})}{n^3 z} = O(n^{-3/2}) , 
\end{align*}
where we have used \eqref{eq: Arno eq 5}. Next, we calculate 
\begin{align}\label{eq:second iteration proof}
&P^{(ii)}_+(z) P^{(ii)}_-(z)^{-1} \notag\\
&= \bigg(\mathbb{I}-\dfrac{C_n^{(i)}(z)-C_n^{(i)}(0)}{n^3 z} \bigg) \bigg(\mathbb{I}+\frac{C_n^{(i)}(z)}{n^3z}+ E(z) O(n^{-3}) [E(z)]^{-1} \bigg) \bigg(\mathbb{I}-\dfrac{C_n^{(i)}(0)}{n^3 z} \bigg) \notag \\
&=\bigg(\mathbb{I}-\dfrac{C_n^{(i)}(z)-C_n^{(i)}(0)}{n^3 z} \bigg) \bigg(\mathbb{I}+\dfrac{C_n^{(i)}(z)-C_n^{(i)}(0)}{n^3 z}-\dfrac{C_n^{(i)}(z)C_n^{(i)}(0)}{n^6 z^2}+O(n^{-1}) \bigg) \notag \\
&=\mathbb{I}-\dfrac{C_n^{(i)}(z)C_n^{(i)}(0)}{n^6 z^2}-\bigg(\dfrac{C_n^{(i)}(z)-C_n^{(i)}(0)}{n^3 z} \bigg)^2+\frac{(C_n^{(i)}(z)-C_n^{(i)}(0))C_n^{(i)}(z)C_n^{(i)}(0)}{n^9z^3}+O(n^{-1})
\notag \\
&=\mathbb{I}+\dfrac{C_n^{(i)}(0)C_n^{(i)}(z)}{n^6 z^2}-\frac{C_n^{(i)}(0)C_n^{(i)}(z)C_n^{(i)}(0)}{n^9z^3}+O(n^{-1}).
\end{align}
The last equality follows using \eqref{eq: Arno eq 1}, \eqref{eq: Arno eq 2}, and \eqref{eq: Arno eq 3} to find $\bigg(\dfrac{C_n^{(i)}(z)}{n^3 z} \bigg)^2=\dfrac{[C_n^{(i)}(z)]^2}{n^6 z^2}  = \dfrac{O(n^2)}{n^3}= O(n^{-1})$, $\bigg(\dfrac{C_n^{(i)}(0)}{n^3 z} \bigg)^2 = O(n^{-1})$, $\dfrac{C_n^{(i)}(z)^2C_n^{(i)}(0)}{n^9z^3}=\dfrac{O(n^3)}{n^{9/2}} = O(n^{-3/2})$ and \eqref{eq: Second iteration} follows since $\dfrac{C_n^{(i)}(0)C_n^{(i)}(z)C_n^{(i)}(0)}{n^9z^3}=\dfrac{O(n^4)}{n^{9/2}} = O(n^{-1/2})$ and $\dfrac{C_n^{(i)}(0)C_n^{(i)}(z)}{n^6 z^2}=\dfrac{O(n^3)}{n^{3}} = O(1)$.
\end{proof}
We introduce new functions $C_n^{(ii)}(z)$ and $C_n^{(iii)}(z)$ to rewrite \eqref{eq:second iteration proof} in a convenient form. Let
\begin{align}
&C_n^{(ii)}(z) = \dfrac{C_n^{(i)}(0)(C_n^{(i)}(z)-C_n^{(i)}(0))}{z} = O(n^{9/2}). \label{eq: C n 2 z}\\
&C_n^{(iii)}(z) = \dfrac{(C_n^{(i)}(z)-C_n^{(i)}(0))C_n^{(i)}(z)(C_n^{(i)}(0)-C_n^{(i)}(z))}{z^2} = O(n^{7}). \label{eq: C n 3 z}
\end{align}
It is easy to see that they both are analytic near the origin. From similar arguments as at the end of the proof of Lemma \ref{lm: Second iteration}, these new functions can be re-written as
\begin{align*}
&\dfrac{C_n^{(i)}(0)C_n^{(i)}(z)}{n^6 z^2} = \dfrac{C_n^{(ii)}(z)}{n^6 z}+ \dfrac{C_n^{(i)}(0)^2}{n^6 z^2} = \dfrac{C_n^{(ii)}(z)}{n^6 z}+O(n^{-1}).\\
&-\frac{C_n^{(i)}(0)C_n^{(i)}(z)C_n^{(i)}(0)}{n^9z^3}=\dfrac{C_n^{(iii)}(z)}{n^9 z}+\frac{C_n^{(i)}(z)^3-C_n^{(i)}(z)^2 C_n^{(i)}(0)-C_n^{(i)}(0) C_n^{(i)}(z)^2}{n^9z^3}\\
&=\dfrac{C_n^{(iii)}(z)}{n^9 z}+O(n^{-3/2}).
\end{align*}
Inserting these estimates into \eqref{eq:second iteration proof}, we get
\begin{align}\label{eq: second iteration end}
P^{(ii)}_+(z) P^{(ii)}_-(z)^{-1} = \mathbb{I}+\dfrac{C_n^{(ii)}(z)}{n^6 z}+\dfrac{C_n^{(iii)}(z)}{n^9 z}+O(n^{-1})
\end{align}
uniformly valid for $z \in \partial D(0,r_n)$ as $n \to \infty$.

\subsubsection{Third Iteration} In a format already familiar to the reader, we define
\begin{definition}
 \begin{align*}
& E_n^{(iii)}(z) = \bigg(\mathbb{I}-\frac{C_n^{(ii)}(z)-C_n^{(ii)}(0)}{n^6 z} \bigg) E_n^{(ii)}(z), \\
&E_n^{\infty(iii)}(z) = \bigg(\mathbb{I}-\frac{C_n^{(ii)}(0)}{n^6 z} \bigg)^{-1} E_n^{\infty(ii)}(z),
 \end{align*}   
and
\begin{align*}
P^{(iii)}(z) = \begin{cases}
    \bigg(\mathbb{I}-\dfrac{C_n^{(ii)}(z)-C_n^{(ii)}(0)}{n^6 z} \bigg) P^{(ii)}(z), & z \in D(0,r_n),\\
   \bigg(\mathbb{I}-\dfrac{C_n^{(ii)}(0)}{n^6 z} \bigg)^{-1} P^{(ii)}(z), & z \in A(0,r_n,r_0).
\end{cases}    
\end{align*}
\end{definition}

\begin{lemma}
On $\partial D(0,r_n)$, we have the jump
\begin{align}\label{eq: third iteration}
P^{(iii)}_+(z) = (\mathbb{I}+O(n^{-1/2})) P^{(iii)}_-(z).
\end{align}
\end{lemma}
\begin{proof}
It can easily be seen from \eqref{eq: second iteration end} that 
\begin{align}\label{eq:third iteration proof}
&P^{(iii)}_+(z) P^{(iii)}_-(z)^{-1} \notag\\
&= \bigg(\mathbb{I}-\dfrac{C_n^{(ii)}(z)-C_n^{(ii)}(0)}{n^6 z} \bigg) \bigg(\mathbb{I}+\dfrac{C_n^{(ii)}(z)}{n^6 z}+\dfrac{C_n^{(iii)}(z)}{n^9 z}+O(n^{-1}) \bigg) \bigg(\mathbb{I}-\dfrac{C_n^{(ii)}(0)}{n^6 z} \bigg) \notag \\
&=\bigg(\mathbb{I}+\dfrac{C_n^{(ii)}(0)}{n^6 z}+\dfrac{C_n^{(iii)}(z)}{n^9 z}+O(n^{-1}) \bigg) \bigg(\mathbb{I}-\dfrac{C_n^{(ii)}(0)}{n^6 z} \bigg) \notag \\
&=\mathbb{I}+\dfrac{C_n^{(iii)}(z)}{n^9 z}+O(n^{-1}).
\end{align}
Since $[C_n^{(ii)}(z)]^2$, $[C_n^{(ii)}(0)]^2$ and $C_n^{(ii)}(0)C_n^{(ii)}(z)$ has 4 factors of $C_n^{(i)}(z)$ and $C_n^{(ii)}(z)C_n^{(iii)}(z)$ and $C_n^{(ii)}(0)C_n^{(iii)}(z)$ has 5 factors of $C_n^{(i)}(z)$. We have the estimates $\bigg(\dfrac{C_n^{(ii)}(z)}{n^6 z} \bigg)^2= \dfrac{O(n^{5})}{n^{12}z^4} = \dfrac{O(n^{5})}{O(n^{6})}=O(n^{-1})$, $\dfrac{C_n^{(ii)}(0)C_n^{(ii)}(z)}{n^{12} z^2} =O(n^{-1})$, $\dfrac{C_n^{(ii)}(z)C_n^{(iii)}(z)}{n^{15} z^2}= \dfrac{O(n^{6})}{n^{15}z^5}=\dfrac{O(n^{6})}{O(n^{15/2})}=O(n^{-3/2}) $, and $\dfrac{C_n^{(ii)}(0)C_n^{(iii)}(z)}{n^{15} z^2} =O(n^{-3/2})$ for the penultimate equality and $\bigg(\dfrac{C_n^{(ii)}(0)}{n^6 z} \bigg)^2=O(n^{-1})$ for the last equality. We further have $\dfrac{C_n^{(iii)}(z)}{n^9 z} = O(n^{-1/2})$ which yields \eqref{eq: third iteration}.
\end{proof}

\subsubsection{Final iteration}
As a final iteration, we define 
\begin{definition}
 \begin{align*}
& E_n^{(iv)}(z) = \bigg(\mathbb{I}-\frac{C_n^{(iii)}(z)-C_n^{(iii)}(0)}{n^9 z} \bigg) E_n^{(iii)}(z), \\
&E_n^{\infty(iv)}(z) = \bigg(\mathbb{I}-\frac{C_n^{(iii)}(0)}{n^9 z} \bigg)^{-1} E_n^{\infty(iii)}(z),
 \end{align*}   
and
\begin{align*}
P^{(iv)}(z) = \begin{cases}
    \bigg(\mathbb{I}-\dfrac{C_n^{(iii)}(z)-C_n^{(iii)}(0)}{n^9 z} \bigg) P^{(iii)}(z), & z \in D(0,r_n),\\
   \bigg(\mathbb{I}-\dfrac{C_n^{(iii)}(0)}{n^9 z} \bigg)^{-1} P^{(iii)}(z), & z \in A(0,r_n,r_0).
\end{cases}    
\end{align*}
\end{definition}

\begin{lemma}
On $\partial D(0,r_n)$, we have the jump
\begin{align}\label{eq: final iteration}
P^{(iv)}_+(z) = (\mathbb{I}+O(n^{-1})) P^{(iv)}_-(z).
\end{align}
\end{lemma}
\begin{proof}
It can easily be seen from \eqref{eq:third iteration proof} that 
\begin{align*}
&P^{(iv)}_+(z) P^{(iv)}_-(z)^{-1} \notag\\
&= \bigg(\mathbb{I}-\dfrac{C_n^{(iii)}(z)-C_n^{(iii)}(0)}{n^9 z} \bigg) \bigg(\mathbb{I}+\dfrac{C_n^{(iii)}(z)}{n^9 z}+O(n^{-1}) \bigg) \bigg(\mathbb{I}-\dfrac{C_n^{(iii)}(0)}{n^9 z} \bigg) \notag \\
&=\bigg(\mathbb{I}-\dfrac{C_n^{(iii)}(z)-C_n^{(iii)}(0)}{n^9 z} \bigg) \bigg(\mathbb{I}+\dfrac{C_n^{(iii)}(z)-C_n^{(iii)}(0)}{n^9 z}+O(n^{-1}) \bigg) \notag \\
&=\mathbb{I}-\bigg(\dfrac{C_n^{(iii)}(z)-C_n^{(iii)}(0)}{n^9 z} \bigg)^2+O(n^{-1}).
\end{align*}
Since $C_n^{(iii)}(z)C_n^{(iii)}(0)$ has 6 factors of $C_n^{(i)}(z)$, therefore we have used $\dfrac{C_n^{(iii)}(z)}{n^9 z} \dfrac{C_n^{(iii)}(0)}{n^9 z} =O(n^{-2}) $ in the penultimate equality. For the same reason, we further have \break $\bigg(\dfrac{C_n^{(iii)}(z)-C_n^{(iii)}(0)}{n^9 z}\bigg)^2 = O(n^{-2})$ in the last equality, which yields \eqref{eq: final iteration}.
\end{proof}
We are now ready to fix the definitions of $E_n^0(z)$ and $E_n^\infty(z)$.
\begin{definition}\label{Def: En 0 and En inf}
We set $E_n^0(z)=E_n^{(iv)}(z)$ and $E_n^\infty(z)=E_n^{\infty(iv)}(z)$, i.e.,
\begin{align*}
&E_n^0(z)= \bigg(\mathbb{I}-\dfrac{C_n^{(iii)}(z)-C_n^{(iii)}(0)}{n^9 z} \bigg) \bigg(\mathbb{I}-\frac{C_n^{(ii)}(z)-C_n^{(ii)}(0)}{n^6 z} \bigg) \bigg(\mathbb{I}-\frac{C_n^{(i)}(z)-C_n^{(i)}(0)}{n^3 z} \bigg) E(z),\\
&E_n^\infty(z) = \bigg(\mathbb{I}-\frac{C_n^{(iii)}(0)}{n^9 z} \bigg)^{-1} \bigg(\mathbb{I}-\frac{C_n^{(ii)}(0)}{n^6 z} \bigg)^{-1} \bigg(\mathbb{I}-\frac{C_n^{(i)}(0)}{n^3 z} \bigg)^{-1} ,
\end{align*}
where $C_n^{(i)}(z)=E(z) C(z) [E(z)]^{-1}$ and $C_n^{(ii)}(z)$ and $C_n^{(iii)}(z)$ are defined in terms of $C_n^{(i)}(z)$ via \eqref{eq: C n 2 z} and \eqref{eq: C n 3 z}, respectively and $C(z)$ and $E(z)$ are defined by \eqref{eq: C(z)} and \eqref{eq: E(z)}.
\end{definition}


\subsection{Local parametrix near the edge point $1$}\label{subsec:Local parametrix near 1}
Consider a disc $D_1$ centered at $1$ of radius $\delta$. We look for a local parametrix $P^{(1)}$ within the disc $D_1$ such that:
\begin{itemize}
\item $P^{(1)}$ is analytic in $D_1 \backslash (\Delta \cup \Delta^\pm)$.
\item $P^{(1)}$ satisfies the following jump conditions:
\begin{align*}
P^{(1)}_+ = P^{(1)}_-J_P^{(1)}, \quad z \in (\Delta \cup \Delta^\pm) \cap D_1,
\end{align*}
where
\begin{align}
&J_{P^{(1)}}= \begin{pmatrix}
      0& w_1 &0\\
        -\frac{1}{w_1} &0 & 0 \\
         0& 0& 1 \\  
    \end{pmatrix}, \quad z \in \Delta \cap D_1, \label{1st Jump of P1} \\
&J_{P^{(1)}} = \begin{pmatrix}
      1& 0 &0\\
        \frac{1}{w_1}e^{n(H_1 - H_0)} &1 & 0\\
         0& 0& 1 \\  
    \end{pmatrix}, \quad z \in \Delta^\pm \cap D_1. \label{2nd Jump of P1}
\end{align}
\item On $D_1$ we have that as $n \rightarrow \infty$, $P^{(1)}$ matches $N$ in the sense
\begin{align}\label{Matching condition local para at 1}
P^{(1)}(z) = (\mathbb{I}+O (1/n)) N(z), \quad {\rm uniformly ~for}~  z \in \partial D_1.
\end{align}
\item $P^{(1)}(z)$ behaves similar to $T(z)$ near the $1$.
\end{itemize}

Let the function $W(z) = [z^{\alpha_1} (1-z)^\beta]^{1/2}$ be analytic for $z \in \mathbb{C} \backslash (-\infty,1]$. Thus, we have
\begin{align*}
W(z) = \begin{cases}
e^{\beta \pi i/2}z^{\alpha_1/2} (1-z)^{\beta/2 },& {\rm for~ Im~}~z > 0, \\
e^{-\beta \pi i/2}z^{\alpha_1/2} (1-z)^{\beta/2 },& {\rm for~ Im~}~z < 0 .
\end{cases}
\end{align*}
Precisely,
\begin{align}\label{Jump relation W}
W_+(z) W_-(z) = z^{\alpha_1} (1-z)^\beta = w_1, \quad z \in \Delta.
\end{align}
We seek $P^{(1)}$ in the form 
\begin{align}\label{P^1 = E_n}
P^{(1)}(z) = E_n^{(1)}(z)\tilde{P}^{(1)}(z) \begin{pmatrix}
      W^{-1}(z)e^{(n/2)(H_1 - H_0)}& 0 &0\\
        0 &W(z)e^{-(n/2)(H_1 - H_0)} & 0\\
         0& 0& 1 \\  
    \end{pmatrix},
\end{align}
where $E_n^{(1)}(z)$ is the analytic factor defined in the neighborhood of $D_1$. Then, it is easy to verify that $\tilde{P}^{(1)}$ satisfied following RH problem:
\begin{itemize}
\item $\tilde{P}^{(1)}$ is analytic in $D_1 \backslash \Delta \cup \Delta^\pm)$.
\item $\tilde{P}^{(1)}$ satisfies the following jump conditions:
\begin{align*}
\tilde{P}^{(1)}_+ = \tilde{P}^{(1)}_-\tilde{P}^{(1)}, \quad z \in (\Delta \cup \Delta^\pm) \cap D_1,
\end{align*}
where
\begin{align}
&J_{\tilde{P}^{(1)}} = \begin{pmatrix}
      0& 1 &0\\
        -1 &0 & 0\\
         0& 0& 1 \\  
    \end{pmatrix}, \quad z \in \Delta \cap D_1, \label{1st Jump of P tilde 1}\\
&J_{\tilde{P}^{(1)}} = \begin{pmatrix}
      1& 0 &0\\
        e^{\pm \beta \pi i}&1 & 0\\
         0& 0& 1 \\  
    \end{pmatrix}, \quad z \in \Delta^\pm \cap D_1. \label{2nd Jump of P tilde 1}
\end{align}
\item $\tilde{P}^{(1)}$ behaves near the end point $1$ like
\begin{align*}
&\tilde{P}^{(1)}(z)= O\begin{pmatrix}
v_1(z) & v_2(z) & v_2(z) \\
v_1(z) & v_2(z) & v_2(z) \\
v_1(z) & v_2(z) & v_2(z)
\end{pmatrix} ~ z \rightarrow 1,
\end{align*}
where
\begin{align*}
v_1(z),v_2(z) = \begin{cases}
  |z-1|^{\beta/2},|z-1|^{\beta/2},  & \beta < 0,\\
   \log|z-1|,\log|z-1|, & \beta = 0,\\
    |z-1|^{-\beta/2},|z-1|^{-\beta/2}, &\beta > 0, z~{\rm inside~ the ~lens},\\
     |z-1|^{\beta/2},|z-1|^{-\beta/2} & \beta > 0, z~{\rm outside~ the ~lens}.
\end{cases} 
\end{align*}
\end{itemize}
The jumps \eqref{1st Jump of P tilde 1} and \eqref{2nd Jump of P tilde 1} are computed using the jump relations \eqref{1st Jump of P1} and \eqref{2nd Jump of P1} for ${P}^{(1)}$ and the relation \eqref{Jump relation W} for $W(z)$. The behavior of $\tilde{P}^{(1)}$ near $1$ depends on the behavior of ${P}^{(1)}$ near $1$ along with the behavior of $W(z)$ near $1$ which is like $|z-1|^{\beta/2}$ and the fact that modulus of $H_1- H_0$ remains bounded in the neighborhood of $1$.

\begin{center}
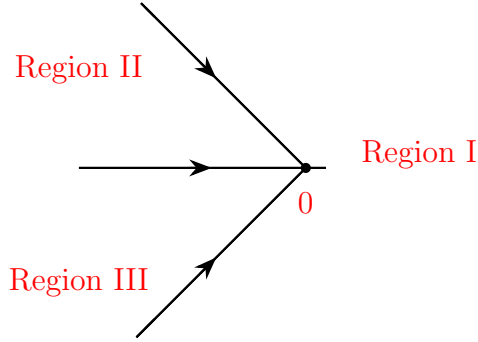
\begin{figure}[htbp]
\begin{tikzpicture}[thick]
\filldraw (-3,2)node[below=10pt,red]{Region II};
\filldraw (-3,-1)node[below=5pt,red]{Region III};
\filldraw (1.5,0.7)node[below=5pt,red]{Region I};
\filldraw (0,0)node[below=5pt,red]{$0$};

\foreach \angle in {135,180,225} {
    \draw (0,0) -- (\angle:3);  
}

\filldraw (0,0) circle (1.5pt);
\draw[arc arrow=to pos .5 with length 3mm] (-2,2) to[out=135,in=135] (0,0);
\draw[arc arrow=to pos .5 with length 3mm] (-2,-2) to[out=225,in=225] (0,0);
\draw[arc arrow=to pos .5 with length 3mm] (-3,0) to[out=0,in=0] (0,0);
\end{tikzpicture}
\caption{The three regions of Bessel parametrix}
\label{Fig: bessel regions}
\end{figure}
\end{center}

The structure of the jump matrices simplifies the Riemann-Hilbert problem to a $2 \times 2$ case, as only the upper $2 \times 2$ block contains non-trivial entries. For hard edge scaling limits, the local parametrix construction typically employs modified Bessel functions as fundamental building blocks \cite[Section 6]{Kuijlaars_McLaughlin_Assche_2004}. Let $I_\beta$ and $K_\beta$ represent the modified Bessel functions of the first and second kind, respectively, of order $\beta$. Additionally, $H^{(1)}_\beta$ and $H^{(2)}_\beta$ denote the Hankel functions of the first and second kind for order $\beta$ \cite[Chapter 9]{Abramowitz_Stegun_1964}. Following \cite[Section 6]{Kuijlaars_McLaughlin_Assche_2004}, we define $\Psi_1$ by
\begin{align}\label{Bessel parametrix in different regions}
&\Psi_1(\eta)=\begin{pmatrix}
   I_\beta(2\eta^{1/2}) & (i/\pi)K_\beta(2\eta^{1/2}) & 0\\
   2\pi iI'_\beta(2\eta^{1/2}) & -2\eta^{1/2}K'_\beta(2\eta^{1/2}) & 0\\
   0 & 0&1
\end{pmatrix}, \quad \eta \in {\rm Region~ I}, \notag\\
&\Psi_1(\eta)=\begin{pmatrix}
   \frac{1}{2}H^{(1)}_\beta(2(-\eta)^{1/2}) & \frac{1}{2}H^{(2)}_\beta(2(-\eta)^{1/2}) & 0\\
   \pi \eta^{1/2} (H^{(1)}_\beta)'(2(-\eta)^{1/2}) & \pi \eta^{1/2}(H^{(2)}_\beta)'(2(-\eta)^{1/2}) & 0\\
   0 & 0&1
\end{pmatrix}e^{(1/2)\beta \pi i \sigma_3}, \quad \eta \in {\rm Region~ II}, \notag\\
&\Psi_1(\eta)=\begin{pmatrix}
   \frac{1}{2}H^{(2)}_\beta(2(-\eta)^{1/2}) & -\frac{1}{2}H^{(1)}_\beta(2(-\eta)^{1/2}) & 0\\
   -\pi \eta^{1/2} (H^{(2)}_\beta)'(2(-\eta)^{1/2}) & \pi \eta^{1/2}(H^{(1)}_\beta)'(2(-\eta)^{1/2}) & 0\\
   0 & 0&1
\end{pmatrix} e^{-(1/2)\beta \pi i \sigma_3}, \quad \eta \in {\rm Region~ III},
\end{align}
where $\eta^{1/2}$ has a branch cut on $(-\infty,0]$ and $\sigma_3 = diag(-1,1,0)$. From the definition \eqref{eq: H012 at 1}, we have
\begin{align*}
H_1-H_0 = 4b_1(z-1)^{1/2}+O(z-1), \quad z \to 1.
\end{align*}
The function $f_1(z)$ is then defined as
\begin{align*}
 f_1(z) = \left[\frac{1}{2} (H_1-H_0)(z)\right]^2.
\end{align*}
Since $f'_1(1) \neq 0$, we can choose $\delta$ small so that the function $f_1$ is a conformal map from $D_1$ onto a convex neighborhood of $0$. Further, we may deform the contours $\Delta^\pm$ near $1$ in such a way that $f$ maps $\Delta^\pm \cap D_1$ to the rays with angles $\pm 2\pi / 3$, respectively. Then, $\Delta^\pm$ and $\Delta$ divide the disk $D_1$ into three regions whose images under the map $f_1$ are the regions I, II and III as shown in Figure \ref{Fig: bessel regions}.
Further, one can show that $\tilde{P}^{(1)}(z)=\Psi_1(n^2 f_1(z))$ is the solution of the RH-problem for $\tilde{P}^{(1)}$ following the procedure proposed in \cite[Theorem 6.3]{Kuijlaars_McLaughlin_Assche_2004}. Then, the RH-problem for ${P}^{(1)}$ is satisfied by ${P}^{(1)}$ given by \eqref{P^1 = E_n} if we define
\begin{align}\label{E n 1 of bessel parametrix}
& E_n^{(1)}(z) = \frac{1}{\sqrt{2}} N(z) \begin{pmatrix}
    W(z) & 0& 0\\
    0 & W^{-1}(z) & 0\\
     0&0 &1
 \end{pmatrix}  \begin{pmatrix}
    1 & i& 0\\
    i & 1 & 0\\
     0&0 & \sqrt{2}
 \end{pmatrix} \notag\\ 
&\hspace{5cm} \times \begin{pmatrix}
    (2\pi n)^{1/2} (f_1(z))^{1/4}  & 0& 0\\
    0 & (2\pi n)^{-1/2} (f_1(z))^{-1/4}  & 0\\
     0&0 & 1
 \end{pmatrix}   ,
\end{align}
where $(f_1(z))^{1/4}$ is defined with a cut along $\Delta$. Hence, it is easy to verify that $E_n^{(1)} $ is analytic in $D_1$ and the matching condition for ${P}^{(1)}$ holds. Thus, the construction of the local parametrix near the edge point $1$ is complete.

\subsection{Local parametrix at $x^*$}\label{subsec:Local parametrix at $x^*$}
Let $D_{x^*}$ be a small disk centered at $x^*$ of radius $r > 0$. We look for a local parametrix $P^{(x^*)}$ inside $D_{x^*}$ such that 
\begin{itemize}
\item $P^{(x^*)}$ is analytic in $D_{x^*} \backslash (\Gamma^* \cup \Gamma^{*\pm})$.
\item $P^{(x^*)}$ satisfies the following jump conditions:
\begin{align*}
P^{(x^*)}_+ = P^{(x^*)}_-J_P^{(x^*)}, \quad z \in (\Gamma \cup \Gamma^{*\pm}) \cap D_{x^*},
\end{align*}
where
\begin{align}
&J_{P^{(x^*)}}= \begin{pmatrix}
      1& 0 &0\\
        0 &0 & c_\alpha\rho\\
         0&-\frac{1}{c_\alpha\rho} & 0 \\  
    \end{pmatrix}, \quad z \in \Gamma^* \cap D_{x^*}, \label{1st Jump of Px*} \\
&J_{P^{(x^*)}} = \begin{pmatrix}
      1& 0 &0\\
        0 &1 & 0\\
         0& \frac{1}{c_\alpha\rho}e^{n(H_2 - H_1)}& 1 \\  
    \end{pmatrix}, \quad z \in \Gamma^{*\pm} \cap D_{x^*}, \label{2nd Jump of Px*}\\
&J_{P^{(x^*)}} = \begin{pmatrix}
      1& 0 &0\\
        0 &1 & c_\alpha\rho e^{n(H_1^- - H_2^+)}\\
         0& 0& 1 \\  
    \end{pmatrix}, \quad z \in (-\infty, x^*) \cap D_{x^*}. \label{3nd Jump of Px*}
\end{align}
\item On $D_{x^*}$ we have that as $n \rightarrow \infty$, $P^{(x^*)}$ matches N in the sense
\begin{align}\label{Matching condition local para at x*}
P^{(x^*)}(z) = (\mathbb{I}+O (1/n)) N(z), \quad {\rm uniformly ~for}~  z \in \partial D_{x^*}.
\end{align}
\end{itemize}
From \eqref{eq: measure behave at x*}, it is evident that the measure vanishes like a square root at point $x^*$. From the definition \eqref{eq: H012 at x*} of $H_1$ and $H_2$, we obtain
\begin{align*}
H_2-H_1 = \pm 2\pi i+ \frac{4 b_{x^*}}{3} (z-x^*)^{3/2}+O(z-x^*)^{2}, \quad z \to x^*.
\end{align*}
The function $f_*(z)$ is then defined such that
\begin{align*}
    f_*(z) = \left[\frac{3}{4} ((H_2-H_1)(z) \mp 2\pi i)\right]^{2/3}
\end{align*}
and is analytic at $x^*$, real-valued on the real axis near $x^*$ and $f'_
*(x^*) \neq 0$. Hence, $f_*(z)$ is a conformal map of $D_{x^*}$ onto a convex neighborhood of $0$, if $r$ is sufficiently
small. We can deform the lenses $\Gamma^{*\pm}$ so that they are mapped by $f_*(z)$ onto the rays in the complex plane of constant arguments $-2\pi/3$ and $2\pi/3$. Thus, $\Gamma^{*\pm}$ and $\mathbb{R}$ divides the disk $D_{x^*}$ into 4 regions whose images by the conformal map $f_*(z)$ are contained in the four regions I, II, III, and IV as shown in Figure \ref{Fig: Airy regions}.
\begin{center}
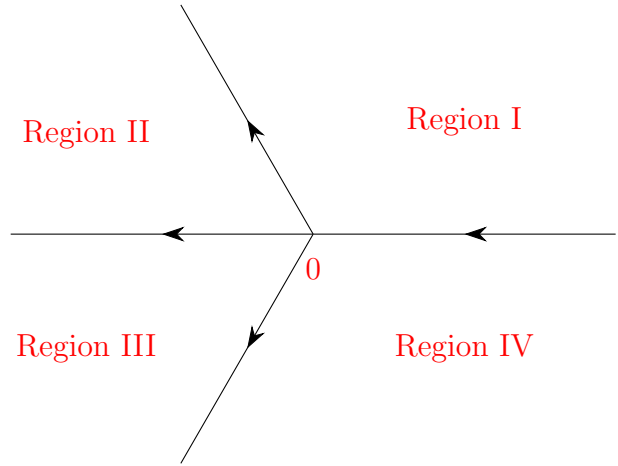
\begin{figure}[htbp]
\begin{tikzpicture}[
    >={Stealth[length=3mm, width=2mm]},
    decoration={
        markings,
        mark=at position 0.5 with {\arrow{>}}  
    }
]
\filldraw (-3,2)node[below=10pt,red]{Region II};
\filldraw (-3,-1)node[below=5pt,red]{Region III};
\filldraw (2,2)node[below=5pt,red]{Region I};
\filldraw (2,-1)node[below=5pt,red]{Region IV};
\filldraw (0,0)node[below=5pt,red]{$0$};

\draw[postaction={decorate}] (4,0) -- (0,0);   
\draw[postaction={decorate}] (0,0) -- (-4,0);    


\draw[postaction={decorate}] (0,0) -- (120:3.5);   
\draw[postaction={decorate}] (0,0) -- (240:3.5);   

\end{tikzpicture}
\caption{The four regions of Airy parametrix}
\label{Fig: Airy regions}
\end{figure}
\end{center}
Since only the $2 \times 2$ lower block of the jump matrices is non‑trivial, the Riemann–Hilbert problem effectively reduces to a $2 \times 2$ problem. The jumps are of a standard form, and the local parametrix can therefore be constructed using Airy functions \cite{Deift_book_1999}. Let us define the function $\Phi(s)$ as
\begin{align*}
\Phi(\zeta) = \begin{cases}
    \begin{pmatrix}
      1  &0 & 0\\
      0  & Ai(\zeta) & -\omega_\zeta^2Ai(\omega_\zeta^2\zeta) \\
       0 & Ai'(\zeta)& -\omega_\zeta Ai'(\omega_\zeta^2\zeta)
    \end{pmatrix}, & s \in I, \\
    \begin{pmatrix}
      1  &0 & 0\\
      0  & -\omega_\zeta Ai(\omega_\zeta\zeta) & -\omega_\zeta^2Ai(\omega_\zeta^2\zeta) \\
       0 & -\omega_\zeta^2 Ai'(\omega_\zeta\zeta)& -\omega_\zeta Ai'(\omega_\zeta^2\zeta)
    \end{pmatrix}, & s \in II, \\
    \begin{pmatrix}
      1  &0 & 0\\
      0  & -\omega_\zeta^2 Ai(\omega_\zeta^2\zeta) & \omega_\zeta Ai(\omega_\zeta\zeta) \\
       0 & -\omega_\zeta Ai'(\omega_\zeta^2 \zeta)& \omega_\zeta^2 Ai'(\omega_\zeta\zeta)
    \end{pmatrix}, & s \in III, \\
    \begin{pmatrix}
      1  &0 & 0\\
      0  & Ai(\zeta) & \omega_\zeta Ai(\omega_\zeta \zeta) \\
       0 & Ai'(\zeta)& \omega_\zeta^2 Ai'(\omega_\zeta\zeta)
    \end{pmatrix}, & s \in IV, 
\end{cases}
\end{align*}
where $\omega_\zeta = e^{2\pi i/3}$. We then choose $P^{(x^*)}$ in the form
\begin{align}\label{eq: P^x* = E_n*}
P^{(x^*)}(z)=E_n^*(z) \Phi(n^{3/2}f_*(z)) \diag [1, (c_\alpha\rho)^{-1/2}e^{n(H_2 - H_1)/2}, (c_\alpha\rho)^{1/2}e^{-n(H_2 - H_1)/2} ].
\end{align}
This $P^{(x^*)}(z)$ can be shown to satisfy the jump conditions \eqref{1st Jump of Px*}, \eqref{2nd Jump of Px*}, and \eqref{3nd Jump of Px*} for any analytic prefactor $E_n^*(z)$, which will now be chosen to satisfy the matching condition \eqref{Matching condition local para at x*} on $\partial D_{x^*}$. Hence, we have
\begin{align}\label{eq: En*z}
& E_n^{(*)}(z) = \sqrt{\pi} N(z) \begin{pmatrix}
    1 & 0& 0\\
    0 & (c_\alpha\rho)^{1/2} & 0\\
     0&0 &(c_\alpha\rho)^{-1/2}
 \end{pmatrix}  \begin{pmatrix}
 1&0 & 0\\
    0 & 1& -1\\
    0 & -i & -i
\end{pmatrix} \notag\\ 
&\hspace{5cm} \times \begin{pmatrix}
    1  & 0& 0\\
    0 &  n^{1/6} (f_*(z))^{1/4} & 0\\
     0&0 & n^{-1/6} (f_*(z))^{-1/4}
 \end{pmatrix}  ,
\end{align}
where $(f_*(z))^{1/4}$ is defined with a cut along $\Gamma$. This completes the construction of local parametrix $P^{(x^*)}(z)$ in the neighborhood $D_{x^*}$ of $x^*$.

\subsection{Final transformation}
Having the global parametrix $N$, and the local paramertrices ${P}^{(1)}$ and $P^{(x^*)}$ in the neighbourhood of disks $D(1,\delta)$ and $D(x^{\ast},\delta)$, respectively, and $P$ in the neighbourhood of disk $D(0,r_0)$, we define the final transformation as follows:
\begin{definition}
We define $R$ as
\begin{align}\label{Transformation R}
R(z) = \begin{cases}
 T(z) N(z)^{-1},   & z \in \mathbb{C} \backslash \{\Sigma_T \cup \overline{D(1,\delta)} \cup \overline{D(0,r_0)} \cup \overline{D(x^{\ast},\delta)}\},\\
  T(z) P(z)^{-1},  & z \in D(0,r_0) \backslash \{\Sigma_T \cup \partial D(0,r_n) \}, \\
  T(z) P^{(1)}(z)^{-1},  &  z \in D(1,\delta) \backslash \Sigma_T, \\
  T(z) P^{(x^{
  \ast
  })}(z)^{-1},  &  z \in D(x^{\ast},\delta) \backslash \Sigma_T,
\end{cases}
\end{align}
where $\Sigma_T = (-\infty,1] \cup \Delta^\pm \cup \Gamma_*^\pm$.
\end{definition}
\color{black}
Since $T$ and $N$ have the same jumps on $\Delta$ and $\Gamma$, we find that $R$ has no jump across $(-\infty,-r_0)$ and $r_0, 1-\delta$. Also, $T$ and $P^{(1)}$ have the same jumps on $D(1,\delta)$ and so $R$ has no jump in $D(1,\delta)$. Further, $T$ and $P^{(x^*)}$ have the same jumps on $D(x^*,\delta)$ and so $R$ has no jump in $D(x^*,\delta)$. Finally, the jumps of $T$ and $P$ agree inside $D(0,r_0)$ and on $(-r_0,0)$ and $(0,r_0)$. Thus, $R$ solves an RH problem on the contour $\Sigma_R$ depicted in Figure \ref{Final Contour} below and defined as
\begin{multline*}
\Sigma_R = (-\infty,x^*-\delta] \cup \partial D(0,r_0) \cup \partial D(0,r_n) \cup \partial D(1,\delta) \cup D(x^*,\delta) \\ \cup (\Delta^\pm \backslash (D(0,r_0) \cup D(1,\delta))) \cup (\Gamma^\pm \backslash ( D(0,r_0) \cup D(x^*,\delta)).
\end{multline*}

\begin{center}
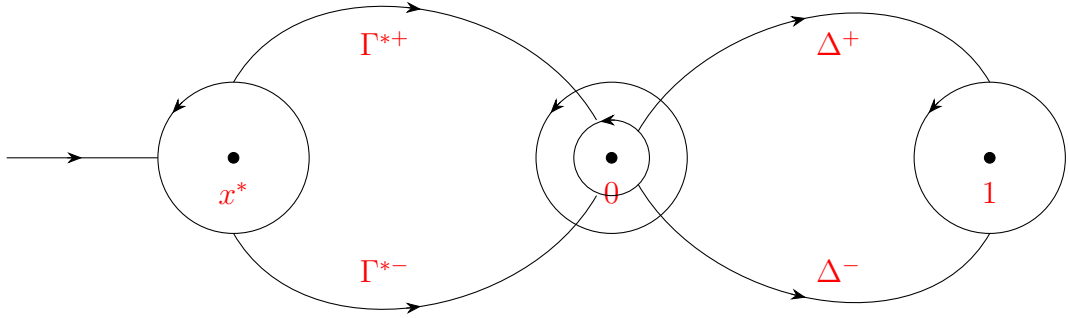
\begin{figure}[htbp]
\begin{tikzpicture}
\filldraw (3,2)node[below=5pt,red]{$\Delta^+$};
\filldraw (3,-1)node[below=5pt,red]{$\Delta^-$};
\filldraw (-3,2)node[below=5pt,red]{$\Gamma^{*+}$};
\filldraw (-3,-1)node[below=5pt,red]{$\Gamma^{*-}$};
\filldraw (0,0)node[below=5pt,red]{$0$} circle [radius=2pt];
\filldraw (5,0)node[below=5pt,red]{$1$} circle [radius=2pt];
\filldraw (-5,0)node[below=5pt,red]{$x^*$} circle [radius=2pt];
 \draw[arc arrow=to pos .5 with length 2mm] (-8,0) to (-6,0);
\draw[arc arrow=to pos .5 with length 2mm] (-5,1) to[out=60,in=120] (-0.2,0.5);
\draw[arc arrow=to pos .5 with length 2mm] (-5,-1) to[out=-60,in=-120] (-0.2,-0.5);
\draw[arc arrow=to pos .5 with length 2mm] (0.35,0.35) to[out=60,in=120] (5,1);
\draw[arc arrow=to pos .5 with length 2mm] (0.35,-0.35) to[out=-60,in=-120] (5,-1);
    \draw[arc arrow=to pos .4 with length 2mm] (0,0) circle (1cm);
    \draw[arc arrow=to pos .3 with length 2mm] (0,0) circle (0.5cm);
    \draw[arc arrow=to pos .4 with length 2mm] (-5,0) circle (1cm);
    \draw[arc arrow=to pos .4 with length 2mm] (5,0) circle (1cm);
\end{tikzpicture}
\caption{The jump contours for the matrix $R$}
\label{Final Contour}
\end{figure}
\end{center} 

\begin{proposition}
As $n \to \infty$, we have
\begin{align}
R_+(z) = R_-(z) \bigg(\mathbb{I}+O\bigg(\frac{1}{n} \bigg) \bigg),& \quad z \in \partial D(0,r_0) \cup \partial D(0,r_n) \cup \partial D(1,\delta) \cup \partial D(x^*,\delta), \label{R match on disks}\\
R_+(z) = R_-(z) (\mathbb{I}+O(e^{-d_2 \sqrt{n}})), & \quad z \in \Delta^\pm ({\rm or~} \Gamma^{*\pm}) \cap A(0;r_n,r_0), \label{R match in annulus}\\
R_+(z) = R_-(z) (\mathbb{I}+O(e^{-d_1n})),& \quad  on~~ remaining~~ parts~~of~~ \Sigma_R, \label{R match on rays}
\end{align}
where $d_1$ and $d_2$ are positive constants. All $O$ terms in \eqref{R match on disks}-\eqref{R match on rays} are uniform for $z$ on the indicated contours. Moreover, we have
\begin{align}\label{R like I}
    R(z) = \mathbb{I}+O\bigg(\frac{1}{n} \bigg), \quad as ~~ n \to \infty,
\end{align}
uniformly for $z \in \mathbb{C} \backslash \Sigma_R$.
\end{proposition}
\begin{proof}
For $z \in \partial D(0,r_n)$, we have by definition \eqref{Transformation R} that
\begin{align*}
R_-(z)^{-1}R_+(z) = [T(z)P_-(z)^{-1}]^{-1} T(z)P_+(z)^{-1} = P_-(z)P_+(z)^{-1} .
\end{align*}
Since $T$ has no jump on $\partial D(0,r_n)$. Then, \eqref{R match on disks} for $z \in \partial D(0,r_n)$ follows from \eqref{double matching of P 2}. Similarly, for $z \in \partial D(0,r_0)$, by \eqref{Transformation R}, we get
\begin{align*}
R_-(z)^{-1}R_+(z) = [T(z)P(z)^{-1}]^{-1} T(z)N(z)^{-1} = P(z)N(z)^{-1},
\end{align*}
and \eqref{R match on disks} follows from \eqref{double matching of P 1} for $z \in \partial D(0,r_n)$. For $z \in \partial D(1,\delta)$ and for $z \in \partial D(x^*,\delta)$, \eqref{R match on disks} follows from the matching condition \eqref{Matching condition local para at 1} and \eqref{Matching condition local para at x*} for the local parametrix at $1$ and $x^*$, respectively. For $z \in \Sigma_R$ outside the disks $D(0,r_0)$, $D(x^*,\delta)$, and $D(1,\delta)$, we have by \eqref{Transformation R},
\begin{align*}
R_-(z)^{-1}R_+(z) = [T_-(z)N(z)^{-1}]^{-1} T_+(z)N(z)^{-1} = N(z) T_-(z)^{-1}T_+(z) N(z)^{-1},
\end{align*}
and $T_-(z)^{-1}T_+(z) = \mathbb{I}+O(e^{-d_1n})$ which follows from the jumps in the RH problem for $T$, since $\Re (H_1-H_0) = 2V^{\lambda_1}-V^{\lambda_2}-\omega_1 \leq 0$ on $\Delta^\pm$ by \eqref{equilibrium problem} and $\Re (H_2-H_1) = V^{\lambda_1}-2V^{\lambda_2}-\omega_2 \leq 0$ on $\Gamma^{*\pm}$ by \eqref{equilibrium problem 2}. For $z \in (-\infty, x^*)$, the asymptotic relation $T_-(z)^{-1}T_+(z) = \mathbb{I}+O(e^{-d_1n})$ follows from the facts: the transformation $T$ coincides with $U$ outside the lenses $\Delta^\pm$ and $\Gamma^{*\pm}$; additionally, $e^{n(\tilde{H}_1^- - \tilde{H}_2^+)}$ is exponentially small, as guaranteed by \eqref{eq: Real H2-H1 +ve} and the condition that $n\theta$ is an integer. Hence, the estimate \eqref{R match on rays} also follows. For $z \in \Delta^\pm \cap A(0;r_n,r_0)$, we have $R(z) = T(z)P(z)^{-1}$. By the jump of $T$ in \eqref{First Jump for T}, we get
\begin{align*}
R_-(z)^{-1}R_+(z) &= [T_-(z)P(z)^{-1}]^{-1} T_+(z)P(z)^{-1} = P(z) T_-(z)^{-1}T_+(z) P(z)^{-1}\\
&=\mathbb{I}+z^{-\alpha_1}e^{n(H_1-H_0)}P(z)E_{21}P(z)^{-1}.
\end{align*}
Since $P(z)$ and $P(z)^{-1}$ have at most power law singularity near the origin, the estimate of $R_-(z)^{-1}R_+(z)$ relies totally on the behavior of $e^{n(H_1-H_0)}$ near $z=0$, we have
\begin{align*}
|e^{n(H_1-H_0)}| = |e^{n\int(h_1-h_0)dz}|=|e^{-c n\int z^{-2/3} dz}| = |e^{-d_2 n z^{1/3}}| \leq e^{-d_2 n |z|^{1/3}}
\end{align*}
for some $c>0$. Since $|z| \geq n^{-3/2}$ on the annulus $A(0;r_n,r_0)$, we have $|e^{n(H_1-H_0)}| \leq e^{-d_2 \sqrt{n}}$ for $z \in \Delta^\pm \cap A(0;r_n,r_0)$. Then, \eqref{R match in annulus} follows from the fact that if $z \in \Gamma^\pm \cap A(0;r_n,r_0)$, then the estimates of $R_-(z)^{-1}R_+(z)$  can be derived in a similar manner, where one needs to explore the behavior of $|e^{n(H_2-H_1)}|$ near $z=0$. Finally, we conclude \eqref{R like I} from the standard arguments in the RH analysis, see for example \cite{Deift_book_1999} and \cite[Appendix A]{Bleher_Arno_2007}.
\end{proof}

\section{Strong asymptotics}
Having known the asymptotic behavior \eqref{R like I} of $R$, we are now ready to state the strong asymptotics results for the Jacobi-Piñeiro polynomials $P_{\vec{n}}$ defined by \eqref{Jacobi Pineiro} by tracing back our steps to the original RH problem for $Y$.

\subsection{Strong asymptotics away from the zeros}
\begin{theorem}\label{Thrm:Asymptotics in unbounded region}
Let $\vec{n}=(n_1,n_2)$ be multi-indices that tend to infinity such that $n_1+n_2=2n$, $n_2/(n_1+n_2) =\theta \in (0,1/2)$ where $\theta$ remains constant with $\theta n \in \mathbb{Z}$. Then, for $\beta=0$, uniformly for $z$ in compact subsets of $\mathbb{C} \backslash [0,1]$, we have 
\begin{align}\label{Asymptotics in unbounded region}
P_{\vec{n}}(z)=\frac{4^{\alpha_1}(\varphi_0(z)-\theta_1)}{[(2+\theta)(4-\theta)]^{\alpha_1}} \dfrac{(\varphi_0(z)-\theta_2)^{\alpha_2-\alpha_1+1}\varphi_0(z)^{2\alpha_1-\alpha_2}}{z^{\alpha_1} \sqrt{\mathcal{D}(\varphi_0(z))}} e^{n\tilde{H_0}(z)}\bigg(1+O\bigg(\frac{1}{n} \bigg) \bigg).
\end{align}
\end{theorem}
\begin{proof}
Let $z \in \mathbb{C} \backslash [0,1]$. From \eqref{Transformation U}, we see that
\begin{align*}
P_{\vec{n}}(z)=Y_{11}(z) = e^{n\tilde{H_0}(z)}U_{11}(z).
\end{align*}
Then, from the definition of T, it is always possible to assume that $z$ does not belong to the lenses around $[0,1]$. Hence, $T=U$ from \eqref{Transformtion T} and we get $T_{11}(z)=U_{11}(z)$. Finally, we note that $T(z)=R(z)N(z)$ that gives
\begin{align*}
  T_{11}(z) = R_{11}(z)N_{11}(z)+R_{12}(z)N_{21}(z)+R_{13}(z)N_{31}(z) .
\end{align*}
Since $R(z)=\mathbb{I}+O(1/n)$ and since $N_{11}(z)$ does not vanish in a neighborhood of $z$, we get
\begin{align*}
 P_{\vec{n}}(z)=   e^{n\tilde{H_0}(z)}N_{11}(z)\bigg(1+O\bigg(\frac{1}{n} \bigg) \bigg)
\end{align*}
uniformly in $\mathbb{C} \backslash [0,1]$. Using the formula for $N_{11}(z)$ from \eqref{Solution matrix N general case}, we get \eqref{Asymptotics in unbounded region}.
\end{proof}

\subsection{Strong asymptotics near the intervals}
\begin{theorem}
Let $\vec{n}=(n_1,n_2)$ be multi-indices that tend to infinity such that $n_1+n_2=2n$, $n_2/(n_1+n_2) =\theta \in (0,1/2)$ where $\theta$ remains constant with $\theta n \in \mathbb{Z}$. Then, for $\beta=0$, uniformly for $z$ on $\pm$ sides of $\Delta$ and $\Gamma$, away from the end points, we have
\begin{align}\label{Asymptotics near intervals}
&P_{\vec{n}}(z)=e^{n\tilde{H_0}(z)}\bigg[\dfrac{4^{\alpha_1}(\varphi_0(z)-\theta_1)(\varphi_0(z)-\theta_2)^{\alpha_2-\alpha_1+1}\varphi_0(z)^{2\alpha_1-\alpha_2}}{[(2+\theta)(4-\theta)]^{\alpha_1}z^{\alpha_1} \sqrt{\mathcal{D}(\varphi_0(z))}}+O\bigg(\frac{1}{n} \bigg)  \bigg] \notag\\
& \pm e^{n(H_1(z) - H_0(z))} \bigg[\dfrac{4^{\alpha_1}(\varphi_1(z)-\theta_1)(\varphi_1(z)-\theta_2)^{\alpha_2-\alpha_1+1} \varphi_1(z)^{2\alpha_1-\alpha_2}}{ [(2+\theta)(4-\theta)]^{\alpha_1} z^{\alpha_1}\sqrt{\mathcal{D}(\varphi_1(z))}}+O\bigg(\frac{1}{n} \bigg)  \bigg].
\end{align}
\end{theorem}
\begin{proof}
We consider $z_0 \in \Delta$, away from the endpoints, and prove that \eqref{Asymptotics near intervals} holds uniformly in a disk $D_{z_0}$ centered at $z_0$. We may choose this disk so that it is disjoint from $D(0,r_0)$ and $D_1(1,\delta)$ and is contained in the lens around $\Delta$ (See Figure \ref{Final Contour}). Then, as in the proof of Theorem \ref{Thrm:Asymptotics in unbounded region}, we have $P_{\vec{n}}(z)= e^{n\tilde{H_0}(z)}U_{11}(z)$. In view of relations \eqref{Transformtion T} between $T$ and $U$, we get
\begin{align*}
P_{\vec{n}}(z)= e^{n\tilde{H_0}(z)}(T_{11}(z) \pm z^{-\alpha_1}(1-z)^{-\beta}e^{n(H_1(z) - H_0(z))}T_{12}(z)),
\end{align*}
where the $+$ sign holds in the upper part of the lens and the $-$ sign holds in the lower part. Since $T(z)=R(z)N(z)$ with $R(z)=\mathbb{I}+O(1/n)$ and the entries of $N$ do not vanish in $D_{z_0}$, we obtain
\begin{align*}
P_{\vec{n}}(z)=e^{n\tilde{H_0}(z)}\bigg[N_{11}(z)+O\bigg(\frac{1}{n} \bigg)  \bigg]\pm z^{-\alpha_1}(1-z)^{-\beta}e^{n(H_1(z) - H_0(z))} \bigg[N_{12}(z)+O\bigg(\frac{1}{n} \bigg)  \bigg],
\end{align*}
where we have also used the fact that $N_{ij}(z)$, $i=1,2,3$, $j=1,2$ are bounded in $D_{z_0}$. Using the formulas for $N_{11}(z)$ and $N_{12}(z)$ from \eqref{Solution matrix N general case}, we get \eqref{Asymptotics near intervals}. The proof when $z_0 \in \Gamma^*$ is similar and hence, we omit the details.
\end{proof}

\subsection{Strong asymptotics near the branch points}
\begin{theorem}
Let $\vec{n}=(n_1,n_2)$ be multi-indices that tend to infinity such that $n_1+n_2=2n$, $n_2/(n_1+n_2) =\theta \in (0,1/2)$ where $\theta$ remains constant with $\theta n \in \mathbb{Z}$. Then, for $\beta=0$, uniformly for $z$ in a small neighborhood $D_1$ of $1$, we have
\begin{align}\label{eq:asymptotics near 1}
P_{\vec{n}}(z)&=\sqrt{\pi n}f_1(z)^{1/4}e^{(n/2)(H_1(z) + H_0(z)-2C_0)} [B_0(z)J_0(2n(-f_1(z))^{1/2})(1+O(n^{-1})) \notag\\
&\hspace{5cm} \mp iB^*_0(z)(J_0)'(2n(-f_1(z))^{1/2})(1+O(n^{-1}))],
\end{align}
where
\begin{align*}
&B_0(z) = \dfrac{(\varphi_0(z)-\theta_2)^{\alpha_2-\alpha_1+1}\varphi_0(z)^{2\alpha_1-\alpha_2}\sqrt{\mathcal{D}(\theta_1)}}{(\theta_1-\theta_2)^{\alpha_2-\alpha_1+1} \theta_1^{2\alpha_1-\alpha_2}z^{\alpha_1}\sqrt{\mathcal{D}(\varphi_0(z))}}  +i\dfrac{(\varphi_1(z)-\theta_2)^{\alpha_2-\alpha_1+1} \varphi_1(z)^{2\alpha_1-\alpha_2}\sqrt{\mathcal{D}(\theta_1)}}{(\theta_1-\theta_2)^{\alpha_2-\alpha_1+1} \theta_1^{2\alpha_1-\alpha_2}z^{\alpha_1}\sqrt{\mathcal{D}(\varphi_1(z))}},  \\
&B^*_0(z)= \dfrac{(\varphi_0(z)-\theta_2)^{\alpha_2-\alpha_1+1}\varphi_0(z)^{2\alpha_1-\alpha_2}\sqrt{\mathcal{D}(\theta_1)}}{(\theta_1-\theta_2)^{\alpha_2-\alpha_1+1} \theta_1^{2\alpha_1-\alpha_2}z^{\alpha_1}\sqrt{\mathcal{D}(\varphi_0(z))}}  -i\dfrac{(\varphi_1(z)-\theta_2)^{\alpha_2-\alpha_1+1} \varphi_1(z)^{2\alpha_1-\alpha_2}\sqrt{\mathcal{D}(\theta_1)}}{(\theta_1-\theta_2)^{\alpha_2-\alpha_1+1} \theta_1^{2\alpha_1-\alpha_2}z^{\alpha_1}\sqrt{\mathcal{D}(\varphi_1(z))}}.
\end{align*}
In \eqref{eq:asymptotics near 1}, the $\mp$ sign represents $z$ in the upper half-plane and lower half-plane, respectively. The functions $f_1(z)^{1/4}$ and $(-f_1(z))^{1/2}$ are, respectively, taken with branch cuts along the real positive and real
negative semi-axis.
\end{theorem}
\begin{proof}
We will use the parametrix $P^{(1)}$ constructed in Subsection \ref{subsec:Local parametrix near 1}. Assume $z \in D_1$ in the upper half region, then from \eqref{Transformation R} and \eqref{P^1 = E_n}, we get
\begin{align*}
T(z)=R(z)P^{(1)}(z)=R(z)E_n^{(1)}(z)\Psi_1(n^2 f_1(z)) \begin{pmatrix}
      W^{-1}(z)e^{\frac{n}{2}(H_1 - H_0)}& 0 &0\\
        0 &W(z)e^{-\frac{n}{2}(H_1 - H_0)} & 0\\
         0& 0& 1 \\  
    \end{pmatrix}.
\end{align*}
Using \eqref{Transformtion T} for region between $\Delta^+$ and $\Delta$ in the above expression, we have

\begin{align}\label{U times the column}
U(z)\begin{pmatrix}
    1 \\
    0 \\
    0
\end{pmatrix}=z^{-\alpha_1/2} (1-z)^{-\beta/2 } e^{(n/2)(H_1 - H_0)}R(z)E_n^{(1)}(z)\Psi_1(n^2 f_1(z)) \begin{pmatrix}
      e^{-\beta \pi i/2}z\\
        e^{\beta \pi i/2}z \\
         0 \\  
    \end{pmatrix}.
\end{align}
In the above-mentioned region, we have $\Psi_1(s)$ given by last expression in \eqref{Bessel parametrix in different regions} and further, using the formulas $9.1.3$ and $9.1.4$ from \cite{Abramowitz_Stegun_1964}, we have
\begin{align}\label{Psi times the column}
\Psi_1(n^2 f_1(z))\begin{pmatrix}
      e^{-\beta \pi i/2}z\\
        e^{\beta \pi i/2}z \\
         0 \\  
    \end{pmatrix} = \begin{pmatrix}
   J_\beta(2n(-f_1(z))^{1/2}) \\
   -2\pi n f_1(z)^{1/2} (J_\beta)'(2n(-f_1(z))^{1/2}) \\
   0 
\end{pmatrix}.
\end{align}
Substituting \eqref{Psi times the column}, \eqref{E n 1 of bessel parametrix}, \eqref{R like I} and \eqref{Solution matrix N general case} in \eqref{U times the column} and calculating the $(1,1)$ entry, we get \eqref{eq:asymptotics near 1}.
\end{proof}

\begin{theorem}
Let $\vec{n}=(n_1,n_2)$ be multi-indices that tend to infinity such that $n_1+n_2=2n$, $n_2/(n_1+n_2) =\theta \in (0,1/2)$ where $\theta$ remains constant with $\theta n \in \mathbb{Z}$. Then, for $\beta=0$, uniformly for $z$ in a small neighborhood $D_{x^*}$ of $x^*$, we have
\begin{align}\label{Asymptotics near x*}
P_{\vec{n}}(z)=\frac{4^{\alpha_1}(\varphi_0(z)-\theta_1)}{[(2+\theta)(4-\theta)]^{\alpha_1}} \dfrac{(\varphi_0(z)-\theta_2)^{\alpha_2-\alpha_1+1}\varphi_0(z)^{2\alpha_1-\alpha_2}}{z^{\alpha_1} \sqrt{\mathcal{D}(\varphi_0(z))}} e^{n\tilde{H_0}(z)}\bigg(1+O\bigg(\frac{1}{n} \bigg) \bigg).
\end{align}
\end{theorem}
\begin{proof}
We will use the parametrix $P^{(x^*)}$ constructed in Subsection \ref{subsec:Local parametrix at $x^*$}. Assume $z \in D_{x^*}$ and $f_*(z)$ lies in region I or IV. Then, from \eqref{Transformation R} and \eqref{eq: P^x* = E_n*}, we get 
\begin{align*}
T(z)=R(z)P^{(x^*)}(z)=R(z)E_n^*(z) \Phi(n^{3/2}f_*(z)) \diag [1, (c_\alpha\rho)^{-1/2}e^{\frac{n}{2}(H_2 - H_1)}, (c_\alpha\rho)^{1/2}e^{-\frac{n}{2}(H_2 - H_1)} ].
\end{align*}
Since $T = U$ in regions I or IV, \eqref{eq: En*z} implies
\begin{align*}
&U(z) = \sqrt{\pi}R(z)N(z)\begin{pmatrix}
    1 & 0& 0\\
    0 & (c_\alpha\rho)^{1/2}n^{1/6} (f_*(z))^{1/4} & -(c_\alpha\rho)^{1/2}n^{-1/6} (f_*(z))^{-1/4}\\
     0&-i(c_\alpha\rho)^{-1/2}n^{1/6} (f_*(z))^{1/4} &-i(c_\alpha\rho)^{-1/2}n^{-1/6} (f_*(z))^{-1/4}
 \end{pmatrix} \times\\
& \begin{pmatrix}
      1  &0 & 0\\
      0  & Ai(n^{3/2}f_*(z)) & \omega_\zeta Ai(\omega_\zeta n^{3/2}f_*(z)) \\
       0 & Ai'(n^{3/2}f_*(z))& \omega_\zeta^2 Ai'(\omega_\zeta n^{3/2}f_*(z))
    \end{pmatrix} \diag [1, (c_\alpha\rho)^{-1/2}e^{\frac{n}{2}(H_2 - H_1)}, (c_\alpha\rho)^{1/2}e^{-\frac{n}{2}(H_2 - H_1)} ],
\end{align*}
which implies
\begin{align*}
U_{11}(z) =   N_{11}(z)\bigg(1+O\bigg(\frac{1}{n} \bigg)\bigg),
\end{align*} 
and hence \eqref{Asymptotics near x*} follows upon substituting above expression in $P_{\vec{n}}(z)= e^{n\tilde{H_0}(z)}U_{11}(z)$ and computing the $(1,1)$ entry.
\end{proof}

We record the following remark for later use.
\begin{remark}\label{lm: Enz  matrix}
From \eqref{eq: E(z)} and \eqref{N hat def} in the upper half-region, we can write
\begin{align*}
E(z)= N(z)   z^{-\frac{A}{3}}  [\mathcal{U}^+]^{-1} H(z)^{-1} \xi^{\frac{M}{3}} = N(z)z^{-\frac{A}{3}} D_n(z)^{-1} [\mathcal{U}^+]^{-1} \xi^{\frac{M}{3}} = N(z)J^0(z).
\end{align*}
Next, using \eqref{Omega and its tilde}, \eqref{Nz behavior at z=0} and \eqref{U plus and U minus}, we get
\begin{align*}
&[\mathcal{U}^+]^{-1}=\begin{pmatrix}
      \omega^{-A_1}& 0 & 0 \\
        0& \omega^{A_1} & 0\\
        0 & 0 & 1 \\  
    \end{pmatrix}
\begin{pmatrix}
      \omega^{2}& -1 & \omega^{-2} \\
        \omega^{-2} & -1 & \omega^{2}\\
        1 & -1 & 1 \\  
    \end{pmatrix} = \begin{pmatrix}
      \omega^{2-A_1}& -\omega^{-A_1} & \omega^{-2-A_1} \\
        \omega^{A_1-2} & -\omega^{A_1} & \omega^{2+A_1}\\
        1 & -1 & 1 \\  
    \end{pmatrix}.
\end{align*}
Then 
\begin{align*}
J^0(z)= \begin{pmatrix}
      \xi^{\frac{1}{3}}\omega^{2-A_1}e^{-nH_0+p_1}& -\omega^{-A_1}e^{-nH_0+p_1} & \xi^{-\frac{1}{3}}\omega^{-2-A_1}e^{-nH_0+p_1} \\
        -\xi^{\frac{1}{3}}\omega^{A_1-2}e^{-nH_1+p_2}   & \omega^{A_1}e^{-nH_1+p_2}  & -\xi^{-\frac{1}{3}}\omega^{2+A_1}e^{-nH_1+p_2} \\
        \xi^{\frac{1}{3}} e^{-nH_2+p_3} & - e^{-nH_2+p_3} & \xi^{-\frac{1}{3}} e^{-nH_2+p_3} \\  
    \end{pmatrix},
\end{align*}
where  $p_1 = (\alpha_1+\alpha_2) \ln z/3 $, $p_2=-(2\alpha_1-\alpha_2)\ln z/3  $ and $p_3=-(2\alpha_2-\alpha_1)\ln z/3  $.
\end{remark}

\begin{theorem}
Let $\vec{n}=(n_1,n_2)$ be multi-indices that tend to infinity such that $n_1+n_2=2n$, $n_2/(n_1+n_2) =\theta \in (0,1/2)$ where $\theta$ remains constant with $\theta n \in \mathbb{Z}$. Then, for $\beta=0$, uniformly for $z$ in a small neighborhood $D(0,r_n)$ of $0$, we have, as $n \to \infty$,
\begin{align*}
&P_{\vec{n}}(z)=e^{n\tilde{H_0}(z)} \bigg\{e^{3\xi^{1/3}\omega^{\pm 2} - nH_0} \left( \frac{\xi}{z} \right)^{A_1/3}
\Big[ \tilde{\alpha}(\xi)L_0 + \tilde{\beta}(\xi)L_1 + \tilde{\gamma}(\xi)L_2 \Big](1+O(n^{-1})) \\
&\pm z^{-\alpha_1} e^{n(H_1(z) - H_0(z))}e^{3\xi^{\frac{1}{3}}\omega^{\mp 2}-nH_1} \left( \frac{\xi}{z} \right)^{A_2/3}
\Big[ \alpha'(\xi)L_0 + \beta'(\xi)L_1 + \gamma'(\xi)L_2 \Big](1+O(n^{-1})) \bigg\},
\end{align*}
where
\begin{equation}\label{eq: alpha,beta,gamma}
\left.
\begin{aligned}
\tilde{\alpha}(\xi) &= A(\xi)\omega^{2-A_1} - B(\xi)\omega^{-A_1} - C(\xi)\omega^{-2-A_1}, \\
\alpha'(\xi) &= A'(\xi)\omega^{2-A_1} - B'(\xi)\omega^{-A_1} - C'(\xi)\omega^{-2-A_1}, \\
\tilde{\beta}(\xi) &= -A(\xi)\omega^{A_1-2} + B(\xi)\omega^{A_1} + C(\xi)\omega^{2+A_1}, \\
\beta'(\xi) &= -A'(\xi)\omega^{A_1-2} + B'(\xi)\omega^{A_1} + C'(\xi)\omega^{2+A_1}, \\
\tilde{\gamma}(\xi) &= A(\xi) - B(\xi) - C(\xi), \\
\gamma'(\xi) &= A'(\xi) - B'(\xi) - C'(\xi).
\end{aligned}
\right\}
\end{equation}
\begin{align}\label{eq: A,B,C xi}
\left.
\begin{aligned}
A(\xi) &= [g_1 - w_1^{-1}g_2^\pm] \xi^{1/3},\\
B(\xi) &= z \partial_z g_1 - w_1^{-1} (z\partial_z - \alpha_1) g_2^\pm, \\
C(\xi) &= \big[(z\partial_z)^2 g_1 - w_1^{-1} (z\partial_z - \alpha_1)^2 g_2^\pm\big] \xi^{-1/3}.
\end{aligned}
\right\}
\end{align}
\begin{align}
A'(\xi) &= [g_2^\pm] \xi^{1/3},\quad
B'(\xi) = (z\partial_z - \alpha_1) g_2^\pm, \quad C'(\xi) = \big[(z\partial_z - \alpha_1)^2 g_2^\pm\big] \xi^{-1/3}.\label{eq: A,B,C xi prime} \\
L_0 = &e^{-nH_0+p_1} \sum_{i=1}^3 C_{1i} \mathcal{F}_i(\varphi_0),\quad
L_1 = e^{-nH_1+p_2} \sum_{i=1}^3 C_{1i} \mathcal{F}_i(\varphi_1),\quad
L_2 = e^{-nH_2+p_3} \sum_{i=1}^3 C_{1i} \mathcal{F}_i(\varphi_2) \label{eq:L0,L1,L2},
\end{align}
where
\begin{align}
C_{ik}=\big[ \mathfrak{C}_n^{(iii)}(z) \cdot \mathfrak{C}_n^{(ii)}(z) \cdot \mathfrak{C}_n^{(i)}(z) \big]_{ik}
= \sum_{j=1}^3 \sum_{\ell=1}^3
\big[ \mathfrak{C}_n^{(iii)}(z) \big]_{ij}
\big[ \mathfrak{C}_n^{(ii)}(z) \big]_{j\ell}
\big[ \mathfrak{C}_n^{(i)}(z) \big]_{\ell k} \label{eq: analytic prefactor matrix repre}
\end{align}
and the entries of 
\begin{align}
&\mathfrak{C}_n^{(iii)}(z)=\mathbb{I}-\dfrac{C_n^{(iii)}(z)-C_n^{(iii)}(0)}{n^9 z} , \quad \mathfrak{C}_n^{(ii)}(z)=\mathbb{I}-\frac{C_n^{(ii)}(z)-C_n^{(ii)}(0)}{n^6 z} , \notag \\ &\mathfrak{C}_n^{(i)}(z) = \mathbb{I}-\frac{C_n^{(i)}(z)-C_n^{(i)}(0)}{n^3 z}, \quad C_n^{(i)}(z)=E(z) C(z) [E(z)]^{-1} \label{eq: analytic prefactors}
\end{align} 
are explicitly determined using \eqref{eq: C(z)}, \eqref{eq: C n 2 z}, and \eqref{eq: C n 3 z}. Further, $E(z)$, $\xi(z)$ and $\mathcal{F}_i(\varphi_i)$ are given by \eqref{eq: E(z)}, \eqref{Conformal map}, and \eqref{Function F1 general case}- \eqref{Function F3 general case}, respectively.
\end{theorem}
\begin{proof}
We will use the parametrix $P$ constructed in Subsection \ref{subsec:Local parametrix near 0}. Assume $z \in D(0,r_n)$, then from \eqref{Transformation R}, \eqref{First guess of local parametrix} and \eqref{definition P(z)}, we get
\begin{multline}\label{eq: T in asymp zero 0}
T(z)=R(z)P(z)=R(z)E_n^0(z) \mathring{P}(z)=R(z)E_n^0(z)\left(\frac{z}{\xi}\right)^{\frac{M}{3}}\widehat{N(z)}^{-1}\mathbf{P}(z)
\\
=R(z)E_n^0(z) \Psi(\xi) \bigg(\frac{\xi}{z}\bigg)^{-A/3} \exp \big({ 3\xi^{\frac{1}{3}} \Omega_\pm } \big) D_n(z).
\end{multline}
As in the proof of Theorem \ref{Thrm:Asymptotics in unbounded region}, we have $P_{\vec{n}}(z)= e^{n\tilde{H_0}(z)}U_{11}(z)$. In view of relations \eqref{Transformtion T} between $T$ and $U$ in the region $\arg \xi \in (0,\frac{\pi}{4})$, we get
\begin{align}\label{eq: Pn in T11 and T12}
P_{\vec{n}}(z)= e^{n\tilde{H_0}(z)}(T_{11}(z) \pm z^{-\alpha_1}(1-z)^{-\beta} e^{n(H_1(z) - H_0(z))}T_{12}(z)),
\end{align}
where the $+$ sign holds in the upper part of the lens and the $-$ sign holds in the lower part. To calculate $T_{11}$ and $T_{12}$, we perform the following operations in \eqref{eq: T in asymp zero 0}
\begin{align}\label{eq: T11 and T12}
T_{11}(z) = (1,0,0)T(z)(1,0,0)^T, \quad T_{12}(z) = (1,0,0)T(z)(0,1,0)^T.
\end{align}
To begin with, let us compute matrix $F$ defined by
\begin{multline*}
F=\bigg(\frac{\xi}{z}\bigg)^{-A/3} \exp \big({ 3\xi^{\frac{1}{3}} \Omega_\pm } \big) D_n(z) 
\\
= \bigg(\frac{nf_0(z)}{3} \bigg)^{-A} \diag\bigg[e^{3\xi^{\frac{1}{3}}\omega^{\pm 2}},e^{3\xi^{\frac{1}{3}}\omega^{\mp 2}}, e^{3\xi^{\frac{1}{3}}} \bigg] \diag [e^{-nH_0}, e^{-nH_1}, e^{-nH_2}]
\\
=\diag \bigg[\bigg(\frac{nf_0(z)}{3} \bigg)^{\alpha_1+\alpha_2}e^{3\xi^{\frac{1}{3}}\omega^{\pm 2}-nH_0}, \bigg(\frac{nf_0(z)}{3} \bigg)^{-2\alpha_1+\alpha_2}e^{3\xi^{\frac{1}{3}}\omega^{\mp 2}-nH_1}, \bigg(\frac{nf_0(z)}{3} \bigg)^{\alpha_1-2\alpha_2}e^{3\xi^{\frac{1}{3}}-nH_2} \bigg].
\end{multline*}
Therefore,
\begin{align*}
F\begin{pmatrix}
      1 \\
        0 \\
        0
    \end{pmatrix}  = \begin{pmatrix}
      F_{11} \\
        0 \\
        0
    \end{pmatrix}.
\end{align*}
From \eqref{G 2 xi jump relations}, in the region $\arg \xi \in (0,\frac{\pi}{4})$,
\begin{align*}
\Psi(\xi) = \begin{pmatrix}
    g_1(\xi) & g^{\pm}_2(\xi) & g_3(\xi)  \\
    z \partial_z g_1(\xi) & \big(z \partial_z -\alpha_1 \big)g^{\pm}_2(\xi) &\big(z \partial_z -\alpha_2 \big)g_3(\xi)  \\
    \big(z \partial_z\big)^2 g_1(\xi) &\big(z \partial_z -\alpha_1 \big)^2 g^{\pm}_2(\xi) & \big(z \partial_z -\alpha_2 \big)^2 g_3(\xi)
\end{pmatrix} \begin{pmatrix}
      1& 0 &0\\
        -\frac{1}{w_1} &1 & 0\\
         0& 0& 1 \\  
    \end{pmatrix}.
\end{align*}
Hence,
\begin{align}\label{eq: G multiply with F}
&\widehat{\Psi}(\xi) \begin{pmatrix}
      1& 0 &0\\
        -\frac{1}{w_1} &1 & 0\\
         0& 0& 1 \\  
    \end{pmatrix}\begin{pmatrix}
      F_{11} \\
        0 \\
        0
    \end{pmatrix} = \widehat{\Psi}(\xi)\begin{pmatrix}
      F_{11} \\
      -w_1^{-1}  F_{11} \\
        0
    \end{pmatrix} \notag\\
&=\begin{pmatrix}
      [g_1(\xi)-w_1^{-1}g^{\pm}_2(\xi)]  F_{11} \\
      \bigg[z \partial_z g_1(\xi)-w_1^{-1}\big(z \partial_z -\alpha_1 \big)g^{\pm}_2(\xi) \bigg]  F_{11} \\
        \bigg[\big(z \partial_z\big)^2 g_1(\xi)-w_1^{-1}\big(z \partial_z -\alpha_1 \big)^2 g^{\pm}_2(\xi) \bigg]  F_{11}
    \end{pmatrix}= \begin{pmatrix}
      \mathbf{G}_{11} \\
        \mathbf{G}_{21} \\
        \mathbf{G}_{31}
    \end{pmatrix}.
\end{align}
Further, we need 
\begin{align*}
E_n^0(z)= \mathfrak{C}_n^{(iii)}(z)\mathfrak{C}_n^{(ii)}(z)\mathfrak{C}_n^{(i)}(z) E(z)
\end{align*}
from the definition \eqref{Def: En 0 and En inf} where $\mathfrak{C}_n^{(iii)}(z)$, $\mathfrak{C}_n^{(ii)}(z)$, and $\mathfrak{C}_n^{(i)}(z)$ are given by \eqref{eq: analytic prefactors}.
Finally, using \eqref{eq: G multiply with F} and Remark \ref{lm: Enz  matrix}, we have
\begin{align*}
&T(z)(1,0,0)^T= R(z)\mathfrak{C}_n^{(iii)}(z)\mathfrak{C}_n^{(ii)}(z)\mathfrak{C}_n^{(i)}(z) N(z)\begin{pmatrix}
      J^0_{11}& J^0_{12} & J^0_{13} \\
        J^0_{21} & J^0_{22} & J^0_{23}\\
        J^0_{31} & J^0_{32} & J^0_{33} \\  
    \end{pmatrix} \begin{pmatrix}
      \mathbf{G}_{11} \\
        \mathbf{G}_{21} \\
        \mathbf{G}_{31}
    \end{pmatrix}\\
& =R(z)\mathfrak{C}_n^{(iii)}(z)\mathfrak{C}_n^{(ii)}(z)\mathfrak{C}_n^{(i)}(z) \begin{pmatrix}
      N_{11}& N_{12} & N_{13} \\
        N_{21} & N_{22} & N_{23}\\
        N_{31} & N_{32} & N_{33} \\  
    \end{pmatrix} \begin{pmatrix}
      L_{11} \\
        L_{21} \\
        L_{31}
    \end{pmatrix} \\
& = R(z) \begin{pmatrix}
      C_{11}& C_{12} & C_{13} \\
        C_{21} & C_{22} & C_{23}\\
        C_{31} & C_{32} & C_{33} \\  
    \end{pmatrix} \begin{pmatrix}
      M_{11} \\
        M_{21} \\
        M_{31}
    \end{pmatrix}= R(z)\begin{pmatrix}
      C_{11}M_{11}+C_{12}M_{21}+C_{13}M_{31} \\
       C_{21}M_{11}+C_{22}M_{21}+C_{23}M_{31} \\
       C_{31}M_{11}+C_{32}M_{21}+C_{33}M_{31} 
    \end{pmatrix},
\end{align*}
where $L_{j1}=\sum_{i=1}^3 J^0_{ji}\mathbf{G}_{i1}$, $M_{j1}=\sum_{i=1}^3 N_{ji}L_{i1}$, and $C_{ik}$ are the entries given by \eqref{eq: analytic prefactor matrix repre}.
Consequently,
\begin{align}\label{eq: C.M T11 On-1}
&T_{11}(z) = [C_{11}M_{11}+C_{12}M_{21}+C_{13}M_{31}](1+O(n^{-1})).
\end{align}
With the definition of $A(\xi)$, $B(\xi)$, and $C(\xi)$ as in \eqref{eq: A,B,C xi}, we obtain
\begin{align*}
&L_{11} = [g_1(\xi)-w_1^{-1}g^{\pm}_2(\xi)] \bigg(\frac{nf_0(z)}{3} \bigg)^{\alpha_1+\alpha_2}e^{3\xi^{\frac{1}{3}}\omega^{\pm 2}-nH_0}\xi^{\frac{1}{3}}\omega^{2-A_1}e^{-nH_0+p_1}\\
&-\bigg[z \partial_z g_1(\xi)-w_1^{-1}\bigg(z \partial_z -\alpha_1 \bigg)g^{\pm}_2(\xi) \bigg]\bigg(\frac{nf_0(z)}{3} \bigg)^{\alpha_1+\alpha_2}e^{3\xi^{\frac{1}{3}}\omega^{\pm 2}-nH_0}\omega^{-A_1}e^{-nH_0+p_1}\\
&+\bigg[\bigg(z \partial_z\bigg)^2 g_1(\xi)-w_1^{-1}\bigg(z \partial_z -\alpha_1 \bigg)^2 g^{\pm}_2(\xi) \bigg]\bigg(\frac{nf_0(z)}{3} \bigg)^{\alpha_1+\alpha_2}e^{3\xi^{\frac{1}{3}}\omega^{\pm 2}-nH_0}\xi^{-\frac{1}{3}}\omega^{-2-A_1}e^{-nH_0+p_1}\\
&=e^{3\xi^{1/3}\omega^{\pm 2} - nH_0}e^{-nH_0+p_1} \left( \frac{\xi}{z} \right)^{A_1/3} \bigg[A(\xi)\xi^{\frac{1}{3}}\omega^{2-A_1}-B(\xi)\omega^{-A_1}+C(\xi)\xi^{-\frac{1}{3}}\omega^{-2-A_1} \bigg].
\end{align*}
Similarly,
\begin{align*}
&L_{21} = e^{3\xi^{\frac{1}{3}}\omega^{\pm 2}-nH_0-nH_1+p_2}  \left( \frac{\xi}{z} \right)^{A_1/3}  \bigg[-A(\xi)\xi^{\frac{1}{3}}\omega^{A_1-2}+B(\xi)\omega^{A_1}-C(\xi)\xi^{-\frac{1}{3}}\omega^{2+A_1} \bigg],\\
&L_{31} = e^{3\xi^{\frac{1}{3}}\omega^{\pm 2}-nH_0-nH_2+p_3}  \left( \frac{\xi}{z} \right)^{A_1/3}  \bigg[A(\xi)\xi^{\frac{1}{3}}-B(\xi)+C(\xi)\xi^{-\frac{1}{3}} \bigg].
\end{align*}

We have \( M_{i1} \) for \( i = 1,2,3 \), all of the form:
\[
M_{i1} = e^{3\xi^{1/3}\omega^{\pm 2} - nH_0} \left( \frac{\xi}{z} \right)^{A_1/3}
\Big\{
A(\xi) \mathcal{X}_i + B(\xi) \mathcal{Y}_i + C(\xi) \mathcal{Z}_i
\Big\}.
\]

The vectors \(\mathcal{X}_i, \mathcal{Y}_i, \mathcal{Z}_i\) are defined as
\begin{align*}
\mathcal{X}_i &= \omega^{2-A_1} e^{-nH_0+p_1} \mathcal{F}_i(\varphi_0(z))
- \omega^{A_1-2} e^{-nH_1+p_2}  \mathcal{F}_i(\varphi_1(z))
+  e^{-nH_2+p_3} \mathcal{F}_i(\varphi_2(z)),\\
\mathcal{Y}_i &= -\omega^{-A_1} e^{-nH_0+p_1} \mathcal{F}_i(\varphi_0(z))
+ \omega^{A_1} e^{-nH_1+p_2}  \mathcal{F}_i(\varphi_1(z))
-  e^{-nH_2+p_3} \mathcal{F}_i(\varphi_2(z)),\\
\mathcal{Z}_i &= -\omega^{-2-A_1} e^{-nH_0+p_1} \mathcal{F}_i(\varphi_0(z))
+ \omega^{2+A_1} e^{-nH_1+p_2}  \mathcal{F}_i(\varphi_1(z))
-  e^{-nH_2+p_3} \mathcal{F}_i(\varphi_2(z)).
\end{align*}

We aim to compute
\begin{multline*}
C_{11}M_{11}+C_{12}M_{21}+C_{13}M_{31}
\\= e^{3\xi^{1/3}\omega^{\pm 2} - nH_0} \left( \frac{\xi}{z} \right)^{A_1/3}\Big[ A(\xi) \sum_{i=1}^3 C_{1i} \mathcal{X}_i
+ B(\xi) \sum_{i=1}^3 C_{1i} \mathcal{Y}_i
+ C(\xi) \sum_{i=1}^3 C_{1i} \mathcal{Z}_i \Big].
\end{multline*}

Define the sums
\[
S_0 = \sum_{i=1}^3 C_{1i} \mathcal{F}_i(\varphi_0),\quad
S_1 = \sum_{i=1}^3 C_{1i} \mathcal{F}_i(\varphi_1),\quad
S_2 = \sum_{i=1}^3 C_{1i} \mathcal{F}_i(\varphi_2).
\]

Then the linear combinations become
\begin{align*}
C_{11}\mathcal{X}_1+C_{12}\mathcal{X}_2+C_{13}\mathcal{X}_3 &= \omega^{2-A_1} e^{-nH_0+p_1} S_0
- \omega^{A_1-2} e^{-nH_1+p_2}  S_1
+  e^{-nH_2+p_3} S_2,\\
C_{11}\mathcal{Y}_1+C_{12}\mathcal{Y}_2+C_{13}\mathcal{Y}_3 &= -\omega^{-A_1} e^{-nH_0+p_1} S_0
+ \omega^{A_1} e^{-nH_1+p_2}  S_1
-  e^{-nH_2+p_3} S_2,\\
C_{11}\mathcal{Z}_1+C_{12}\mathcal{Z}_2+C_{13}\mathcal{Z}_3 &= -\omega^{-2-A_1} e^{-nH_0+p_1} S_0
+ \omega^{2+A_1} e^{-nH_1+p_2}  S_1
-  e^{-nH_2+p_3} S_2.
\end{align*}

Using \eqref{eq:L0,L1,L2}, we can introduce the compact notation
\[
L_0 = e^{-nH_0+p_1} S_0,\qquad
L_1 = e^{-nH_1+p_2}  S_1,\qquad
L_2 = e^{-nH_2+p_3}  S_2.
\]

Then
\begin{align*}
C_{11}\mathcal{X}_1+C_{12}\mathcal{X}_2+C_{13}\mathcal{X}_3 &= \omega^{2-A_1} L_0 - \omega^{A_1-2} L_1 + L_2,\\
C_{11}\mathcal{Y}_1+C_{12}\mathcal{Y}_2+C_{13}\mathcal{Y}_3  &= -\omega^{-A_1} L_0 + \omega^{A_1} L_1 - L_2,\\
C_{11}\mathcal{Z}_1+C_{12}\mathcal{Z}_2+C_{13}\mathcal{Z}_3 &= -\omega^{-2-A_1} L_0 + \omega^{2+A_1} L_1 - L_2.
\end{align*}



Thus
\begin{align}\label{eq: C.M T11}
C_{11}M_{11}+C_{12}M_{21}+C_{13}M_{31}
= e^{3\xi^{1/3}\omega^{\pm 2} - nH_0} \left( \frac{\xi}{z} \right)^{A_1/3}
\Big[ \tilde{\alpha}(\xi)L_0 + \tilde{\beta}(\xi)L_1 + \tilde{\gamma}(\xi)L_2 \Big],
\end{align}
where $\tilde{\alpha}(\xi)$, $\tilde{\beta}(\xi)$, and $\tilde{\gamma}(\xi)$ are given by \eqref{eq: alpha,beta,gamma}.
Now, to extract $T_{12}$ as in \eqref{eq: T11 and T12}, we calculate,
\begin{align*}
F\begin{pmatrix}
      0 \\
        1 \\
        0
    \end{pmatrix}  = \begin{pmatrix}
      0 \\
        F_{22} \\
        0
    \end{pmatrix},
\end{align*}
which implies
\begin{align*}
&\widehat{\Psi}(\xi) \begin{pmatrix}
      1& 0 &0\\
        -\frac{1}{w_1} &1 & 0\\
         0& 0& 1 \\  
    \end{pmatrix}\begin{pmatrix}
      0 \\
        F_{22} \\
        0
    \end{pmatrix} =\begin{pmatrix}
      [g^{\pm}_2(\xi)]  F_{22} \\
      \big[\big(z \partial_z -\alpha_1 \big)g^{\pm}_2(\xi) \big]  F_{22} \\
        \big[\big(z \partial_z -\alpha_1 \big)^2 g^{\pm}_2(\xi) \big]  F_{22}
    \end{pmatrix}= \begin{pmatrix}
      \mathbf{G}_{12} \\
        \mathbf{G}_{22} \\
        \mathbf{G}_{32}
    \end{pmatrix}.
\end{align*}
Following a similar procedure as detailed above, we get
\begin{align}\label{eq: C.M T12 On-1}
T_{12}(z) = (1,0,0)T(z) = [C_{11}M_{12}+C_{12}M_{22}+C_{13}M_{32}](1+O(n^{-1})),
\end{align}
where
\[
M_{i2} = e^{3\xi^{\frac{1}{3}}\omega^{\mp 2}-nH_1} \left( \frac{\xi}{z} \right)^{A_2/3}
\Big\{
A'(\xi) \mathcal{X}_i + B'(\xi) \mathcal{Y}_i + C'(\xi) \mathcal{Z}_i
\Big\},
\]
and $A'(\xi)$, $B'(\xi)$, and $C'(\xi)$ are given in \eqref{eq: A,B,C xi prime}. Thus,
\begin{align}\label{eq: C.M T12}
C_{11}M_{12}+C_{12}M_{22}+C_{13}M_{32}
= e^{3\xi^{\frac{1}{3}}\omega^{\mp 2}-nH_1} \left( \frac{\xi}{z} \right)^{A_2/3}
\Big[ \alpha'(\xi)L_0 + \beta'(\xi)L_1 + \gamma'(\xi)L_2 \Big],
\end{align}
where $\alpha'(\xi)$, $\beta'(\xi)$, and $\gamma'(\xi)$ are given by \eqref{eq: alpha,beta,gamma}. Substituting \eqref{eq: C.M T12}, \eqref{eq: C.M T11}, \eqref{eq: C.M T11 On-1}, and \eqref{eq: C.M T12 On-1} into \eqref{eq: Pn in T11 and T12}, we get the final asymptotics.
\end{proof}

\section*{Acknowledgement} 
The research of the first author was carried out within the state assignment for IAM FEB RAS (N 075-00460-26-00).

\bibliographystyle{abbrv} 
\bibliography{references.bib}

\end{document}